\documentclass[10pt,a4paper]{amsart}

\usepackage{hyperref}
\usepackage{latexsym,amsmath, bm}
\usepackage{enumerate}
\usepackage{amsfonts}
\usepackage{amssymb}
\usepackage{geometry}
\usepackage{latexsym}
\usepackage{fixmath}
\usepackage{faktor}
\usepackage{mathtools}
\usepackage[T1]{fontenc}
\usepackage{fourier}
\usepackage{bbm}
\usepackage{color}
\usepackage{pgfplots}
\usepackage{tikz}
\usetikzlibrary{shapes,decorations,arrows,calc,arrows.meta,fit,positioning}
\usepackage{graphicx}
\usepackage{fullpage}
\usetikzlibrary{decorations.pathreplacing, matrix}
\usepackage[boxed, linesnumbered, noend, noline]{algorithm2e}

\numberwithin{equation}{section}

\newcommand\cA{\mathcal{A}}

\newcommand\cC{\mathcal{C}}

\newcommand\cE{\mathcal{E}}
\newcommand\cF{\mathcal{F}}

\newcommand\cH{\mathcal{H}}

\newcommand\cM{\mathcal{M}}

\newcommand\cO{\mathcal{O}}
\newcommand\cP{\mathcal{P}}

\newcommand\SIGMA{\vec\sigma}

\newcommand\TAU{\vec\tau}

\newcommand{\Po}{{\rm Po}}

\renewcommand{\vec}[1]{\boldsymbol{#1}}
\newcommand{\vecone}{\vec{1}}

\newcommand\KL[2]{D_{\mathrm{KL}}\bc{{{#1}\|{#2}}}}

\newcommand\eps{\varepsilon}

\newcommand\pr{\mathbb{P}} 
\renewcommand\Pr{\pr}

\newcommand\Erw{\mathbb{E}}

\newcommand{\whp}{w.h.p.}

\newcommand{\tensor}{\otimes}

\newtheorem{definition}{Definition}[section]

\newtheorem{theorem}[definition]{Theorem}
\newtheorem{lemma}[definition]{Lemma}
\newtheorem{proposition}[definition]{Proposition}
\newtheorem{corollary}[definition]{Corollary}

\newtheorem{fact}[definition]{Fact}

\newcommand\Lem{Lemma}
\newcommand\Prop{Proposition}
\newcommand\Thm{Theorem}

\newcommand\Cor{Corollary}
\newcommand\Sec{Section}

\newcommand\bc[1]{\left({#1}\right)}
\newcommand\cbc[1]{\left\{{#1}\right\}}

\newcommand\brk[1]{\left\lbrack{#1}\right\rbrack}

\newcommand\norm[1]{\left\|{#1}\right\|}
\newcommand\abs[1]{\left|{#1}\right|}

\newcommand{\Erdos}{Erd\H{o}s}
\newcommand{\Renyi}{R\'enyi}

\newcommand{\bks}{\beta_{\text{KS}}}
\renewcommand\Pr{\pr}
\newcommand{\G}{\mathbb{G}}
\newcommand{\Gnd}{\mathbb{G}(n,d)}
\newcommand{\Zgb}{Z_{\G, \beta}}

\renewcommand{\vec}[1]{\boldsymbol{#1}}

\newcommand{\evO}{\vecone \cbc{\cO}}
\newcommand\evC[1]{\vecone \cbc{\cC{#1}}}

\newcommand\Fac{Fact}

\newcommand{\Gpone}{\mathbf{G}^*_1}
\newcommand{\Gptwo}{\mathbf{G}^*_2}
\newcommand{\Gbold}{\mathbf{G}}

\begin{document}
	
	\title{The Ising antiferromagnet in the replica symmetric phase}
	
	\thanks{The authors thank Amin Coja-Oghlan for helpful discussions and insights. The authors also thank Mark Sellke for helpful comments. Philipp Loick is supported by DFG CO 646/3.}

	\author{Christian Fabian, Philipp Loick}

    \address{Christian Fabian, {\tt cfabian@stud.uni-frankfurt.de}, Goethe University, Mathematics Institute, 10 Robert Mayer St, Frankfurt 60325, Germany.}
    \address{Philipp Loick, {\tt loick@math.uni-frankfurt.de}, Goethe University, Mathematics Institute, 10 Robert Mayer St, Frankfurt 60325, Germany.}
	
	\begin{abstract}
	Partition functions are an important research object in combinatorics and mathematical physics [Barvinok, 2016].
    In this work, we consider the partition function of the Ising antiferromagnet on random regular graphs and characterize its limiting distribution in the replica symmetric phase up to the Kesten-Stigum bound. Our proof relies on a careful execution of the method of moments, spatial mixing arguments and small subgraph conditioning. 
	\end{abstract}
	
	\maketitle

\section{Introduction}

\subsection{Motivation}

The Ising model, invented by Lenz in 1920 to explain magnetism, is a cornerstone in statistical physics. 
Consider any graph $G$ with vertex set $V$ and edge set $E$. Each vertex carries one of two possible spins $\pm 1$ and the interactions between vertices are represented by $E$. For a spin configuration $\sigma \in \cbc{\pm 1}^V$ on $G$, we can consider the Hamiltonian $\cH_G$
\begin{equation*}
    \cH_G(\sigma) = \sum_{(v,w) \in E} \frac{1 + \sigma_v \sigma_w}{2}.
\end{equation*}
Together with a real parameter $\beta>0$ the Hamiltonian gives rise to a distribution on spin configurations defined by
\begin{equation} \label{eq_boltzmann}
    \mu_{G,\beta}(\sigma) = \frac{\exp\bc{-\beta \cH_G(\sigma)}}{Z_{G,\beta}} \quad \bc{\sigma \in \cbc{\pm 1}^V} \quad \text{where} \quad Z_{G,\beta} = \sum_{\tau \in \cbc{\pm 1}^V} \exp \bc{-\beta \cH_G(\tau)}.
\end{equation}
The probability measure $\mu_{G,\beta}$ is known as the Boltzmann distribution with the normalizing term $Z_{G, \beta}$ being the partition function. $\mu_{G,\beta}$ favors configurations with few edges between vertices of the same spin which is known as the antiferromagnetic Ising model. There is a corresponding formulation of \eqref{eq_boltzmann} where edges between vertices of the same spin are preferred - the ferromagnetic Ising model. Both models are of great interest in combinatorics and physics and the literature on each is vast \cite{Huang_2009}.

In this paper, we study the Ising antiferromagnet on the random $d$-regular graph $\G=\G(n,d)$. One might be tempted to think that the regularities of this graph model provide a more amenable study object than its well-known \Erdos-\Renyi\ counterpart with fluctuating vertex degrees. However, for the Ising model the reverse seems to be true. 
Indeed, the independence of edges in the \Erdos-\Renyi-model greatly facilitates deriving the distribution of short cycles in the planted model and simplifies the calculation of both the first and second moment. 

Clearly, $\mu_{\G, \beta}$ gives rise to correlations between spins of nearby vertices. The degree of such correlations is governed by the choice of $\beta$. A question which is of keen interest in combinatorics and statistical physics is whether such correlations persist for two uniformly sampled (and thus likely distant) vertices. According to physics predictions, for small values of $\beta$ we should observe a rapid decay of correlation \cite{Mezard_2009} and thus no \textit{long-range correlations}. This regime is known as the \textit{replica symmetric phase}. It is suggested that there exists a specific $\beta$ which marks the onset of long-range correlations in $\G$. This value is conjectured to be at the combinatorially meaningful Kesten-Stigum bound \cite{Coja_2020}
\begin{equation*}
    \bks = \log \bc{\frac{\sqrt{d-1}+1}{\sqrt{d-1}-1}}.
\end{equation*}
The question of long-range correlations is tightly related to the partition function $Z_{\G, \beta}$ from which also various combinatorially meaningful observables can be derived. The {\sc Max Cut} on random $d$-regular graphs is a case in point due to the well-known relation
\begin{equation*}
\mbox{\sc MaxCut}(G)=\frac{dn}{2}+\lim_{\beta\to\infty}\frac{\partial}{\partial\beta}\log Z_{G,\beta}.
\end{equation*}
for any graph $G$. Thus, it is of key interest to understand the behavior of $Z_{\G, \beta}$.

\subsection{Result}

In recent work, \cite{Coja_2020} were able to pinpoint the replica symmetry breaking phase transition at the Kesten-Stigum bound, thus charting the replica symmetric phase for the Ising antiferromagnet on random $d$-regular graphs.
The key feature of the replica-symmetric phase is that \whp\ two independent samples $\SIGMA_1, \SIGMA_2$ from the Boltzmann distribution $\mu_{\G, \beta}$ exhibit an almost flat overlap in the sense that $\abs{\SIGMA_1 \cdot \SIGMA_2} = o(n) $.
To be precise, \cite{Coja_2020} determined $\Zgb$ up to an error term $\exp(o(n))$ for $\beta < \bks$.
In this paper, we move beyond this crude approximation. By deriving the limiting distribution in the replica-symmetric phase, we show that $Z_{\G, \beta}$ is tightly concentrated with bounded fluctuations which we can quantify and attribute to short cycles in $\G$. 

\begin{theorem} \label{thm_dist}
Assume that $0 < \beta < \bks$ and $d \geq 3$. Let $(\Lambda_i)_{i}$ be a sequence of independent Poisson variables with $\Erw \brk{\Lambda_i} = \lambda_i$ where  $\lambda_i = \frac{\bc{d - 1}^i }{2 i}$.
Then as $n \to \infty$ we have
\begin{align*}
	&\log \bc{Z_{\Gnd, \beta}} - \frac{1}{2} \log \bc{\frac{1 + e^{ \beta}}{ 2 + d  e^{\beta} - d }} - n \bc{\bc{1 - \frac{d}{2}} \log \bc{2}  + \frac{d}{2} \log\bc{1 + e^{- \beta}}} + \frac{d - 1}{2} \frac{e^{-\beta}-1}{e^{-\beta}+1} + \frac{\bc{d - 1}^2}{4} \bc{\frac{e^{-\beta}-1}{e^{-\beta}+1}}^2\\
	&\qquad \qquad \qquad \qquad \xrightarrow{d} \log \bc{W} := \sum_{i = 3}^\infty \Lambda_i \log \bc{1+\bc{\frac{e^{-\beta}-1}{e^{-\beta}+1}}^i} -  \frac{\bc{d - 1}^i}{2 i} \bc{\frac{e^{-\beta}-1}{e^{-\beta}+1}}^i.
\end{align*}
The infinite product defining $W$ converges a.s. and in $L^2$.
\end{theorem}

Taking the expectation of this distribution readily recovers the first part of the result by \cite{Coja_2020}. 
The proof of \Thm~\ref{thm_dist} relies on the combination of the method of moments and small subgraph conditioning enriched in our case by spatial mixing arguments to make the calculation of the second moment tractable.

\section{Techniques}

\subsection{Notation}

Let $\G=\Gnd$ denote a random $d$-regular graph on $n$ vertices. We consider sparse graphs with constant $d$ as $n \to \infty$. Throughout the paper, we will employ standard Landau notation with the usual symbols $o(\cdot), O(\cdot), \Theta(\cdot), \omega(\cdot)$, and $\Omega(\cdot)$ to refer to the limit $n \to \infty$. We say that a sequence of events $(\cE_n)_n$ holds \textit{with high probability} (\whp) if $\lim_{n\to\infty} \Pr \brk{\cE_n} = 1$. When the context is clear we might drop the index of the expectation. Moreover, we will use the proportional $\propto$ to hide necessary normalisations.

\subsection{Outline}

To get a handle on the distribution of $Z_{\G, \beta}$ in the replica symmetric phase, we need to identify the sources of fluctuations of $Z_{\G, \beta}$. One obvious source is the number of short cycles. Since $\G$ is sparse and random, standard arguments reveal that $\G$ contains only few short cycles. In the following, let $C_i(G)$ denote the number of short cycles of length $i$ in a graph $G$ and $\cF_\ell$ the $\sigma$-algebra generated by the random variables $C_i(\G)$ for $i \leq \ell$. A key quantity to consider is the variance of $Z_{\G, \beta}$. By standard decomposition, we have
\begin{equation*}
    \Erw \brk{Z_{\G, \beta}^2} - \Erw \brk{Z_{\G, \beta}}^2 = \Erw \brk{\Erw \brk{Z_{\G, \beta} \mid \cF_\ell}^2 - \Erw \brk{Z_{\G, \beta}}^2} + \Erw \brk{\Erw \brk{Z_{\G, \beta}^2 \mid \cF_\ell} - \Erw \brk{Z_{\G, \beta} \mid \cF_\ell}^2}
\end{equation*}
for any $\ell \geq 1$.
Note that the first term of the r.h.s. describes the contribution to the variance by the fluctuations in the number of short cycles, while the second term accounts for the conditional variance given the number of short cycles. It turns out that as $\ell \to \infty$ after taking $n \to \infty$, the second summand vanishes. In other words, the entire variance of $Z_{\G, \beta}$ is due to fluctuations in the number of short cycles.

To show this property formally, we leverage a result by \cite{Janson_1995} that stipulates conditions under which one is able to describe the limiting distribution of $Z_{\G, \beta}$ (see \Thm~\ref{thm_janson} in the appendix). 
One ingredient is the distribution of short cycles in $\G$ and a planted model $\G^*$. In $\G^*$, we first select a spin configuration $\sigma$ uniformly at random and subsequently sample a graph $G$ with probability proportional to $\exp \bc{- \beta \cH_G(\sigma)}$. While the distribution of short cycles in $\G$ is well established, the distribution of short cycles in the planted model $\G^*$ is a key contribution of this paper. The second ingredient is a careful application of the method of moments. Unfortunately, standard results on the first and second moment on random regular graphs (see i.e. \cite{Coja_2020}), do not suffice in our case and we have to sharpen our pencils to yield an error term of order $O\bc{\exp(1/n)}$. While the need for this lower error term prolongs calculations, it also poses some challenges that we resolve by a careful application of the Laplace's method as suggested by \cite{Greenhill_2010} and spatial mixing arguments.

\subsection{Short cycles}

To get started, let us write
\begin{equation} \label{eq_def_delta_lambda}
    \delta_i = \bc{\frac{e^{-\beta}-1}{e^{-\beta}+1}}^i \qquad \text{and} \qquad \lambda_i = \frac{\bc{d-1}^i}{2i}.
\end{equation}
The first item on the agenda is to derive the distribution of short cycles in $\G$. This is a well-established result.

\begin{fact} [Theorem 9.5 in \cite{Janson_2011}] \label{fact_cycle_null}
Let $\Lambda_i \sim \Po(\lambda_i)$ be a sequence of independent Poisson random variables for $i \geq 3$. Then jointly for all $i$ we have $C_{i}(\G) \xrightarrow{d} \Lambda_i$ as $n \to \infty$.
\end{fact}

Deriving the distribution of short cycles in the planted model $\G^*$ informally introduced above requires some more work. Let us start with the definitions. Given $\sigma \in \cbc{\pm 1}^V$ and for any $\beta > 0$, let us define the distribution of $\G^*(\sigma)$ for any event $\cA$ as
\begin{equation} \label{eq_g_star}
    \Pr \brk{\G^*(\sigma) \in \cA} \propto \Erw \brk{\exp\bc{-\beta \cH_\G(\sigma)} \vecone \cbc{\G \in \cA}}.
\end{equation}
This definition gives rise to the following experiment. First, draw a spin configuration $\SIGMA^*$ uniformly at random among all configurations $\cbc{\pm 1}^V$. In the next step, draw $\G^*=\G^*(\SIGMA^*)$ according to \eqref{eq_g_star}. Hereafter, $\G^*$ will be denoted the planted model.

\begin{proposition} \label{prop_cycle_planted}
Let 
\begin{equation*}
    \Xi_i \sim \Po \bc{\lambda_i \bc{1 + \delta_i}}
\end{equation*}
be a sequence of independent Poisson random variables for $i \geq 3$. Then jointly for all $i$ we have $C_{i}(\G^*) \xrightarrow{d} \Xi_i$ as $n \to \infty$.
\end{proposition}

Establishing the distribution of short cycles in $\G^*$ is one of the main contributions of this paper. To this end, we start off with similar arguments as used in \cite{Mossel_2011}, but need to diligently account for the subtle dependencies introduced by the regularities in $\G^*$.

Applying \Fac~\ref{fact_cycle_null} and \Prop~\ref{prop_cycle_planted} to \Thm~1 in \cite{Janson_1995} requires a slight detour via the Nishimori property. To this end, note that the random graph $\G$ induces a reweighted graph distribution $\hat \G$ which for any event $\cA$ is defined by
\begin{equation} \label{eq_g_hat}
    \Pr \brk{\hat \G \in \cA} \propto \Erw \brk{Z_{\G,\beta} \vecone \cbc{\G \in \cA}}.
\end{equation}
Moreover, consider the distribution $\hat \SIGMA$ on spin configurations defined by
\begin{equation} \label{eq_sigma_hat}
    \Pr \brk{\hat \SIGMA = \sigma} \propto \Erw \brk{\exp\bc{-\beta \cH_{\G}(\sigma)}}
\end{equation}
for any $\beta > 0$.
$\hat \G, \G^*, \hat \SIGMA, \SIGMA^*$, and the Boltzmann distribution from \eqref{eq_boltzmann} are connected via the well-known Nishimori property.
\begin{fact} [Proposition 3.2 in \cite{Coja_2021}]\label{fact_nishimori}
For any graph $G$ and spin configuration $\sigma \in \cbc{\pm 1}^V$ we have
\begin{equation*}
    \Pr \brk{\hat \G = G} \mu_{G}(\sigma) = \Pr \bc{\hat \SIGMA = \sigma} \Pr \bc{\G^* = G \mid \SIGMA^* = \sigma}.
\end{equation*}
\end{fact}

\subsection{The first and second moment}

The second key ingredient towards the proof of \Thm~\ref{thm_dist} is the method of moments. As standard random regular graph results are too crude, we need a more precise calculation. Fortunately, with some patience and equipped with Laplace's method as stated in \cite{Greenhill_2010}, the first moment is not too hard to find.

\begin{proposition}\label{prop_first_moment}
Assume that $0 < \beta < \bks$ and $d \geq 3$. Then we have
\begin{equation*}
	\Erw \brk{\Zgb } =  \exp \bc{- \lambda_1 \delta_1 - \lambda_2 \delta_2 + O \bc{\frac{1}{n}}}  \sqrt{\frac{1 + e^{ \beta}}{ 2 + d  e^{\beta} - d }}  \exp \bc{n \bc{\bc{1 - d/2} \log \bc{2}  + d\log\bc{1 + e^{- \beta}}/2}}
\end{equation*}
\end{proposition}

The second moment is not as amenable. The key challenge for applying Laplace's method is to exhibit that the obvious choice of the optimum is indeed a global maximum. We resolve this issue by resorting to results on the broadcasting process on an infinite $d$-regular tree and the disassortative stochastic block model. This spatial mixing argument allows us to focus our attention on an area close to the anticipated optimum. To this end, let us exhibit an event $\cO$ that is concerned with the location of two typical samples $\SIGMA_\G, \SIGMA'_\G$ from the Boltzmann distribution $\mu_{\G, \beta}$, i.e.
\begin{equation} \label{eq_def_O}
    \cO = \cbc{\Erw \brk{\vert \SIGMA_\G \cdot \SIGMA'_\G  \vert \mid \G} < \eps_n n}
\end{equation}
for a sequence of $\eps_n = o(1)$. Then we can leverage the following result from \cite{Coja_2020}.

\begin{lemma}[Lemma 4.7 in \cite{Coja_2020}] \label{lem_o}
For the event $\cO$ defined in \eqref{eq_def_O} we have for $d \geq 3, 0<\beta < \bks$
\begin{equation*}
    \Erw \brk{Z_{\G, \beta} \evO} = (1 - o(1))\Erw \brk{Z_{\G, \beta}}.
\end{equation*}
\end{lemma}

Conditioning on $\cO$ greatly facilitates the calculation of the second moment.
\begin{proposition}\label{prop_second_moment}
For $0 < \beta < \bks$ and $d \geq 3$ we have
\begin{equation*}
	\Erw \brk{\Zgb^2 \evO} = \exp \bc{\lambda_1 + \lambda_2 -  \frac{4 \lambda_1}{\bc{1 + e^\beta}^2} - \frac{4 \lambda_2  \bc{1 + e^{2 \beta}}^2}{\bc{1 + e^\beta}^4}+O \bc{\frac{1}{n}}} \frac{ \bc{ 1 + e^{ \beta}}^2  \exp \bc{n \bc{\bc{ 2 - d}  \log \bc{2} + d \log \bc{1 + e^{- \beta}}}}}{\bc{d e^\beta - d + 2} \sqrt{2 e^{2 \beta} + 2 d e^\beta  - d e^{2 \beta} - d + 2}}
\end{equation*}
\end{proposition}

\subsection{Proof of \Thm~\ref{thm_dist}}

We apply \Thm~1 in \cite{Janson_1995} to the random variable $Z_{\G, \beta} \vecone \cbc{\cO}$.
Condition $(1)$  readily follows from \Fac~\ref{fact_cycle_null}. For Condition $(2)$ let us write
\begin{equation*}
    \cC(G) = \cbc{C_{1}(G)=c_1, \dots, C_{\ell}(G)=c_\ell}
\end{equation*}
for any graph $G$. By \Lem~\ref{lem_o} considering $\Zgb$ rather than $\Zgb \vecone \cbc{\cO}$ only introduces an error of order $1+o(1)$ in Condition $(2)$. Using standard reformulations and the definition of $\hat \G$ from \eqref{eq_g_hat} we find
\begin{equation*}
    \frac{\Erw \brk{\Zgb \mid \cC(\G)}}{\Erw \brk{\Zgb}} = \frac{\Erw \brk{\Zgb \evC{(\G)}}}{\Pr \brk{\cC(\G)} \Erw \brk{\Zgb}} = \frac{\Pr \brk{\cC(\hat \G)}}{\Pr \brk{\cC(\G)}} = \frac{\Erw_{\hat \SIGMA} \brk{\Pr \brk{\cC(\hat \G)\mid \hat \SIGMA}}}{\Pr \brk{\cC(\G)}}.
\end{equation*}
Since a typical sample $\sigma$ from $\hat \SIGMA$ has the property that $\abs{\sigma \cdot \vecone} = O\bc{n^{2/3}}$, i.e. is relatively balanced, the Nishimori property (\Fac~\ref{fact_nishimori}) implies
\begin{equation*}
    \Erw_{\hat \SIGMA} \brk{\Pr \brk{\cC(\hat \G) \mid \hat \SIGMA}} \sim \Pr \brk{\cC{(\G^*)}}.
\end{equation*}
Condition $(2)$ now follows from \Fac~\ref{fact_cycle_null} and \Prop~\ref{prop_cycle_planted}.
For Condition $(3)$ consider any $\beta = \bks - \eps$ for some small $\eps>0$. Letting $\eta=\eta(\eps)>0$ a simple calculation reveals
\begin{equation*}
     \sum_{i \geq 1} \lambda_i \delta_i^2 \leq \sum_{i \geq 1} \lambda_i  \bc{\frac{e^{-\bks+\eps}-1}{e^{-\bks+\eps}+1}}^{2i} = \sum_{i \geq 1} \frac{(1-\eta)^i}{2i} < \infty
\end{equation*}
which also implies $\sum_{i \geq 3} \lambda_i \delta_i^2 < \infty$. Finally, by \Lem~\ref{lem_o}, \Prop s~\ref{prop_first_moment} and \ref{prop_second_moment} and the fact that for any $0 < x < 1$
$\log \bc{1-x} = -\sum_{i \geq 1} x^i/i$
we find for $0 < \beta < \bks$ and $d \geq 3$
\begin{align*}
    &\frac{\Erw \brk{\Zgb^2 \evO}}{\Erw \brk{\Zgb \evO}^2} = (1+o(1)) \frac{\Erw \brk{\Zgb^2 \evO}}{\Erw \brk{\Zgb}^2}\\
    &=  (1+o(1)) \frac{1 + e^\beta}{\sqrt{2  + 2 e^{2 \beta} + 2 d e^\beta -  d  - d e^{2 \beta}}} \exp \bc{\lambda_1 + \lambda_2 -  \frac{4 \lambda_1}{\bc{1 + e^\beta}^2} - \frac{4 \lambda_2  \bc{1 + e^{2 \beta}}^2}{\bc{1 + e^\beta}^4} + 2 \lambda_1 \delta_1 + 2 \lambda_2 \delta_2} \\
    &= (1+o(1)) \bc{1-(d-1)\bc{\frac{e^{-\beta}-1}{e^{-\beta}+1}}}^{-1/2} \exp \bc{ - \lambda_1 \delta_1^2 -\lambda_2 \delta_2^2}= (1+o(1)) \exp \bc{\sum_{i \geq 3} \lambda_i \delta_i^2}
\end{align*}
establishing Condition $(4)$ and thus the distribution of $\Zgb \evO$. Since $\Erw \brk{\Zgb \bc{1-\evO}} = o\bc{\Erw\brk{\Zgb}}$ by \Lem~\ref{lem_o}, \Thm~\ref{thm_dist} follows from Markov's inequality.

\section{Discussion}

Studying partition functions has a long tradition in combinatorics and mathematical physics. $k$-SAT, $q$-coloring or the stochastic block model are just some noteworthy examples where the partition function reveals fundamental and novel combinatorial insights. Due to its connection to the {\sc Max Cut} problem and the disassortative stochastic block model, the Ising antiferromagnet fits nicely into this list. 
For random $d$-regular graphs, Coja-Oghlan et al. \cite{Coja_2020} pinpointed its replica symmetry breaking phase transition at the Kesten-Stigum bound. Using the method of moments and spatial mixing arguments, they they determine $\Zgb$ up to $\exp(o(n))$.
In this paper, we move beyond this approximation and derive the limiting distribution of $Z_{\G, \beta}$ in the replica symmetric regime. We note that the distribution of $Z_{\G, \beta}$ above the Kesten-Stigum bound is fundamentally different.
A similar analysis for the \Erdos-\Renyi-model was carried out in \cite{Mossel_2011}.

Using the combination of the method of moments and small subgraph conditioning underlying our proof was initially pioneered by Robinson \& Wormald \cite{Robinson_1992} to prove that cubic graphs are \whp\ Hamiltonian. Janson \cite{Janson_1995} subsequently showed that small subgraph conditioning can be used to obtain limiting distributions. This strategy was successfully applied, among others, to the stochastic block model \cite{Mossel_2011} and the Viana-Bray model \cite{Guerra_2004}. 
For other problems, the second moment appears to be too crude for the entire replica symmetric phase and enhanced techniques are needed \cite{Coja_2018_2}. In this work, we enrich the classical strategy of the method of moments and small subgraph conditioning by spatial mixing arguments to cover the entire replica symmetric phase.

An interesting remaining question is to throw a bridge between the properties of the partition function $Z_{\G, \beta}$ and long-range correlations in $\G$. While it should be a small step from \Thm~\ref{thm_dist} to vindicate the absence of long-range correlations in the replica symmetric phase, proving the presence of long-range correlations above the Kesten-Stigum bound is a more challenging, yet important endeavour.

\section{Getting started} \label{app_start}

Before moving to the proofs of \Prop s~\ref{prop_cycle_planted}, \ref{prop_first_moment} and \ref{prop_second_moment}, let us introduce some additional notation.
With $\cP$ denoting the set of all probability distributions on a finite set $\Omega \neq \emptyset$ and two probability measures $\mu, \nu \in \cP(\Omega)$, let us introduce the entropy $H(\mu)$ and Kullback-Leibler divergence $\KL{\mu}{\nu}$
\begin{equation*}
    H(\mu) = - \sum_{\omega \in \Omega} \mu(\omega) \log \mu(\omega) \qquad \text{and} \qquad \KL{\mu}{\nu} = \sum_{\omega \in \Omega} \mu(\omega) \log \frac{\mu(\omega)}{\nu(\omega)} \in [0, \infty].
\end{equation*}
Note the convention $0 \cdot \log \left( \frac{0}{0}\right)=0$ and furthermore that if there exists some $\omega \in \Omega$ such that $\mu (\omega) > 0$ and $\nu (\omega) = 0$, this implies $D_{\textrm{KL}}(\mu \| \nu) = \infty$.
When we consider the product measure between two probability distribution $\mu$ and $\nu$, we will use the notation $\mu \tensor \nu$.

Next, let us state a fundamental result by Janson \cite{Janson_1995} which stipulates conditions under which one is able to obtain the limiting distribution of the partition function.

\begin{theorem}[Theorem 1 in \cite{Janson_1995}] \label{thm_janson}
Let $\lambda_i>0$ and $\delta_i \geq -1, i=1,2,\dots,$ be constants and suppose that for each $n$ there are random variables $C_{in}, i=1,2,\dots,$ and $Z_n$ (defined on the same probability space) such that $X_{in}$ is non-negative integer valued and $\Erw \brk{Z_{n}} \neq 0$ (at least of large n), and furthermore the following conditions are satisfied:
\begin{enumerate}
    \item $C_{in} \xrightarrow{d} \Lambda_i$ as $n \to \infty$, jointly for all $i$ where $\Lambda_i \sim \Po\bc{\lambda_i}$ are independent Poisson random variables;
    \item For any finite sequence $c_1, \dots, c_m$ of non-negative integers,
    \begin{equation*}
        \frac{\Erw \brk{Z_n \mid C_{1n}=c_1, \dots, C_{mn}=c_m}}{\Erw \brk{Z_n}} \to \prod_{i=1}^m (1+\delta_i)^{x_i} \exp \bc{-\lambda_i \delta_i} \qquad \text{as} \quad n \to \infty;
    \end{equation*}
        \item $\sum_{i} \lambda_i \delta_i^2 < \infty;$
        \item $\Erw\brk{Z_n^2}/\bc{\Erw\brk{Z_n}}^2 \to \exp \bc{\sum_i \lambda_i \delta_i^2} \qquad \text{as} \quad n \to \infty.$
\end{enumerate}
Then, we have
\begin{equation*}
    \frac{Z_n}{\Erw \brk{Z_n}} \xrightarrow{d} W = \prod_{i \geq 1} (1+\delta_i)^{\Lambda_i} \exp \bc{-\lambda_i \delta_i};
\end{equation*}
moreover, this and the convergence in (1) hold jointly. The infinite product defining $W$ converges a.s. and in $L^2$, with $\Erw \bc{W} = 1$ and $\Erw \bc{W^2} = \exp \bc{\sum_{i=1}^\infty \lambda_i \delta_i^2}$. Hence, the normalized variables $Y_n/\Erw \bc{Y_n}$ are uniformly square integrable. Furthermore, the event $W > 0$ equals, up to a set of probability zero, the event that $Z_i > 0$ for some $i$ with $\delta_i = -1$. In particular, $W > 0$ a.s. if and only if every $\delta_i > -1$.
\end{theorem}

A substantial part of this paper is devoted to determining the first and second moment of $Z_\G$. As we will see in due course, this task requires a special version of the well-known Laplace's method, which is usually formulated in terms of integrals. In contrast to that, the model considered here is discrete and therefore requires a variation of Laplace's method which is applicable to countable sums. Fortunately, \cite{Greenhill_2010} provides an adaptation that we can leverage here. Let us start by providing the result of interest:

\begin{theorem}[Theorem 2.3 in \cite{Greenhill_2010}]\label{theoremGJR}
	Suppose the following:
	\begin{enumerate}
		\item $\mathcal{L} \subset \mathbb{R}^N$ is a lattice with rank $r \leq N$.
		\item $V \subseteq  \mathbb{R}^N$ is the $r$-dimensional subspace spanned by $\mathcal{L}$.
		\item $W = V + w$ is an affine subspace parallel to $V$, for some $w \in \mathbb{R}^N$.
		\item $K \subset \mathbb{R}^N$ is a compact convex set with non empty interior $K^\circ$.
		\item $\phi : K \rightarrow \mathbb{R}$ is a continuous function and the restriction of $\phi$ to $K \cap W$ has a unique maximum at some point $x_0 \in K^\circ \cap W$.
		\item $\phi$ is twice continuously differentiable in a neighbourhood of $x_0$ and $H := D^2 \phi \left( x_0\right)$ is its Hessian at $x_0$.
		\item $\psi : K_1 \rightarrow \mathbb{R}$ is a continuous function on some neighbourhood $K_1 \subseteq K$ of $x_0$ with $\psi \left( x_0\right) > 0$.
		\item For each positive integer $n$ there is a vector $\ell_n \in \mathbb{R}^N$ with $\frac{\ell_n}{n} \in W$.
		\item For each positive integer $n$ there is a positive real number $b_n$ and a function $a_n: \left( \mathcal{L} + \ell_n \right) \cap n K \rightarrow \mathbb{R}$ such that, as $n \rightarrow \infty$,
		\begin{align*}
			a_n \left( \ell \right)  = O \left( b_n e^{n \phi \left( \ell / n \right) + o\left( n\right)  } \right), \hspace{5 em} \ell \in \left( \mathcal{L} + \ell_n \right) \cap n K, 
		\end{align*}
		and 
		\begin{align*}
			a_n \left( \ell\right) = b_n \left( \psi \left( \frac{\ell}{n}\right) + o(1)  \right) e^{n \phi \left( \ell / n\right) } , \hspace{4 em} \ell \in \left( \mathcal{L} + \ell_n \right) \cap n K_1,
		\end{align*}
		uniformly for $\ell$ in the indicated sets.
	\end{enumerate}
	Then, provided $\det \left( - H \vert_V \right) \neq 0$, as $n \rightarrow \infty$,
	\begin{align*}
		\sum_{\ell \in \left( \mathcal{L} + \ell_n \right) \cap n K} a_n \left( \ell \right) \sim \frac{\left( 2 \pi n \right)^{r/2} \psi \left( x_0\right)  b_n e^{n \phi \left( x_0 \right) }}{\det \left( \mathcal{L} \right) \sqrt{\det \left( - H \vert_V \right)}}.
	\end{align*}
\end{theorem}
\Thm~\ref{theoremGJR} is largely self-explanatory. The concept of lattices, however, is not obvious from the theorem itself. Therefore, we briefly revisit the idea of lattices and how they are connected to our model. In general, lattices are discrete subgroups of $\mathbb{R}^N$ where each lattice is isomorphic to $\mathbb{Z}^r$ for some $0 \leq r \leq N$. In this context, \textit{discrete} simply means that the intersection of a lattice with an arbitrary, bounded set in $\mathbb{R}^N$ is finite. Furthermore, $r$ is commonly called the rank of the respective lattice. This means that each lattice has a (not necessarily unique) basis consisting of the vectors $x_1, \ldots, x_r$. The crucial characteristics of these basis vectors are on the one hand that they are independent. On the other hand, every element of the respective lattice has a unique representation of the form $\sum_{i =1}^{r} k_i \cdot x_i$ where $k_i \in \mathbb{Z}$ for all $i \in [r]$.\\
In applying \Thm~\ref{theoremGJR} we are especially interested in understanding the determinant $\det \left(\mathcal{L} \right)$ for a given lattice $\mathcal{L}$. Formally, $\det \left(\mathcal{L} \right)$ is simply obtained by calculating the determinant of the matrix that consists of the basis vectors $x_1, \ldots, x_r$ mentioned above. Intuitively, the determinant provides the $r$-dimensional volume of a unit cell of the lattice $\mathcal{L}$. Note that the term $\left( \det \left(\mathcal{L} \right)\right)^{-1}$ in Theorem \ref{theoremGJR} is the key difference compared to more common versions of Laplace's method for integrals.

\section{Short Cycles in the Regular Stochastic Block Model / Proof of Proposition \ref{prop_cycle_planted}}

Let us start with a brief repetition of the Regular Stochastic Block Model (RSBM) which is the result of the following experiment. Given a vertex set $V_n = \cbc{v_1, \ldots, v_n}$, we first sample a spin configuration uniformly at random. We denote this uniformly sampled configuration by $\SIGMA^*$. Next, we draw a $d$-regular graph $\G^* = \G^* (\SIGMA^*)$ from the distribution
\begin{align*}
	\Pr \brk{\mathbb{G}^* = G \vert \SIGMA^* = \sigma} \propto \exp \left( - \beta \mathcal{H}_G \left( \sigma \right) \right).
\end{align*}
For some graph $d$-regular $G$ with $n$ nodes and some spin configuration $\sigma \in \left\lbrace \pm 1 \right\rbrace^n$ on the nodes of $G$ we define
\begin{align} \label{eq_def_mu_1}
	\mu_{++}(G,\sigma) := \frac{2}{dn} \sum_{(u,v) \in E} \vecone {\left\lbrace \sigma \left( v \right) = \sigma \left( u\right)  = +1 \right\rbrace }.
\end{align}
Since $G$ has $\frac{d n}{2}$ edges in total, $\mu_{++}$ simply measures the fraction of edges that connect two positive vertices. Analogously, we define 
\begin{align} \label{eq_def_mu_2}
	\mu_{--}(G,\sigma)  &:= \frac{2}{dn} \sum_{(u,v) \in E} \vecone {\left\lbrace \sigma \left( v \right) = \sigma \left( u\right)  = -1 \right\rbrace }
	\hspace{3 em} \textrm{ and } \hspace{3 em} \\
	\mu_{+-}(G,\sigma) &= \mu_{-+}(G,\sigma) : = \frac{1}{dn} \sum_{(u,v) \in E} \vecone {\left\lbrace \sigma \left( v \right) \neq \sigma \left( u\right) \right\rbrace }.
\end{align}
Due to the fact that our model is built on undirected edges, we just count all the edges connecting vertices with different spins and evenly 'split' them between $\mu_{+-}$ and $\mu_{-+}$. In a similar way, we define
\begin{align} \label{eq_def_rho}
	\rho_+(G,\sigma) := \frac{1}{n} \sum_{v \in V} \vecone {\left\lbrace \sigma \left( v \right) = +1 \right\rbrace } \hspace{3 em} \textrm{ and } \hspace{3 em} \rho_-(G,\sigma) := \frac{1}{n} \sum_{v \in V} \vecone {\left\lbrace \sigma \left( v \right) = -1 \right\rbrace }
\end{align}
where $\rho_{+}$ and $\rho_{- }$ depict the fractions of nodes that have been assigned a positive spin or a negative one, respectively. For notational convenience, we usually drop the reference to the graph $G$ and the spin configuration $\sigma$. Accordingly, let $\mu' = \mu(\G^*, \SIGMA^*)$ and $\cM(\sigma)$ denote the set of all probability distributions fulfilling the obvious symmetry and marginalization conditions, i.e.
\begin{align*}
    \mu_{++} + \mu_{+-} = \rho_+, \qquad \mu_{--} + \mu_{+-} = \rho_-, \qquad \mu_{+-} = \mu_{-+}
\end{align*}
and where $\mu_{++}dn/2, \mu_{--}dn/2$ and $\mu_{+-}dn/2$ are integers.
Further, we define a probability measure $\hat{\mu}$ with 
\begin{align*}
    \hat{\mu}_{++}= \hat{\mu}_{--}= \frac{e^{- \beta}}{2 \left(1 +e^{-\beta} \right)} \qquad \text{and} \qquad \hat{\mu}_{+-}= \hat{\mu}_{-+}= \frac{1}{2 \left(1 +e^{-\beta} \right)}.
\end{align*}
To determine the distribution of short cycles in the RSBM, we start by considering the event
\begin{align*}
	\cA_{\mu} := \cbc{ \norm{\mu' - \mu} = O\bc{ n^{-1/2} \log n}}.
\end{align*}
In the next lines, we establish that $\cA_{\hat \mu}$ is a high probability event.

\begin{lemma} \label{A_mu_high_prob_event}
We have $\Pr\brk{\cA_{\hat\mu}} = 1- o(1)$.
\end{lemma}

\begin{proof}[Proof of Lemma \ref{A_mu_high_prob_event}]
In the following we will write $\mu$ for $\mu(G, \sigma)$ when the reference to $G$ and $\sigma$ is obvious. For this proof, we leverage some results that are derived in detail in Section \ref{section_first_moment}. More specifically, we consider equation \eqref{first_moment_before_Laplace}, that is
	\begin{align*}
		\mathbb{E}\left[ Z_{\mathbb{G}\left( n, d\right), \beta}   \right] 
		=  \exp\left(O \left( \frac{1}{n}\right)\right)  \cdot \sum_{\left( \rho_+, \mu_{++}\right) \in \mathcal{Q}}  \frac{1}{\pi n\sqrt{ 2 \mu_{++}\mu_{--} \mu_{+-} d}} \exp\left(  n  \psi\left( \mu_{++}, \rho_+\right)\right).
	\end{align*}
	where
	\begin{align*}
		\psi\left( \mu_{++}, \rho_+\right)  :=  \textrm{H} \left( \rho\right) - \frac{d}{2}\left( D_\textrm{KL} (\mu \vert\vert \rho \otimes \rho) + \beta  \left(1 +2 \mu_{++} -2 \rho_+ \right)\right)
	\end{align*}
and $\mathcal{Q}$ is the set of all conceivable pairs $\bc{ \rho_+, \mu_{++}}$. Furthermore, from Lemma \ref{maximum_first_moment} we know that $\psi\bc{ \mu_{++}, \rho_+}$ obtains it unique maximum on $\mathcal{Q}$ at $\bc{ \hat{\mu}_{++}, \hat{\rho}_+} = \bc{\frac{e^{- \beta}}{2 \bc{1 + e^{- \beta}}}, \frac{1}{2}}$. The entries of the Hessian turn out to be
\begin{align*}
	\frac{\partial^2 \psi}{\partial \mu_{++}^2}  \bc{\hat{\mu}, \hat{\rho}}  &= -2 d \left( 1 + e^{- \beta}\right)^2  e^\beta = \Theta(1) \\
	\frac{\partial^2 \psi} {\partial \mu_{++}\partial \rho_+}\bc{\hat{\mu}, \hat{\rho}} &= 2 d \left( 1 + e^{- \beta}\right)^2 e^\beta = \Theta(1)\\
	\frac{\partial^2 \psi}{\partial \rho_+^2}\bc{\hat{\mu}, \hat{\rho}} &= -4 -2 d \left(1 + e^{-\beta} + 2 e^\beta \right) = \Theta(1).
\end{align*}
Note that a detailed calculation of the Hessian can be found in Section \ref{section_first_moment}.
With all these results at hand, the two dimensional Taylor expansion of $\psi$ at $\bc{\hat{\mu}, \hat{\rho}}$ turns out to be
\begin{align*}
	\psi\left( \mu_{++}, \rho_+\right) &= 	\psi\bc{\hat{\mu}, \hat{\rho}}
	+ \Theta \bc{1} \bc{ \bc{\rho_+ - \hat{\rho}_+}^2 + \bc{\mu_+ - \hat{\mu}_+}^2 + \bc{\rho_+ - \hat{\rho}_+}\bc{\mu_+ - \hat{\mu}_+}} + O \bc{ \norm{\mu - \hat{\mu}}^3}\\
	&= 	\psi\bc{\hat{\mu}, \hat{\rho}} +  \Theta \bc{ \norm{\mu - \hat{\mu}}^2}
\end{align*}
where we exploited that the higher order derivatives are bounded.
Keeping this in mind, we obtain
\begin{align*}
	 &\exp\left(O \left( \frac{1}{n}\right)\right)  \cdot \sum_{\left( \rho_+, \mu_{++}\right) \in \mathcal{Q}}  \frac{1}{\pi n\sqrt{ 2 \mu_{++}\mu_{--} \mu_{+-} d}} \exp\left(  n  \psi\left( \mu_{++}, \rho_+\right)\right) \vecone{\left\lbrace 1 - \cA_{\hat \mu} \right\rbrace}\\
	&= \sum_{\left( \rho_+, \mu_{++}\right) \in \mathcal{Q}}  \exp\left(  n  \psi\bc{\hat{\mu}, \hat{\rho}} - \Omega \bc{\log^2 n} \right) \vecone{\left\lbrace 1 - \cA_{\hat \mu} \right\rbrace}\\
	&=   O \bc{n^2}  \exp\left(  n  \psi\bc{\hat{\mu}, \hat{\rho}} - \Omega \bc{\log^2 n} \right) 
	=   O \bc{n^{- \log n}}  \exp\left(  n  \psi\bc{\hat{\mu}, \hat{\rho}}  \right)\\
\end{align*}
which in turn yields
\begin{align*}
	\mathbb{E} \brk{Z_{\mathbb{G}, \beta}}&= \exp\left(O \left( \frac{1}{n}\right)\right)  \cdot \sum_{\left( \rho_+, \mu_{++}\right) \in \mathcal{Q}}  \frac{1}{\pi n\sqrt{ 2 \mu_{++}\mu_{--} \mu_{+-} d}} \exp\left(  n  \psi\left( \mu_{++}, \rho_+\right)\right) \vecone{\left\lbrace \cA_{\hat \mu} \right\rbrace}\\
	&\qquad + \exp\left(O \left( \frac{1}{n}\right)\right)  \cdot \sum_{\left( \rho_+, \mu_{++}\right) \in \mathcal{Q}}  \frac{1}{\pi n\sqrt{ 2 \mu_{++}\mu_{--} \mu_{+-} d}} \exp\left(  n  \psi\left( \mu_{++}, \rho_+\right)\right) \vecone{\left\lbrace 1 - \cA_{\hat \mu} \right\rbrace}\\
	&= \bc{1 + o(1)}  \cdot \sum_{\left( \rho_+, \mu_{++}\right) \in \mathcal{Q}}  \frac{1}{\pi n\sqrt{ 2 \mu_{++}\mu_{--} \mu_{+-} d}} \exp\left(  n  \psi\left( \mu_{++}, \rho_+\right)\right) \vecone{\left\lbrace \cA_{\hat \mu} \right\rbrace}\\
	&= \bc{1 + o(1)}	\mathbb{E} \brk{Z_{\mathbb{G}, \beta}\vecone{\left\lbrace \cA_{\hat \mu} \right\rbrace}}.
\end{align*}
Now, the proof is almost completed. Corollary 4.5 in \cite{Coja_2020} states that iff $\mathbb{E} \brk{Z_{\mathbb{G}, \beta}} = \bc{1 + o(1)} \mathbb{E} \brk{Z_{\mathbb{G}, \beta}\vecone{\left\lbrace \cA_{\hat \mu} \right\rbrace}}$ holds, we have $\Pr \brk{\mathbb{G}^* \in \cA_{\hat \mu}} = 1 - o(1)$. This is just the desired statement.
\end{proof}

The following preliminary arguments combine ideas from  \cite{Janson_2011} and \cite{Mossel_2011} to derive the distribution of short cycles in $\G^*$. We apply the method of moments to derive expected values conditional on $\mu'$ being close to $\hat{\mu}$. Then, with Lemma \ref{A_mu_high_prob_event}, we draw conclusions for the unconditional expectation.
Let $C_l \left( \mathbb{G}^* \right) $ be the number of cycles of length $l$ in  $\mathbb{G}^*$.
Furthermore, let $M$ denote the number of edges $e_1, \ldots e_l$ that connect vertices with opposite spins. This construction immediately implies that $M$ is an even number.
Let us briefly recap the configuration model to construct a $d$-regular graph on $n$ uniformly at random. To get started, we take $d$ copies of each of the $n$ nodes. Thus, we have $dn$ nodes in total. In the next step, we choose a perfect matching uniformly at random. To obtain a graph with $n$ nodes again, we merge the $d$ copies of each node, providing a graph with $\frac{d n}{2}$ edges in total. Since this procedure does not rule out self-loops or double-edges, we condition on the event $\mathcal{S}$ that we obtain a simple graph. Note that standard results from the literature entail that $\mathbb{P} \left[ G \in \mathcal{S}\right] = \Omega \left( 1\right) $. Similarly, conditional on $\mathcal{S}$, each of the admissable $d$-regular graphs is created with the same probability. \\
Now recall the probability to observe a specific graph in the regular stochastic block model 
\begin{align} \label{eq_rsbm}
	\mathbb{P} \left[ \mathbb{G}^* = G \vert \sigma \right] \propto \exp \left( - \beta \mathcal{H}_G \left( \sigma \right) \right).
\end{align}
Clearly, the definition of $\mathbb{G}^*$ does not give rise to a uniform distribution over all admissable graphs. However, it is easy to see that \eqref{eq_rsbm} yields a uniform distribution over all graphs exhibiting a specific $\mu$. This observation is central towards deriving the distribution of short cycles in $\G^*$.

\begin{lemma} \label{lemma_cycle_planted_optimal_mu}
	Let 
	\begin{equation*}
		\Xi_i \sim \Po \bc{\lambda_i \bc{1 + \delta_i}}
	\end{equation*}
	be a sequence of independent Poisson random variables for $i \geq 3$. Then jointly for all $i$ we have $C_{i}(\mathbb{G}^*) \vert_{\hat{\mu}} \xrightarrow{d} \Xi_i$ as $n \to \infty$.
\end{lemma}
\begin{proof}
Let $p_{l, M}$ be the probability that any given set of $l$ edges where $l_{++}$ edges connect two positive vertices and $l_{--}$ edges connect two negative edges results from the construction of $\mathbb{G}^*$ conditioned on some some fixed $\mu$. We readily find
\begin{align*}
    	p_{l, M} \left( \mu \right) &= \frac{\binom{dn \rho_{+} - 2 l_{++} - M}{dn \mu_{++}- 2 l_{++}}\left( dn \mu_{++} - 2 l_{++}- 1 \right) !!}
	{\binom{dn \rho_{+}}{dn \mu_{++}}\left( dn \mu_{++} - 1 \right) !!} \cdot \frac{\binom{dn \rho_{-} - 2 l_{--} - M}{dn \mu_{--} - 2 l_{--}} \left( dn \mu_{--}  - 2 l_{--}- 1 \right) !! \left( d n \mu_{+-} - M  \right)!}
	{\binom{dn \rho_{-}}{dn \mu_{--}} \left( dn \mu_{--} - 1 \right) !! \left( d n \mu_{+-}  \right)!}.
\end{align*}
Using the following well-known identity.
\begin{align} \label{eq_df}
    \left( 2k -1 \right) !! = \frac{\left( 2k\right) !}{k! 2^k}
\end{align}
we find
\begin{align} \label{eq_cycle_1}
	\frac{\left( dn \mu_{++} - 2 l_{++}- 1 \right) !!  \left( dn \mu_{--}  - 2 l_{--}- 1 \right) !! \left( d n \mu_{+-} - M \right)!}
	{\left( dn \mu_{++} - 1 \right) !!  \left( dn \mu_{--} - 1 \right) !! \left( d n \mu_{+-} \right)!} 
	&=\frac{\frac{\left( dn \mu_{++} - 2 l_{++}\right)! }{\left( \frac{d n}{2}\mu_{++} - l_{++}\right)! 2^{\frac{d n}{2}\mu_{++} - l_{++}}}
	\cdot \frac{\left( dn \mu_{--} - 2 l_{--}\right)! }{\left( \frac{d n}{2}\mu_{--} - l_{--}\right)! 2^{\frac{d n}{2}\mu_{--} - l_{--}}}}
	{\frac{\left( dn \mu_{++}\right)!}{\left( \frac{d n}{2} \mu_{++} \right)! 2^{\frac{d n}{2} \mu_{++}}}  \cdot \frac{\left( dn \mu_{--}\right)!}{\left( \frac{d n}{2} \mu_{--} \right)! 2^{\frac{d n}{2} \mu_{--}}} \cdot \left( d n \mu_{+-} \right)_M } \notag\\
	&= 2^{l - M}  \cdot \frac{ \left( \frac{d n}{2}\mu_{++}\right)_{l_{++}} \left( \frac{d n}{2}\mu_{--}\right)_{l_{--}} }{\left( d n \mu_{++} \right)_{2 l_{++}} \left( d n \mu_{--} \right)_{2 l_{--}}\left( d n \mu_{+-} \right)_{ M}}.
\end{align}
Moving on to the binomial coefficients and using Stirling's formula
\begin{align} \label{eq_stirling}
    k! = \sqrt{2 \pi k} \left( \frac{k}{e}\right) ^k \exp \left( O \left( \frac{1}{k} \right) \right)
\end{align}
we obtain
\begin{align} \label{eq_cycle_2}
	\frac{\binom{dn \rho_{+} - 2 l_{++} - M}{dn \mu_{++}- 2 l_{++}} \binom{dn \rho_{-} - 2 l_{--} - M}{dn \mu_{--} - 2 l_{--}} }
	{\binom{dn \rho_{+}}{dn \mu_{++}}\binom{dn \rho_{-}}{dn \mu_{--}}} 
	&= \frac{\left( d n \mu_{++}\right)_{2 l_{++}} \left( d n \mu_{--}\right)_{2 l_{--}} }{\left( d n \rho_{+}\right)_{2 l_{++} + M} \left( d n \rho_{-}\right)_{2 l_{--} + M}}
	\cdot \frac{\left( d n \rho_{+} - d n \mu_{++} \right)!  \left( d n \rho_{-} - d n \mu_{--} \right)!}{\left( d n \rho_{+} - d n \mu_{++} - M \right)!  \left( d n \rho_{-} - d n \mu_{--} - M \right)!} \notag \\
	&= \frac{\left( d n \mu_{++}\right)_{2 l_{++}} \left( d n \mu_{--}\right)_{2 l_{--}} \left( d n \rho_{+} - d n \mu_{++} \right)_M  \left( d n \rho_{-} - d n \mu_{--} \right)_M}{\left( d n \rho_{+}\right)_{2 l_{++} + M} \left( d n \rho_{-}\right)_{2 l_{--} + M}}	
\end{align}
Combining \eqref{eq_cycle_1} and \eqref{eq_cycle_2}, we yield
\begin{align*}
	p_{l, M} \left( \mu \right) = 2^{l - M} \cdot \frac{\left( \frac{d n}{2}\mu_{++}\right)_{l_{++}} \left( \frac{d n}{2}\mu_{--}\right)_{l_{--}} \left( d n \rho_{+} - d n \mu_{++} \right)_M  \left( d n \rho_{-} - d n \mu_{--} \right)_M }{\left( d n \rho_{+}\right)_{2 l_{++} + M} \left( d n \rho_{-}\right)_{2 l_{--} + M}\left( d n \mu_{+-} \right)_{ M}}.
\end{align*}
In particular, we thus have for all $\hat{\mu}' \in \cA_{\hat\mu}$
\begin{align*}
	p_{l, M} \left( \hat{\mu}' \right) 
	&= 2^{l - M} \frac{ \left( \frac{d n}{2} \frac{e^{- \beta}}{2 \left(1 +e^{-\beta}\right) }\right)^{l_{++}} \left( \frac{d n }{2}\frac{e^{- \beta}}{2 \left(1 +e^{-\beta}\right) }\right)^{l_{--}} \left( d n \frac{1}{2 \left( 1 + e^{- \beta}\right) }\right)^M \left( d n \frac{1}{2 \left( 1 + e^{- \beta}\right) }\right)^M  }{\left( \frac{d n}{2} \right)^{2 l_{++} + M} \left( \frac{d n}{2}\right)^{2 l_{--} + M}\left( d n \frac{1}{2 \left(1 +e^{-\beta}\right) } \right)^{M}} \bc{1 + o(1)}\\
	&= \frac{2^{l - M} \left(  \frac{e^{- \beta}}{2 \left(1 +e^{-\beta}\right) }\right)^{l - M} \left(  \frac{1}{  1 + e^{- \beta} }\right)^{2 M}}{\left( \frac{d n}{2} \right)^{ l_{++} } \left( \frac{d n}{2}\right)^{ l_{--}}\left( d n \frac{1}{2 \left(1 +e^{-\beta}\right) } \right)^{M}}\bc{1 + o(1)}
	=  \left( \frac{2}{dn} \right)^l \left( \frac{e^{- \beta}}{ \left(1 +e^{-\beta}\right) }\right)^{l - M} \left( \frac{1}{ \left(1 +e^{-\beta}\right) } \right)^{M}\bc{1 + o(1)}.
\end{align*}
We point out that $p_{l, M}\left( \hat{\mu}' \right)$ can asymptotically be expressed without $l_{++}$ and $l_{--}$. 
Next, we consider the number of possible cycles with length $l$ and exactly $M$ edges that connect vertices with opposite spins, subsequently denoted by $a_{l, M}\left( \mu\right)$. For starters, we have
\begin{align*}
	2 l \cdot a_{l, M}\left( \mu\right) = 2 \binom{l}{M}\left( n \rho_{+} \right)_{l_{+}}  \left( n \rho_{-} \right)_{l_{-}} \left( d \left( d -1\right) \right)^l.
\end{align*}
This implies for $\hat{\mu}' \in \cA_{\hat\mu}$
\begin{align*}
	a_{l, M}\left( \hat{\mu}' \right) = \binom{l}{M} \frac{1}{ l} n^l 2^{-l} \left( d \left( d -1\right) \right)^l \bc{1 + o(1)}.
\end{align*}
Now, we are in a position to calculate the conditional expectation of the number of short cycles, that is
\begin{align*}
	\mathbb{E}\left[  C_l \left( \mathbb{G}^*\right) \vert  \cA_{\hat\mu} \right] &= \sum_{i = 0}^l  p_{l, M=i}\left(  \hat{\mu}\right) a_{l, M=i}\left(  \hat{\mu}\right) \bc{1 + o(1)} \\
	&\sim \sum_{i = 0, i \textrm{ even}}^l \ \left( \frac{2}{dn} \right)^l \left( \frac{e^{- \beta}}{ \left(1 +e^{-\beta}\right) }\right)^{l - M} \left( \frac{1}{ \left(1 +e^{-\beta}\right) } \right)^{M} \binom{l}{i} \frac{1}{ l} n^l 2^{-l} \left( d \left( d -1\right) \right)^l \\
	&=\frac{\left( d - 1\right)^l}{ l}\sum_{i = 0, i \textrm{ even}}^l  \binom{l}{i}   \left( \frac{e^{- \beta}}{1 + e^{- \beta}}\right)^{l - i} \left( \frac{1}{1 + e^{- \beta}}\right)^i  \\
	&= \frac{\left( d - 1\right)^l}{2 l}  \left(  \left( \frac{e^{- \beta}}{1 + e^{- \beta}} +\frac{1}{1 + e^{- \beta}} \right)^l +   \left( \frac{e^{- \beta}}{1 + e^{- \beta}} -\frac{1}{1 + e^{- \beta}} \right)^l  \right)\\
	&= \frac{\left( d - 1\right)^l}{2 l}  \left(  1+   \left( \frac{e^{- \beta} - 1}{1 + e^{- \beta}}  \right)^l  \right) =: \lambda^*.
\end{align*}
In order to establish \Prop~\ref{prop_cycle_planted} we next need to calculate the higher moments of the number of short cycles in $\G^*$. To this end, we consider $\mathbb{E}\left[  C_l \left( \mathbb{G}^*\right)^2 \vert \cA_{\hat\mu} \right]$ which can be interpreted as the expected number of ordered pairs of cycles in $\mathbb{G}$. We introduce two new random variables, namely $X'$ and $X''$. $X'$ denotes the number of ordered cycle pairs that are vertex-disjoint whereas $X''$ counts the ordered cycle pairs that have at least one vertex in common. This immediately brings us to
\begin{align*}
	\mathbb{E}\left[  C_l \left( \mathbb{G}^*\right)^2 \vert  \cA_{\hat\mu} \right] = \mathbb{E}\left[  X' \vert  \cA_{\hat\mu} \right] + \mathbb{E}\left[  X'' \vert  \cA_{\hat\mu} \right].
\end{align*}
Starting with $X'$ and adopting a corresponding definition of $p'_{ l,  M}$ and $a'_{l, M}$ - just now referring to two vertex-disjoint cycles - an analogue calculation to the one above yields
\begin{align*}
	p'_{ l,  M} \left( \hat{\mu} \right) \sim 
	\left( \frac{2}{dn} \right)^{2 l} \left( \frac{e^{- \beta}}{ \left(1 +e^{-\beta}\right) }\right)^{2 l - 2 M} \left( \frac{1}{ \left(1 +e^{-\beta}\right) } \right)^{2 M}
\end{align*}
and
\begin{align*}
	\left( 2 l \right)^2 \cdot a'_{l, M}\left( \mu\right) = 4 \left( \binom{l}{M} \right)^2 \left( n \rho_{+} \right)_{2 l_{+}}  \left( n \rho_{-} \right)_{2 l_{-}} \left( d \left( d -1\right) \right)^{2l}.
\end{align*}
Therefore, we arrive at
\begin{align*}
	\mathbb{E}\left[  X' \vert  \cA_{\hat\mu} \right] \sim \left( \lambda^* \right)^2. 
\end{align*}
All that remains to do is to show that $\mathbb{E}\left[  X'' \vert  \cA_{\hat\mu} \right]$ is asymptotically dominated by $\mathbb{E}\left[  X' \vert  \cA_{\hat\mu} \right]$. More precisely, we show that $\mathbb{E}\left[  X'' \vert  \cA_{\hat\mu} \right] = O \left( n^{- 1}\right)$ where we adopt an argument from \cite{Janson_2011} to our case. Whenever we have two cycles of length $l$ that have $k$ vertices in common, the number of shared vertices will exceed the number of shared edges by at least one. Put differently, the number of shared edges is at most $k - 1$. As a result of this insight we have
\begin{align*}
	a_{l, M}\left( \hat{\mu}' \right) = \Theta \left( n^{2 l - k}\right)  \hspace{3 em} \textrm{ and } \hspace{3 em} 	p_{l, M}\left( \hat{\mu}' \right) = O \left( n^{ -2 l + k - 1}\right)
\end{align*}
for any $k < l$ and $\hat{\mu}' \in \cA_{\hat\mu}$. Summing up over all $k \in \left[ l - 1 \right] $ yields the desired statement
\begin{align*}
	\mathbb{E}\left[  X'' \vert  \cA_{\hat\mu} \right] = O \left( n^{- 1}\right).
\end{align*}
This same argumentation can be extended to arbitrary higher moments $\mathbb{E}\left[  C_l \left( \mathbb{G}^*\right)^j \vert  \cA_{\hat\mu} \right]$ with $j \in \mathbb{N}$. Thus, the method of moments provides the desired statement.
\end{proof}

\begin{proof}[Proof of Proposition \ref{prop_cycle_planted}]
	The Proposition results from combining \Lem s~\ref{A_mu_high_prob_event} and \ref{lemma_cycle_planted_optimal_mu}. 
\end{proof}

\section{The First Moment/ Proof of Proposition \ref{prop_first_moment}}\label{section_first_moment}

In this section, we first focus on the so-called pairing model $\Gbold = \Gbold \bc{n, d}$. In pairing model, each of the $n$ initial nodes is represented by $d$ clones. Then, a perfect matching for these $d n$ clones is chosen uniformly at random. Finally, the clones are merged back into their initial vertex, such that each node in the original vertex set has degree $d$. By design, this setup allows for loops and double edges. If the graph does not contain either of them, we call the graph simple. Furthermore, we denote the event that a graph is simple by $\mathcal{S}$. The following result (which we will prove first) can be leveraged for showing Proposition \ref{prop_first_moment}.

\begin{proposition}\label{prop_first_moment_pairing_model}
	Assume that $0 < \beta < \bks$ and $d \geq 3$. Then we have
	\begin{equation*}
		\Erw \brk{Z_{\Gbold, \beta} } =  \exp \bc{O \bc{\frac{1}{n}}}  \sqrt{\frac{1 + e^{ \beta}}{ 2 + d  e^{\beta} - d }}  \exp \bc{n \bc{\bc{1 - d/2} \log \bc{2}  + d\log\bc{1 + e^{- \beta}}/2}}
	\end{equation*}
\end{proposition}
\subsection{Getting started}

Recall the definitions of $\mu(G, \sigma)$ and $\rho(G, \sigma)$ from \eqref{eq_def_mu_1}--\eqref{eq_def_rho}.
As a starting point for our first moment calculations, consider the following result due to \cite{Coja_2020} which encodes the combinatorial structure of the first moment of the partition function. Let $\cM_n = \cup_{\sigma \in \cbc{\pm 1}^{V_n}} \cM(\sigma)$ be the set of all conceivable distributions $\mu$.

\begin{lemma}[\Lem s~4.1 and 4.3 in \cite{Coja_2020}]\label{first_mom_starting_point}
We have
\begin{align*}
	\mathbb{E}\left[ Z_{\Gbold , \beta}   \right] 
	&=  \sum_{\mu \in \mathcal{M}_n} 	\binom{n}{\rho_+ n} \frac{\left( d n \mu_{++} - 1\right) !! \left( d n \mu_{--} - 1\right)!!  \left(d n \mu_{+-} \right)  !}{\left( d n  - 1\right) !!} 	\binom{dn \rho_+}{dn \mu_{++}} 	\binom{dn \rho_-}{dn \mu_{--}}\cdot \exp \left(- \beta \frac{d n}{2} \left( \mu_{++} + \mu_{--}\right)  \right).
\end{align*}
\end{lemma}


\subsection{Reformulation of the first moment}
Recall Stirling's formula \eqref{eq_stirling} and the identity for the double factorial from \eqref{eq_df}.
The next Lemma yields a simplified expression for the first moment which is  obtained by applying \eqref{eq_stirling} and \eqref{eq_df} to the factorials and binomial coefficients in Lemma \ref{first_mom_starting_point}. The proof follows \cite{Coja_2020}, but now explicitly accounting for smaller-order terms to yield an error term of order $O(\exp(1/n))$.
\begin{lemma}\label{reformulation_first_moment}
We have
\begin{align*}
	\mathbb{E}\left[ Z_{\Gbold, \beta}   \right] 
	&= \sum_{\mu \in \mathcal{M}_n}   \frac{\exp\left(n  \textrm{H} \left( \rho\right) - \frac{dn}{2}\left( D_\textrm{KL} (\mu \vert\vert \rho \otimes \rho) + \beta  \left( \mu_{++} + \mu_{--}\right)\right)  +  O \left( \frac{1}{n}\right)\right)}{\pi n\sqrt{ 2 \mu_{++}\mu_{--} \mu_{+-} d}} .
	\end{align*}
\end{lemma}

\begin{proof}
Starting with \Lem~\ref{first_mom_starting_point} and considering the fraction of factorials first, we find
	\begin{align} \label{eq_firstmoment_1}
		&\frac{\left( d n \mu_{++} - 1\right) !! \left( d n \mu_{--} - 1\right)!!  \left(d n \mu_{+-} \right)  !}{\left( d n  - 1\right) !!}
		= \frac{\left( d n \mu_{++} \right) ! \left( d n \mu_{--} \right) ! \left(d n \mu_{+-} \right) !   \left( \frac{d n}{2}\right) !\cdot  2^{\frac{d n}{2}}}{\left( d n \right)!   \left( \frac{d n \mu_{++}}{2}\right) !  \left( \frac{d n \mu_{--}}{2}\right) ! \cdot 2^{\frac{d n}{2} \left( \mu_{++} + \mu_{--}\right) }} \notag \\
		&= \exp \left( O \left( \frac{1}{n} \right) \right) \sqrt{2 \pi  \frac{dn \mu_{++} dn \mu_{--} dn \mu_{+-} \frac{d n}{2}}{dn \frac{d n \mu_{++}}{2}\frac{d n \mu_{--}}{2}}} \cdot 2^{\frac{d n}{2} \left( 1 -\mu_{++} - \mu_{--}\right)} \cdot  \left( \frac{dn}{e}\right)^{dn \left(  \mu_{+-}  - \frac{1}{2} + \frac{\mu_{++}}{2} + \frac{\mu_{--}}{2}\right) } \notag\\
		&\qquad \qquad  \cdot 2^{  dn \left( \frac{\mu_{++}}{2} + \frac{\mu_{--}}{2} - \frac{1}{2} \right)}\cdot  \mu_{++}^{dn \left(\mu_{++} - \frac{\mu_{++}}{2} \right) } \cdot  \mu_{--}^{dn \left(\mu_{--} - \frac{\mu_{--}}{2} \right)} \cdot  \mu_{+-}^{dn \mu_{+-}} \notag\\
		&= \exp \left( O \left( \frac{1}{n} \right) \right) 2 \sqrt{ \pi d n \mu_{+-} } \cdot  \mu_{++}^{dn \frac{\mu_{++}}{2} } \cdot  \mu_{--}^{dn  \frac{\mu_{--}}{2} } \cdot  \mu_{+-}^{dn \mu_{+-}} \notag\\
		&=2 \exp \left( O \left( \frac{1}{n}\right) + \frac{1}{2}  \log \left( \pi d n \mu_{+-}\right) + \underbrace{dn \frac{\mu_{++}}{2} \log\left( \mu_{++} \right)  +  dn \frac{\mu_{--}}{2} \log\left( \mu_{--} \right) + dn \mu_{+-} \log \left( \mu_{+-}\right)}_{= -\frac{dn}{2} \textrm{H}\left( \mu \right)}\right) \notag\\
		&= \exp \left( -\frac{dn}{2} \textrm{H}\left( \mu \right) +  \frac{1}{2}  \log \left( n\right)  + \frac{1}{2}  \log \left(4 \pi d  \mu_{+-}\right) +  O \left( \frac{1}{n}\right)\right) 
	\end{align}
where we used \eqref{eq_df} for the first equality and Stirling's formula \eqref{eq_stirling} for the second equality. Similarly, we rearrange the second term of interest:
	\begin{align*}
		\binom{dn \rho_+}{dn \mu_{++}} 	\binom{dn \rho_-}{dn \mu_{--}} = \frac{\left( d n \rho_+ \right)! \left( d n \rho_- \right)!   }{\left( d n \mu_{++} \right)!  \left( d n \mu_{--} \right)! \left( d n \left(\rho_+ - \mu_{++} \right) \right)!  \left( d n \left(\rho_- - \mu_{--} \right) \right)!} = \frac{\left( d n \rho_+ \right)! \left( d n \rho_- \right)!   }{\left( d n \mu_{++} \right)!  \left( d n \mu_{--} \right)! \left( \left( dn \mu_{+-}\right)!\right) ^2}.
	\end{align*}
Another application of \eqref{eq_stirling} yields
	\begin{align} \label{eq_firstmoment_2}
	    &\binom{dn \rho_+}{dn \mu_{++}} 	\binom{dn \rho_-}{dn \mu_{--}} \notag \\
		&=\exp \left( O \left( \frac{1}{n} \right) \right) \frac{1}{2 \pi d n} \sqrt{\frac{\rho_+ \rho_-}{\mu_{++} \mu_{--} \mu_{+-}^2}} \left( \frac{d n}{e}\right) ^{dn \left( \rho_+ + \rho_- - \mu_{++} - \mu_{--} - 2 \mu_{+-}\right) } \notag\\
		&\qquad \qquad \cdot \rho_+^{d n \rho_+}	\cdot \rho_-^{d n \rho_-} \cdot \mu_{++}^{-dn \mu_{++}} \cdot \mu_{--}^{-dn \mu_{--}} \cdot \mu_{+-}^{- 2dn \mu_{+-}} \notag\\
		&= \exp \left( d n \rho_+  \log \left( \rho_+\right) + d n \rho_-  \log \left( \rho_-\right) - dn \mu_{++} \log \left( \mu_{++} \right)  - dn \mu_{--} \log \left( \mu_{--} \right)  - 2 dn \mu_{+-} \log \left( \mu_{+-} \right)   \right)\notag \\
		&\qquad \qquad \cdot \exp \left( - \log \left( n\right) + \frac{1}{2} \log \left( \frac{\rho_+ \rho_-}{\mu_{++} \mu_{--} \mu_{+-}^2 4 \pi^2 d^2}\right)+ O \left( \frac{1}{n} \right) \right) \notag\\
		&= \exp\left(dn \left( \textrm{H}\left( \mu\right)  - \textrm{H}\left( \rho\right)\right)   - \log \left( n\right) + \frac{1}{2} \log \left( \frac{\rho_+ \rho_-}{\mu_{++} \mu_{--} \mu_{+-}^2 4 \pi^2 d^2}\right) + O \left( \frac{1}{n} \right) \right) 
	\end{align}
Combining \eqref{eq_firstmoment_1} and \eqref{eq_firstmoment_2} and denoting by we have
\begin{align*}
	&\frac{\left( d n \mu_{++} - 1\right) !! \left( d n \mu_{--} - 1\right)!!  \left(d n \mu_{+-} \right)  !}{\left( d n  - 1\right) !!} 
		\binom{dn \rho_+}{dn \mu_{++}} 	\binom{dn \rho_-}{dn \mu_{--}}  \\
	&=\exp\left(\frac{dn}{2} \left( \textrm{H}\left( \mu\right)  - \textrm{H}\left( \rho \otimes \rho \right)\right)   - \frac{1}{2} \log \left( n\right) + \frac{1}{2} \log \left( \frac{\rho_+ \rho_-}{\mu_{++} \mu_{--} \mu_{+-}  \pi d}\right)+  O \left( \frac{1}{n}\right)\right) \\
	&=\exp\left(- \frac{dn}{2} D_\textrm{KL} (\mu \vert\vert \rho \otimes \rho)   - \frac{1}{2} \log \left( n\right) + \frac{1}{2} \log \left( \frac{\rho_+ \rho_-}{\mu_{++} \mu_{--} \mu_{+-}  \pi d}\right)+  O \left( \frac{1}{n}\right)\right) 
\end{align*}
As an immediate consequence, the first moment  from \Lem~\ref{first_mom_starting_point} can be expressed as
	\begin{align*}
		\mathbb{E}\left[ Z_{\Gbold, \beta}   \right] = \sum_{\mu \in \mathcal{M}_n} \binom{n}{\rho_+ n} \sqrt{\frac{\rho_+ \rho_-}{\mu_{++} \mu_{--} \mu_{+-}  \pi d n }}\exp\left(- \frac{dn}{2}\left( D_\textrm{KL} (\mu \vert\vert \rho \otimes \rho) + \beta  \left( \mu_{++} + \mu_{--}\right)\right) +  O \left( \frac{1}{n}\right) \right)
\end{align*}
where $\cM_n$ is again the set of all conceivable distributions $\mu$. A short auxiliary calculation using Stirling's formula \eqref{eq_stirling} yields
\begin{align*}
	\binom{n}{\rho_+ n} &= \frac{n!}{\left( \rho_+ n\right)! \left( \rho_- n\right)! } = \frac{1}{\sqrt{2 \pi n \rho_+ \rho_-}} \left( \frac{n}{e}\right) ^n \left( \frac{\rho_+ n}{e}\right) ^{- \rho_+ n}\left( \frac{\rho_- n}{e}\right) ^{- \rho_- n} \exp \left(O \left(\frac{1}{n} \right)  \right) \\
	&= \frac{\rho_+^{- \rho_+ n} \rho_-^{- \rho_- n}}{\sqrt{2 \pi n \rho_+ \rho_-}} \exp \left(O \left(\frac{1}{n} \right)  \right)  
	= \exp\left( n \textrm{H} \left( \rho\right) - \frac{1}{2} \log\left( n \right) - \frac{1}{2} \log\left( 2 \pi  \rho_+ \rho_-\right) +O \left(\frac{1}{n} \right) \right) 
\end{align*}
which enables us to state
\begin{align*}
	\mathbb{E}\left[ Z_{\Gbold, \beta}   \right] 
	&= \sum_{\mu \in \mathcal{M}_n}   \frac{\exp\left(n  \textrm{H} \left( \rho\right) - \frac{dn}{2}\left( D_\textrm{KL} (\mu \vert\vert \rho \otimes \rho) + \beta  \left( \mu_{++} + \mu_{--}\right)\right)  +  O \left( \frac{1}{n}\right)\right)}{\pi n\sqrt{ 2 \mu_{++}\mu_{--} \mu_{+-} d}}
\end{align*}
as claimed.
\end{proof}

Revisiting the setup of our model, we see that all values of $\mu$ and $\rho$ are completely determined by the choice of $\mu_{++}$ and $\rho_{+}$. Exploiting the fact that $\rho$ is a probability distribution, we have
\begin{align*}
	\rho_{-} = 1 - \rho_{+}.
\end{align*}
A similar argument can be made for $\mu$. Since edges are by definition undirected in our setup, we have $\mu_{+-} = \mu_{-+}$. Keeping in mind that $\mu$ is also a probability measure, the missing weights of $\mu$ can be deduced from $\mu_{++}$ and $\rho_{+}$ by the equations
\begin{align*}
	\mu_{+-} &= \mu_{-+} = \rho_+ - \mu_{++}\\
	\mu_{--} &= 1 - 2 \left( \rho_+ - \mu_{++}\right) - \mu_{++} = 1 + \mu_{++} - 2 \rho_+.
\end{align*}
Substituting the above into \Lem~\ref{reformulation_first_moment} and some simplifications give us
\begin{align}\label{first_moment_before_Laplace}
	\mathbb{E}\left[ Z_{\Gbold, \beta}   \right] 
	=  \exp\left(O \left( \frac{1}{n}\right)\right)  \cdot \sum_{\left( \rho_+, \mu_{++}\right) \in \mathcal{Q}}  \frac{1}{\pi n\sqrt{ 2 \mu_{++}\mu_{--} \mu_{+-} d}} \exp\left(  n  \psi\left( \mu_{++}, \rho_+\right)\right).
\end{align}
where
\begin{align*}
	\psi\left( \mu_{++}, \rho_+\right)  :=  \textrm{H} \left( \rho\right) - \frac{d}{2}\left( D_\textrm{KL} (\mu \vert\vert \rho \otimes \rho) + \beta  \left(1 +2 \mu_{++} -2 \rho_+ \right)\right)
\end{align*}
and $\mathcal{Q}$ is the set of all conceivable pairs $\bc{ \rho_+, \mu_{++}}$. The KL-divergence can also be expressed just in terms of $\mu_{++}$ and $\rho_{+}$, as the following calculation shows.
\begin{align*}
	D_\textrm{KL} \left( \mu \vert\vert \rho \otimes \rho\right)  = &\mu_{++} \log \left(\frac{\mu_{++}}{\rho_+^2} \right) + \mu_{--} \log \left(\frac{\mu_{--}}{\rho_-^2} \right) + 2 \mu_{+-} \log \left(\frac{\mu_{+-}}{\rho_+ \rho_-} \right) \\
	= &\mu_{++} \log \left(\frac{\mu_{++}}{\rho_+^2} \right) + \left(1 + \mu_{++} - 2 \rho_+ \right)  \log \left(\frac{1 + \mu_{++} - 2 \rho_+ }{\left( 1 -\rho_+\right)^2} \right) + 2 \left( \rho_+ - \mu_{++}\right)  \log \left(\frac{\rho_+ - \mu_{++}}{\rho_+ \left( 1 -\rho_+\right) } \right)\\
	= &\mu_{++} \log\left( \mu_{++}\right)  + \left(1 + \mu_{++} - 2 \rho_+ \right)  \log \left(1 + \mu_{++} - 2 \rho_+  \right) \\
	&- 2 \left(1 -  \rho_+ \right) \log \left( 1 - \rho_+\right)+  2 \left( \rho_+ - \mu_{++}\right) \log \left( \rho_+ -\mu_{++}\right) - 2  \rho_+  \log\left( \rho_+\right).
\end{align*}
Having effectively reduced the number of involved variables, we now can move on to apply the Laplace method as stated in Theorem $2.3$ in \cite{Greenhill_2010}.

\subsection{Application of the Laplace method to the first moment}

Before we can apply the Laplace method to the expression for the first moment in \eqref{first_moment_before_Laplace}, we need some preliminary work. To be precise, we need to determine the unique maximum $\left( \hat{\mu}_{++}, \hat{\rho}_{+}\right)$ of $\psi$ on the set $\mathcal{Q}$ and evaluate the Hessian at this point. To this end, consider 
\begin{align} \label{eq_def_hat_rho}
    \hat{\rho}_{+} =\hat{\rho}_{-} = \frac{1}{2},
\end{align}
i.e. balanced number of vertices with positive and negative spins. Moreover, let
\begin{align} \label{eq_def_hat_mu}
	\hat{\mu}_{++} = \hat{\mu}_{--} = \frac{e^{- \beta}}{2 \left( 1 + e^{- \beta} \right) } \hspace{3 em} \textrm{and} \hspace{3 em} 	\hat{\mu}_{+-} = \hat{\mu}_{-+} = \frac{1}{2 \left( 1 + e^{- \beta} \right) }.
\end{align}
We will see in due course  in \Lem~\ref{maximum_first_moment} that $\left( \hat{\mu}_{++}, \hat{\rho}_{+}\right)$ indeed constitutes the unique maximum of $\psi$. Let us first calculate partial derivatives and establish the Hessian of $\psi$ at $\left( \hat{\mu}_{++}, \hat{\rho}_{+}\right)$.



\begin{lemma}[Hessian for the first moment]\label{hessian_first_moment}
We have
\begin{align*}
	\det\left( - \textrm{Hes}_\psi \left( \frac{1}{2}, \frac{e^{-\beta}}{2 \left(1 + e^{- \beta}\right) } \right) \right) 
	= 4 d \left( 1 + e^{- \beta}\right)^2  e^\beta \left(2 + d \left( e^{\beta} - 1\right)  \right).
\end{align*}
\end{lemma}

\begin{proof}
Let us get started simple and state the partial derivatives of the Kullback-Leibler divergence from \eqref{first_moment_before_Laplace} with respect to $\mu_{++}$ and $\rho_+$.
\begin{align*}
	\frac{\partial D_\textrm{KL} \left( \mu \vert\vert \rho \otimes \rho\right)}{\partial \mu_{++}}  
	&= \log \left( \mu_{++}\right)  + \log \left(1 + \mu_{++} - 2 \rho_+  \right)   - 2 \log \left( \rho_+ -\mu_{++}\right) \\
	\frac{\partial^2 D_\textrm{KL} \left( \mu \vert\vert \rho \otimes \rho\right)}{\partial \mu_{++}^2}  &= \frac{1}{\mu_{++}} + \frac{1}{1 + \mu_{++} - 2 \rho_+} + \frac{2}{\rho_+ -\mu_{++}}\\
	\frac{\partial D_\textrm{KL} \left( \mu \vert\vert \rho \otimes \rho\right) }{\partial \rho_{+}} 
	&= - 2 \log \left(1 + \mu_{++} - 2 \rho_+  \right) +2 \log \left( 1 - \rho_+\right)  +2  \log \left( \rho_+ -\mu_{++}\right) - 2 \log\left( \rho_+\right)\\
	\frac{\partial^2 D_\textrm{KL} \left( \mu \vert\vert \rho \otimes \rho\right) }{\partial \rho_{+}^2} &= \frac{4}{1 + \mu_{++} - 2 \rho_+ } - \frac{2}{1 - \rho_+} + \frac{2}{\rho_+ - \mu_{++}} - \frac{2}{\rho_+}.
\end{align*}
Furthermore, for the entropy we recall
\begin{align*}
	\frac{\partial \textrm{H} \left( \rho_+\right)}{\partial \rho_+} = \log\left(1 - \rho_+ \right) - \log\left(\rho_+ \right)  \hspace{2 em} \textrm{and} \hspace{2 em} \frac{\partial^2 \textrm{H} \left( \rho_+\right)}{\partial \rho_+^2} = - \frac{1}{1 - \rho_+} - \frac{1}{\rho_+}.
\end{align*}
Keeping these auxiliary calculations in mind, the first derivatives of $\psi$ turn out to be
\begin{align*}
	\frac{\partial \psi\left( \mu_{++}, \rho_+\right)}{\partial \mu_{++}}  &=  -\frac{d}{2} \left(\log \left( \mu_{++}\right)  + \log \left(1 + \mu_{++} - 2 \rho_+  \right)   - 2 \log \left( \rho_+ -\mu_{++}\right) + 2 \beta   \right)\\
	\frac{\partial \psi\left( \mu_{++}, \rho_+\right)}{\partial \rho_+}  &=  \log\left(1 - \rho_+ \right) - \log\left(\rho_+ \right)- d \left( -  \log \left(1 + \mu_{++} -  2 \rho_+  \right) + \log \left( 1 - \rho_+\right) +  \log \left( \rho_+ -\mu_{++}\right) -  \log\left( \rho_+\right) - \beta \right) 
\end{align*}
while the second derivatives of $\psi$ are given by
\begin{align*}
	\frac{\partial^2 \psi\left( \mu_{++}, \rho_+\right)}{\partial \mu_{++}^2}  &=  - \frac{d}{2}\left( \frac{1}{\mu_{++}} + \frac{1}{1 + \mu_{++} - 2 \rho_+} + \frac{2}{\rho_+ -\mu_{++}}\right)\\
	\frac{\partial^2 \psi\left( \mu_{++}, \rho_+\right)} {\partial \mu_{++}\partial \rho_+}&=\frac{\partial^2 \psi\left( \mu_{++}, \rho_+\right)} {\partial \rho_+\partial \mu_{++}} = -\frac{d}{2} \left( - \frac{2}{1 + \mu_{++} - 2 \rho_+}- 2 \frac{1}{\rho_+ - \mu_{++}}\right)=  \frac{d \left(  1  -  \rho_+\right) }{\left( 1 + \mu_{++} - 2 \rho_+\right) \left( \rho_+ - \mu_{++}\right)}\\
	\frac{\partial^2 \psi\left( \mu_{++}, \rho_+\right)}{\partial \rho_+^2} &=  - \frac{1}{1 - \rho_+} - \frac{1}{\rho_+} - d\left(\frac{2}{1 + \mu_{++} - 2 \rho_+ } - \frac{1}{1 - \rho_+} + \frac{1}{\rho_+ - \mu_{++}} - \frac{1}{\rho_+} \right).
\end{align*}
With the above at hand, the entries of the Hessian turn out to be
\begin{align}
	\frac{\partial^2 \psi}{\partial \mu_{++}^2}  \left( \frac{1}{2}, \frac{e^{-\beta}}{2 \left(1 + e^{- \beta} \right)}\right)  &=  - \frac{d}{2}\left( \frac{1}{ \frac{e^{-\beta}}{2 \left(1 + e^{- \beta} \right)}} + \frac{1}{1 +  \frac{e^{-\beta}}{2 \left(1 + e^{- \beta} \right)} - 1} + \frac{2}{\frac{1}{2} - \frac{e^{-\beta}}{2 \left(1 + e^{- \beta} \right)}}\right) \notag \\
	&= - 2d \left( \frac{1 + e^{- \beta} + e^{- \beta} + e^{-2 \beta}}{e^{- \beta}}\right)  = -2 d \left( 1 + e^{- \beta}\right)^2  e^\beta < 0 \label{first_principal_minor}
\end{align}
and
\begin{align*}
	\frac{\partial^2 \psi} {\partial \mu_{++}\partial \rho_+}\left( \frac{1}{2}, \frac{e^{-\beta}}{2 \left(1 + e^{- \beta} \right)}\right)&=\frac{\partial^2 \psi} {\partial \rho_+\partial \mu_{++}}\left( \frac{1}{2}, \frac{e^{-\beta}}{2 \left(1 + e^{- \beta} \right)}\right)  =  \frac{ \frac{d}{2} }{ \frac{e^{-\beta}}{2 \left(1 + e^{- \beta} \right)} \left( \frac{1}{2} -  \frac{e^{-\beta}}{2 \left(1 + e^{- \beta} \right)}\right)}\\
    &= \frac{ 2 d }{ \frac{e^{-\beta}}{ 1 + e^{- \beta} } \cdot \frac{1}{ 1 + e^{- \beta} }} = 2 d \left( 1 + e^{- \beta}\right)^2 e^\beta 
\end{align*}
and
\begin{align*}
	\frac{\partial^2 \psi}{\partial \rho_+^2}\left( \frac{1}{2}, \frac{e^{-\beta}}{2 \left(1 + e^{- \beta} \right)}\right) &=  - 2 - 2 - d\left(\frac{2}{ \frac{e^{-\beta}}{2 \left(1 + e^{- \beta} \right)}  } - 2 + \frac{1}{\frac{1}{2 \left(1 + e^{- \beta} \right)}} - 2 \right)\\
	&= -4 -d \left( \frac{4 \left(1 + e^{- \beta} \right)}{e^{- \beta}} -2 + 2 e^{- \beta}\right)
	= -4 -2 d \left(1 + e^{-\beta} + 2 e^\beta \right).
\end{align*}
Combining the above, the determinant of the Hessian at $\left( \hat{\mu}_{++}, \hat{\rho}_{+}\right)$ is given by
\begin{align}
	\det\left( - \textrm{Hes}_\psi \left( \frac{1}{2}, \frac{e^{-\beta}}{2 \left(1 + e^{- \beta}\right) } \right) \right) 
	&=  8 d \left( 1 + e^{- \beta}\right)^2  e^\beta  +4 d^2 \left(1 + e^{-\beta} + 2 e^\beta \right) \left( 1 + e^{- \beta}\right)^2  e^\beta - 4 d^2 \left( 1 + e^{- \beta}\right)^4 e^{2\beta} \notag \\
	&= 8 d \left( 1 + e^{- \beta}\right)^2  e^\beta  +4  d^2 \left( 1 + e^{- \beta}\right)^2  e^\beta \underbrace{\left( 1 +  e^{- \beta} + 2 e^\beta -   e^{\beta} - 2 - e^{- \beta} \right)}_{= e^\beta -1} \notag\\
	&= 4 d \left( 1 + e^{- \beta}\right)^2  e^\beta \left(2 + d \left( e^{\beta} - 1\right)  \right) >0.\label{second_principal_minor}
\end{align}
closing the proof of the lemma.
\end{proof}

With the partial derivatives in place, we can proceed to establish that the unique maximum of $\psi$ is indeed at $\left( \hat{\mu}_{++}, \hat{\rho}_{+}\right)$.

\begin{lemma}[Maximum for the First Moment Calculation]\label{maximum_first_moment}
With the definitions of $\hat{\rho}_{+}$ and $\hat{\mu}_{++}$ from \eqref{eq_def_hat_rho} and \eqref{eq_def_hat_mu} we have
\begin{align*}
	\arg \max_ {\left( \mu_{++}, \rho_{+} \right) \in \mathcal{Q}} \psi \left( \mu_{++}, \rho_{+} \right) =  \left( \hat{\mu}_{++}, \hat{\rho}_{+}\right)
\end{align*}
\end{lemma}

\begin{proof}
	As a starting point, we set the first derivatives equal to zero, resulting in
	\begin{align*}
		\frac{\partial \psi\left( \hat{\mu}_{++}, \hat{\rho}_+\right)}{\partial \hat{\mu}_{++}}  &=  -\frac{d}{2} \left(\log \left( \hat{\mu}_{++}\right)  + \log \left(1 + \hat{\mu}_{++} - 2 \hat{\rho}_+  \right)   - 2 \log \left(\hat{\rho}_+ -\hat{\mu}_{++}\right) + 2 \beta   \right) = 0
	\end{align*}
	which is equivalent to
	\begin{align*}
		0 &= \log \left( \hat{\mu}_{++}\right)  + \log \left(1 + \hat{\mu}_{++} - 2 \hat{\rho}_+  \right)   - 2 \log \left(\hat{\rho}_+ -\hat{\mu}_{++}\right) + 2 \beta \\
		\Leftrightarrow 1 &= \frac{\hat{\mu}_{++}  \left(1 + \hat{\mu}_{++} - 2 \hat{\rho}_+  \right)}{ \left( \hat{\rho}_+ -\hat{\mu}_{++}\right)^2} e^{2 \beta} \\
		\Leftrightarrow 0 &= \hat{\mu}_{++}^2 \left(1 - e^{-2 \beta} \right) + \hat{\mu}_{++} \left(1 - 2 \hat{\rho}_+ + e^{-2 \beta} 2 \hat{\rho}_+  \right) -  e^{-2 \beta} \hat{\rho}_+^2.
	\end{align*}
	Then, the quadratic formula yields two candidates for the solution, namely
	\begin{align*}
		\hat{\mu}_{++, 1/2} &= \frac{-1 + 2 \hat{\rho}_+ - e^{-2 \beta} 2 \hat{\rho}_+ \pm \sqrt{ \left(1 - 2 \hat{\rho}_+ \left( 1 - e^{-2 \beta}\right)  \right)^2+ 4  \left(1 - e^{-2 \beta} \right) e^{-2 \beta} \hat{\rho}_+^2}}{2 \left(1 - e^{-2 \beta} \right)}\\
		&= \hat{\rho}_+ - \frac{1  \mp \sqrt{ \left(1 - 2\hat{\rho}_+ \right)^2 + 4 \hat{\rho}_+ e^{-2 \beta}\left( 1- \hat{\rho}_+\right)  }}{2 \left(1 - e^{-2 \beta} \right)}.
	\end{align*}
This result immediately poses the question of possible extrema. First we note that
\begin{align*}
	\left(1 - 2 \hat{\rho}_+ \right)^2 + 4 \hat{\rho}_+ e^{-2 \beta}\left( 1- \hat{\rho}_+\right) > 0
\end{align*}
since both summands are positive.
This in turn enables us to rule out $\hat{\mu}_{++, 2}= \hat{\rho}_+ - \frac{1  + \sqrt{ \left(1 - 2\hat{\rho}_+ \right)^2 + 4 \hat{\rho}_+ e^{-2 \beta}\left( 1- \hat{\rho}_+\right)  }}{2 \left(1 - e^{-2 \beta} \right)}$ as a solution since that would imply
\begin{align*}
	\hat{\mu}_{-+} = \hat{\mu}_{+-} = \hat{\rho}_+ - \hat{\mu}_{++, 2}= \frac{1  + \sqrt{ \left(1 - 2\hat{\rho}_+ \right)^2 + 4 \hat{\rho}_+ e^{-2 \beta}\left( 1- \hat{\rho}_+\right)  }}{2 \left(1 - e^{-2 \beta} \right)} > \frac{1 }{2 \left(1 - e^{-2 \beta} \right)} > \frac{1}{2}
\end{align*}
which contradicts the fact that $\hat{\mu}$ is a probability measure. As a consequence, the only solution that is consistent with our model assumptions is
\begin{align}\label{mu++_in_terms_of_rho+}
	\hat{\mu}_{++}  = \hat{\rho}_+ - \frac{1  - \sqrt{ \left(1 - 2\hat{\rho}_+ \right)^2 + 4 \hat{\rho}_+ e^{-2 \beta}\left( 1- \hat{\rho}_+\right)  }}{2 \left(1 - e^{-2 \beta} \right)} = \hat{\rho}_+ - \frac{1  - \eta}{2 \left(1 - e^{-2 \beta} \right)}
\end{align}
where
\begin{align*}
    \eta := \sqrt{ \left(1 - 2\hat{\rho}_+ \right)^2 + 4 \hat{\rho}_+ e^{-2 \beta}\left( 1- \hat{\rho}_+\right)  } = \sqrt{ 1 - 4 \hat{\rho}_+ \left( 1 - e^{-2 \beta}  \right) +4 \hat{\rho}_+^2 \left( 1 - e^{-2 \beta}  \right)}
\end{align*}
In the next step, we plug \eqref{mu++_in_terms_of_rho+} into the first derivative of $\psi$ with respect to $\rho_{+}$
\begin{align*}
	\frac{\partial \psi\left( \hat{\mu}_{++},\hat{\rho}_+\right)}{\partial \hat{\rho}_+}  &=  \log\left(1 - \hat{\rho}_+ \right) - \log\left(\hat{\rho}_+ \right)\\ 
	&- d \left( -  \log \left(1 + \hat{\mu}_{++} -  2 \hat{\rho}_+  \right) + \log \left( 1 - \hat{\rho}_+\right)  +  \log \left( \hat{\rho}_+ -\hat{\mu}_{++}\right) -  \log\left( \hat{\rho}_+\right) - \beta \right) = 0
\end{align*}
which yields
	\begin{align*}
		\xi \left( \hat{\rho}_+ \right) &:= \log\left(1 - \hat{\rho}_+ \right) - \log\left(\hat{\rho}_+ \right)\\
		&- d \left( -  \log \left(1 - \frac{1  - \eta}{2 \left(1 - e^{-2 \beta} \right)} -  \hat{\rho}_+  \right) + \log \left( 1 - \hat{\rho}_+\right)  +  \log \left( \frac{1  - \eta}{2 \left(1 - e^{-2 \beta} \right)}\right) -  \log\left( \hat{\rho}_+\right) - \beta \right)\\
		&= \log\left(1 - \hat{\rho}_+ \right) - \log\left(\hat{\rho}_+ \right) - d \left( -  \log \left(\frac{\left( 1 -    \hat{\rho}_+\right) 2 \left(1 - e^{-2 \beta} \right)}{1  - \eta }  - 1 \right) + \log \left( 1 - \hat{\rho}_+\right)  -  \log\left( \hat{\rho}_+\right) - \beta \right)\\
		&= 0
	\end{align*}
Next, let us take a look at the derivative of $\xi$ with respect to $\hat{\rho}_+$:
\begin{align*}
	\frac{\partial \xi \left( \hat{\rho}_+ \right)}{\partial \hat{\rho}_+}&= \left( 1-d\right) \left( - \frac{1}{1 - \hat{\rho}_+} - \frac{1}{\hat{\rho}_+} \right) 
	+ d \left( \frac{\frac{- \left( 1 - \eta  \right) 2 \left(1 - e^{-2 \beta} \right) + \frac{2}{ \eta} \left( 2 \hat{\rho}_+ - 1 \right)  \left( 1 - e^{-2 \beta}  \right) \left( 1 -    \hat{\rho}_+\right) 2 \left(1 - e^{-2 \beta} \right)}{\left( 1 - \eta \right)^2 }}{\frac{\left( 1 -    \hat{\rho}_+\right) 2 \left(1 - e^{-2 \beta} \right)}{1  - \eta }  - 1} \right) \\
	&=   \frac{ d - 1}{\left( 1 - \hat{\rho}_+\right)  \hat{\rho}_+}  
	+ d \left( \frac{- \left( 1 - \eta  \right) 2 \left(1 - e^{-2 \beta} \right) + \frac{4}{ \eta} \left( 2 \hat{\rho}_+ - 1 \right)   \left( 1 -    \hat{\rho}_+\right)  \left(1 - e^{-2 \beta} \right)^2}{\left( 1 -    \hat{\rho}_+\right) 2 \left(1 - e^{-2 \beta} \right)\left( 1  - \eta \right)   - \left( 1 - \eta \right)^2} \right) \\	
\end{align*}
where we made use of the simple fact
\begin{align*}
	\frac{\partial \eta }{\partial \hat{\rho}_+} = \frac{1}{2 \eta} \left( 8\hat{\rho}_+ - 4 \right)  \left( 1 - e^{-2 \beta}  \right) =  \frac{2}{ \eta} \left( 2 \hat{\rho}_+ - 1 \right)  \left( 1 - e^{-2 \beta}  \right).
\end{align*}
To simplify the first derivative, we focus on
\begin{align*}
		&\frac{- \left( 1 - \eta  \right) 2 \left(1 - e^{-2 \beta} \right) + \frac{4}{ \eta} \left( 2 \hat{\rho}_+ - 1 \right)   \left( 1 -    \hat{\rho}_+\right)  \left(1 - e^{-2 \beta} \right)^2}{\left( 1 -    \hat{\rho}_+\right) 2 \left(1 - e^{-2 \beta} \right)\left( 1  - \eta \right)   - \left( 1 - \eta \right)^2}
		= \frac{ 2 \left(1 - e^{-2 \beta} \right) \cdot \left( - \left( 1 - \eta  \right)  + \frac{1 - \eta^2}{\eta} + \frac{2}{ \eta} \left(  \hat{\rho}_+ - 1 \right)   \left(1 - e^{-2 \beta} \right)\right) }
		{\left( 1 - \eta \right) \left(  1  -   2 \hat{\rho}_+  - 2 e^{-2 \beta} + 2 e^{-2 \beta} \hat{\rho}_+    + \eta \right) }\\
		&= \frac{2 \left(1 - e^{-2 \beta} \right) }{\left( 1 - \eta \right) \eta} \cdot \frac{   - \eta + 2  \hat{\rho}_+ - 1    - 2  \hat{\rho}_+ e^{-2 \beta} + 2 e^{-2 \beta}}
		{ 1  -   2 \hat{\rho}_+  - 2 e^{-2 \beta} + 2 e^{-2 \beta} \hat{\rho}_+    + \eta } 
		= -  \frac{2 \left(1 - e^{-2 \beta} \right) }{\left( 1 - \eta \right) \eta}.
	\end{align*}
As a consequence, the derivative can be simplified to
\begin{align*}
	\frac{\partial \xi \left( \hat{\rho}_+ \right)}{\partial \hat{\rho}_+} &=   \frac{d -1}{\left( 1- \rho_+\right) \rho_+} - 2 d \left( \frac{ 1- e^{-2 \beta}}{\left(1 - \eta \right) \eta} \right).
\end{align*}
Before proceedings, we point out that 
\begin{align*}
	\left( 1 - \hat{\rho}_+\right) \hat{\rho}_+ = \frac{1 - \eta^2}{\left( 1 - e^{-2 \beta}\right) 4 } 
\end{align*}
which brings us to
	\begin{align*}
		\frac{\partial \xi \left( \hat{\rho}_+ \right)}{\partial \hat{\rho}_+} &=\frac{d -1}{\left( 1- \hat{\rho}_+\right) \hat{\rho}_+} - 2 d  \left( \frac{ 1- e^{-2 \beta}}{    \left(1 - \eta \right) \eta} \right) 
		= \frac{1 - e^{-2 \beta}}{1 - \eta}\left( \frac{4d -4 }{1 + \eta}- \frac{2 d}{\eta}\right) \\
		&=\frac{1 - e^{-2 \beta}}{\left( 1 - \eta^2\right)  \eta}\left( 4d \eta - 4 \eta - 2d - 2d \eta \right) =\frac{1 - e^{-2 \beta}}{\left( 1 - \eta^2\right)  \eta}\left( 2d \left( \eta  -1 \right) - 4 \eta - 2d  \right) < 0
	\end{align*}
where we implicitly assumed that $\hat{\rho}_+$ is conceivable which especially means that $0 < \eta < 1$ holds.
$\frac{\partial \xi \left( \hat{\rho}_+ \right)}{\partial \hat{\rho}_+} < 0$  implies that if we can locate any root of $\xi \left( \hat{\rho}_+ \right)$ it is automatically  the unique one. 
Recalling our definition of $\rho_+$ and $\mu_{++}$ from \eqref{eq_def_hat_rho} and \eqref{eq_def_hat_mu}, we conjecture that this root is located at $\hat{\rho}_+ = \frac{1}{2}$. A short calculation indeed verifies
\begin{align*}
	\xi \left(\frac{1}{2} \right) & = - d \left( -  \log \left(\frac{\left( 1 -    \frac{1}{2}\right) 2 \left(1 - e^{-2 \beta} \right)}{1  - \eta }  - 1 \right) + \log \left( 1 - \frac{1}{2}\right)  -  \log\left(\frac{1}{2}\right) - \beta \right)\\
	& =  d \left(   \log \left(\frac{ 1 - e^{-2 \beta}}{1  - e^{- \beta} }  - 1 \right)  + \beta \right) =  d \left(   \log \left(e^{- \beta } \right)  + \beta \right) = 0
\end{align*}
where we used
\begin{align*}
	\eta = \sqrt{ \left(1 - 2\hat{\rho}_+ \right)^2 + 4 \hat{\rho}_+ e^{-2 \beta}\left( 1- \hat{\rho}_+\right)  } =  \sqrt{ \left(1 - 2 \cdot \frac{1}{2} \right)^2 + 4 \cdot \frac{1}{2}  \cdot e^{-2 \beta}\left( 1- \frac{1}{2} \right)  } = e^{- \beta}.
\end{align*}
This immediately allows us to calculate the optimal $\hat{\mu}_{++}$ by plugging $\hat{\rho}_+ = \frac{1}{2}$ into equation (\ref{mu++_in_terms_of_rho+})
\begin{align*}
	\hat{\mu}_{++}   = \frac{1}{2} - \frac{1  - e^{- \beta}}{2 \left(1 - e^{-2 \beta} \right)} = \frac{1}{2} - \frac{1}{2 \left(1 + e^{- \beta} \right)} = \frac{e^{- \beta}}{2 \left(1 + e^{- \beta} \right)}.
\end{align*}
The above establishes that $\left( \hat{\mu}_{++}, \hat{\rho}_{+}\right)$ is the (only) extremum of $\psi$. Let us next show that it is indeed the global maximum (and not a minimum or stationary point). From the calculation of the Hessian (\Lem~\ref{hessian_first_moment}), we saw that the first leading principal minor is negative (see inequality (\ref{first_principal_minor})) and the second one is positive (see inequality (\ref{second_principal_minor})). Thus, $\psi$ is strictly concave at $\left( \hat{\mu}_{++}, \hat{\rho}_{+}\right)$ which makes it a local maximum. Due to the uniqueness of the extremum, $\left( \hat{\mu}_{++}, \hat{\rho}_{+}\right)$ thereby also is the unique maximum.
\end{proof}

As an application of \Lem~\ref{maximum_first_moment}, we obtain the following corollary.

\begin{corollary} \label{maximum_first_moment_cor}
We have
\begin{align*}
	\max_ {\left( \mu_{++}, \rho_{+} \right) \in \mathcal{Q}} \psi \left( \mu_{++}, \rho_{+} \right) = \psi \left( \hat{\mu}_{++}, \hat{\rho}_{+}\right) = \left(1 - \frac{d}{2} \right) \log\left( 2\right)  + \frac{d}{2}\log\left( 1 + e^{- \beta} \right)
\end{align*}
\end{corollary}

\begin{proof}
We evaluate $\psi$ at the optimal point $\left( \hat{\mu}_{++}, \hat{\rho}_{+}\right)$. Starting with the entropy we have
\begin{align} \label{eq_first_moment_entropy_opt}
	\textrm{H} \left( \hat{\rho} \right) = - \frac{1}{2} \log\left( \frac{1}{2}\right) - \frac{1}{2} \log\left( \frac{1}{2}\right) = \log \left( 2\right). 
\end{align}
Continuing with the Kullback-Leibler divergence, we find
\begin{align} \label{eq_first_moment_kldivergence_opt}
	D_\textrm{KL} \left( \hat{\mu }\vert\vert \hat{\rho} \otimes\hat{\rho}\right)  
	&=  \frac{e^{- \beta}}{2 \left( 1 + e^{- \beta} \right) } \log\left( \frac{e^{- \beta}}{2 \left( 1 + e^{- \beta} \right) }\right)  +  \frac{e^{- \beta}}{2 \left( 1 + e^{- \beta} \right) } \log \left(  \frac{e^{- \beta}}{2 \left( 1 + e^{- \beta} \right)} \right)\\
	&+  2 \left( \frac{1}{2} -  \frac{e^{- \beta}}{2 \left( 1 + e^{- \beta} \right) }\right) \log \left( \frac{1}{2} - \frac{e^{- \beta}}{2 \left( 1 + e^{- \beta} \right) }\right) - 2 \log\left( \frac{1}{2}\right)  \\
	&=  \frac{e^{- \beta}}{ \left( 1 + e^{- \beta} \right) } \log\left( \frac{e^{- \beta}}{2 \left( 1 + e^{- \beta} \right) }\right)   + 2 \log \left( 2\right) +    \frac{1}{ \left( 1 + e^{- \beta} \right) } \log \left( \frac{1}{2 \left( 1 + e^{- \beta} \right) }\right) \\
	&=  \log \left( 2\right) - \frac{\beta e^{- \beta}}{ \left( 1 + e^{- \beta} \right) } - \log \left( 1 + e^{- \beta} \right).
\end{align}
Combining \eqref{eq_first_moment_entropy_opt} and \eqref{eq_first_moment_kldivergence_opt} we arrive at
\begin{align*}
	\psi \left( \hat{\mu}_{++}, \hat{\rho}_{+}\right)  &= \textrm{H} \left( \hat{\rho}\right) - \frac{d}{2}\left( D_\textrm{KL} (\hat{\mu} \vert\vert \hat{\rho} \otimes\hat{\rho}) + \beta  \left(1 +2 \hat{\mu}_{++} -2 \hat{\rho}_+ \right)\right)\\
	&= \log \left( 2\right)  - \frac{d}{2}\left( \log \left( 2\right) - \frac{\beta e^{- \beta}}{ \left( 1 + e^{- \beta} \right) } - \log \left( 1 + e^{- \beta} \right) + 2 \beta  \frac{e^{- \beta}}{2 \left( 1 + e^{- \beta} \right) } \right)\\
	&= \left(1 - \frac{d}{2} \right) \log\left( 2\right)  + \frac{d}{2}\log\left( 1 + e^{- \beta} \right).
\end{align*}
as claimed.
\end{proof}

\begin{proof}[Proof of \Prop~\ref{prop_first_moment_pairing_model}]
With \Lem s~\ref{hessian_first_moment} and \ref{maximum_first_moment} and \Cor~\ref{maximum_first_moment_cor} in place, all that is left for the application of Laplace's method from \cite{Greenhill_2010} is the determination of the appropriate lattice. 
Put differently, we are interested in the respective matrix $A_\textrm{first}$ which consist of the basis elements of the lattice. 
For the first moment, the matrix can be constructed in a rather simple way. Since $\rho_+$ is of the form
\begin{align*}
	\rho_+ = \frac{1}{n} \sum_{v \in V} \vecone {\left\lbrace \sigma \left( v \right) = +1 \right\rbrace }
\end{align*}
the first entry $A_{\textrm{first},1,1}$ immediately turns out to be equal to one. Similarly, keeping in mind
\begin{align*}
	\mu_{++}  = \frac{2}{dn} \sum_{(u,v) \in E} \vecone {\left\lbrace \sigma \left( v \right) = \sigma \left( u\right)  = +1 \right\rbrace }
\end{align*}
yields  $A_{\textrm{first},2,2} =  \frac{2}{d}$. Having constructed the matrix $A_\textrm{first}$, we are left to compute its determinant
\begin{align*}
	\det \left( A_{\textrm{first}} \right)=	\det \begin{pmatrix} 
		1 & 0  \\ 
		0 & \frac{2}{d}  \\ 
	\end{pmatrix} = \frac{2}{d}.
\end{align*}
Now, we can bring together all the findings of this section to obtain a precise statement of the first moment up to an error term of order $O(\exp(1/n))$.
Applying the Laplace method, i.e. Theorem 2.3 in \cite{Greenhill_2010} to  expression \eqref{first_moment_before_Laplace}  yields
\begin{align*}
    \Erw \left[ Z_{\Gbold, \beta}   \right] 
    &=  \exp\left(O \left( \frac{1}{n}\right)\right)  \cdot \sum_{\left( \rho_+, \mu_{++}\right) \in \mathcal{Q}}  \frac{1}{\pi n\sqrt{ 2 \mu_{++}\mu_{--} \mu_{+-} d}} \exp\left(  n  \psi\left( \mu_{++}, \rho_+\right)\right)\\
	&=  \exp\left(O \left( \frac{1}{n}\right)\right) \frac{2 \pi n \exp \left(n \psi \left( \hat{\mu}_{++}, \hat{\rho}_{+}\right) \right) }{\det \left( A_{\textrm{first}} \right) \sqrt{\det\left( - \textrm{Hes}_\psi \left( \hat{\mu}_{++}, \hat{\rho}_{+}\right)  \right) }\pi n\sqrt{ 2 \hat{\mu}_{++}\hat{\mu}_{--} \hat{\mu}_{+-} d}} \\
	&=  \exp\left(O \left( \frac{1}{n}\right)\right) \frac{2  \exp \left(n \left(  \left(1 - \frac{d}{2} \right) \log\left( 2\right)  + \frac{d}{2}\log\left( 1 + e^{- \beta} \right)\right)  \right) }{\frac{2}{d} \sqrt{4 d \left( 1 + e^{- \beta}\right)^2  e^\beta \left(2 + d \left( e^{\beta} - 1\right)  \right) }\sqrt{ 2 \frac{e^{- 2 \beta}}{8 \left( 1 + e^{- \beta}\right)^3 } d}} \\
	&=  \exp\left(O \left( \frac{1}{n}\right)\right)  \sqrt{\frac{1 + e^{ \beta}}{ 2 + d  e^{\beta} - d }}  \exp \left(n \left(  \left(1 - \frac{d}{2} \right) \log\left( 2\right)  + \frac{d}{2}\log\left( 1 + e^{- \beta} \right)\right)  \right).
\end{align*}
as claimed.
\end{proof}

\subsection{The simple d-regular case}

Having established the first moment in the pairing model $\Gbold$,  we next adapt the result to the $d$-regular model $\G$ of interest. As we will see, a pairing variant $\Gpone$ of the planted model will be a useful tool to do so. The pairing variant $\Gpone$ is defined as follows. First, draw a spin assignment $\SIGMA^* \in \left\lbrace \pm 1 \right\rbrace^n$ uniformly at random. Then, draw a graph $\Gpone$ according to the probability distribution
\begin{align*}
	\mathbb{P} \left[ \Gpone = G \vert \SIGMA^* \right] \propto \exp \left( - \beta \mathcal{H}_G \left( \SIGMA^* \right) \right).
\end{align*}
where $G$ might contain self-loops and double-edges. In the following, we will call a graph $G$ \textit{simple} if it does not feature any such self-loops or double-edges. With this definition, we are able to prove \Prop~\ref{prop_first_moment}.

\begin{proof}[Proof of \Prop~\ref{prop_first_moment}]
To get started, we note the asymptotic equality
\begin{equation}\label{first_moment_pairing_regular_connection}
	\Erw \left[ Z_{\G, \beta}   \right] \sim  \frac{\Pr \brk{\Gpone \mathrm{\  is \ simple}}}{\Pr \brk{\Gbold \textrm{ is simple}}}\Erw \left[ Z_{\Gbold, \beta}   \right].
\end{equation}
Fortunately, both $\Pr \brk{\Gpone \mathrm{\  is \ simple}}$ and $\Pr \brk{\Gbold \textrm{ is simple}}$ can be readily found in the literature.
\begin{fact}[Corollary 9.7 in \cite{Janson_2011}]\label{prob_pairing_model_simple}
	For $d \geq 3$, we have
	\begin{align*}
		\Pr \brk{\mathbf{G} \mathrm{\  is \ simple}} \sim \exp \bc{ - \frac{d -1}{2} - \frac{\bc{d - 1}^2}{4}}.
	\end{align*}
\end{fact}

\begin{lemma}[Lemma 4.6 in \cite{Coja_2020}]\label{prob_planted_pairing_model_simple}
	For $d \geq 0$ and $\beta > 0$ we have 
	\begin{equation*}
		\Pr \brk{\Gpone\mathrm{\  is \ simple}} \sim \exp \bc{- \bc{d -1} \frac{1}{1 + e^\beta} - \bc{d -1}^2 \frac{1 + e^{2\beta}}{2 \bc{1 + e^\beta}^2}}.
	\end{equation*}
\end{lemma}
In combination with \Prop~\ref{prop_first_moment_pairing_model} and equation \eqref{first_moment_pairing_regular_connection}, Fact \ref{prob_pairing_model_simple} and Lemma \ref{prob_planted_pairing_model_simple} yield the desired result.
\end{proof}

\section{The Second Moment/ Proof of Proposition \ref{prop_second_moment}}
Similar to the first moment, we will first establish the following result for the paring model $\Gbold$.

\begin{proposition}\label{prop_second_moment_pairing_model}
	For $0 < \beta < \bks$ and $d \geq 3$ we have
	\begin{equation*}
		\Erw \brk{Z_{\Gbold, \beta} \evO} = \exp \bc{O \bc{\frac{1}{n}}} \frac{ \bc{ 1 + e^{ \beta}}^2  \exp \bc{n \bc{\bc{ 2 - d}  \log \bc{2} + d \log \bc{1 + e^{- \beta}}}}}{\bc{d e^\beta - d + 2} \sqrt{2 e^{2 \beta} + 2 d e^\beta  - d e^{2 \beta} - d + 2}}.
	\end{equation*}
\end{proposition}
Once we have done so, we bridge the gap between $\Gbold$ and $\G$.

\subsection{Getting started}

For the second moment calculation, we introduce a set of variables that is similar in meaning to the ones employed in the previous sections. Yet, the definitions become more complicated since for the second moment each node $v$ in some graph $G$ is assigned two spins $\sigma_{v}$ and $\tau_{v}$ which can be either positive or negative. As before, we aim to measure the fractions of edges that connect two vertices with certain spin configurations. Since each node is equipped with two spins, there are 16 possible spin configurations for two connected vertices. Usually, we will denote such a configuration as $\left( \sigma_1, \tau_1, \sigma_2, \tau_2\right) \in \left\lbrace \pm 1 \right\rbrace^4 $ where $\sigma_1$ and $\tau_1$ denote the spins assigned to the first node. Accordingly, $\sigma_2$ and $\tau_2$ are the spins of the second node.
With this notation of spin assignments in mind, we define
\begin{align*}
	\mu_{r,s,t,u}  &:= \frac{2}{d n} \sum_{(u,v) \in E} \vecone \left\lbrace \sigma (u)=r, \sigma (v)=s, \tau (u) = t, \tau(v) = u\right\rbrace \qquad \cbc{r,s,t,u \in {\pm 1}}.  \\
\end{align*}
with the shorthand notation $\mu_{++++} = \mu_{+1, +1, +1, +1}$ and so forth.
Our choices of $\mu$ are constrained by the following relationship.
\begin{align}\label{mu_symmetry}
	\mu_{\left( \sigma_1, \tau_1, \sigma_2, \tau_2\right)} = \mu_{\left( \sigma_2, \tau_2, \sigma_1, \tau_1\right)} \hspace{5 em}\forall  \left( \sigma_1, \tau_1, \sigma_2, \tau_2\right) \in \left\lbrace \pm 1 \right\rbrace^4.
\end{align}
Note that $\mu_{++++}, \mu_{+-+-}, \mu_{-+-+}$, and $\mu_{----}$ get a special meaning: these four configurations satisfy both $\sigma (u) = \sigma (v)$ and $\tau (u) = \tau(v)$. Hence, they trivially fit condition (\ref{mu_symmetry}). All of the remaining 12 $\mu$'s can be divided into pairs which are the same up to the order of the two vertices. Since the edges are undirected, for each of these pairs we simply count all the edges that could be assigned to either of the two components of $\mu$. Then, to ensure that the $\mu$ pairs satisfy \eqref{mu_symmetry}, the count is equally split between the $\mu$ pair. Combining these thoughts yields
\begin{align*}
	&\mu_{++--} =  \mu_{--++} := \frac{1}{d n} \sum_{(u,v) \in E} \vecone \left\lbrace \sigma (u) = \tau (u) \neq \sigma(v) = \tau(v)\right\rbrace\\
	&\mu_{+--+} =  \mu_{-++-} := \frac{1}{d n} \sum_{(u,v) \in E} \vecone \left\lbrace \sigma (u) = \tau (v) \neq \tau (u) = \sigma(v)\right\rbrace \\
	&\mu_{+++-} =  \mu_{+-++} := \frac{1}{d n} \sum_{(u,v) \in E} \vecone \left\lbrace \sigma (u) = \sigma (v) = +1  \wedge \tau (u) \neq \tau(v)\right\rbrace \\
	&\mu_{++-+}=  \mu_{-+++} := \frac{1}{d n} \sum_{(u,v) \in E} \vecone\left\lbrace \sigma (u) \neq \sigma (v)   \wedge \tau (u) = \tau(v) = +1\right\rbrace \\
	&\mu_{---+} =  \mu_{-+--} := \frac{1}{d n} \sum_{(u,v) \in E} \vecone \left\lbrace \sigma (u) = \sigma (v) = -1  \wedge \tau (u) \neq \tau(v)\right\rbrace\\
	&\mu_{--+-} =  \mu_{+---} := \frac{1}{d n} \sum_{(u,v) \in E} \vecone \left\lbrace \sigma (u) \neq \sigma (v)   \wedge \tau (u) = \tau(v) = -1\right\rbrace.
\end{align*}
Finally, we need expressions to indicate which fraction of vertices is assigned a certain spin configuration $\left( \sigma_1, \tau_1 \right) \in \left\lbrace \pm 1 \right\rbrace^2$. This is achieved rather easily by defining
\begin{align*}
	\rho_{\sigma_1, \tau_1} := \frac{1}{n} \sum_{v \in V} \vecone \left\lbrace \sigma \left(v\right) = \sigma_1 , \tau \left( v\right) =  \tau_1 \right\rbrace
\end{align*}
for $\left( \sigma_1, \tau_1\right) \in \left\lbrace \pm 1 \right\rbrace^2$. With the definitions in place, we can move on to calculating the second moment.
As a starting point we choose an equation that was derived in detail in \cite{Coja_2020}.

\begin{lemma}[(4.42) in \cite{Coja_2020}] \label{second_moment_start}
We have
\begin{align}\label{eqsecmom}
	\mathbb{E}\left[ Z_{\Gbold, \beta}^2   \right] &= \sum_{\mu \in \mathcal{U}} \frac{\mathcal{X}_\mu \mathcal{Y}_\mu \mathcal{Z}_\mu}{\left( dn -1\right) !!}  \cdot \exp \left( - \beta \frac{d n}{2} \left( \sum_{\sigma \in A_1}  \mu (\sigma) + 2 \sum_{\sigma \in A_2}  \mu(\sigma) \right) \right)
\end{align}
where $\mathcal{U}$ is the set of conceivable distributions $\mu$, $\sigma$ is of the form $\sigma := (\sigma_1, \tau_1, \sigma_2, \tau_2) \in \{\pm 1\}^4$, $A_1$ is defined by
\[
A_1 := \left\lbrace x \in \{\pm 1\}^4: \sigma_1 = \sigma_2 \textrm{ and } \tau_1 \neq \tau_2 \right\rbrace \cup \left\lbrace x \in \{\pm 1\}^4: \sigma_1 \neq \sigma_2 \textrm{ and } \tau_1 = \tau_2 \right\rbrace,
\]
$A_2$ is given by
\[
A_2 := \left\lbrace x \in \{\pm 1\}^4: \sigma_1 = \sigma_2 \textrm{ and } \tau_1 = \tau_2 \right\rbrace,
\]
and
\begin{align*}
	\mathcal{X}_\mu &= \binom{n}{\rho_{++}n, \rho_{+-}n, \rho_{-+}n, \rho_{--} n },\\
	\mathcal{Y}_\mu &= \prod_{i, j \in \{ \pm\}} \binom{d n \rho_{i j}}{d n \mu_{i j ++}, d n \mu_{i j +-}, d n \mu_{i j -+}, d n \mu_{i j --}},\\
	\mathcal{Z}_\mu &= \left( d n \mu_{++--}\right)!  \left( d n \mu_{-++-}\right)!  \prod_{k \in \{ \pm\}} \left( \left( d n \mu_{+ k - k}\right)! \left( d n \mu_{k + k -}\right)!\right) \prod_{i, j \in \{ \pm\}}\left(d n \mu_{i j i j} -1 \right)!! .
\end{align*}
\end{lemma}

\subsection{Reformulation of the second moment}
The next Lemma equips us with an useful reformulation of the second moment.
\begin{lemma}\label{second_moment_simplified}
We have
\begin{align*}
	\mathbb{E}\left[ Z_{\Gbold, \beta}^2   \right] 
	= \sum_{\mu \in \mathcal{U}} 	\frac{1}{8 d^3 \pi^\frac{9}{2} n^\frac{9}{2} \sqrt{\prod_{\sigma \in B}\mu_{\sigma}}}\exp \left(n \delta \left( \mu, \rho \right) + O \left( \frac{1}{n}\right) \right) 
\end{align*}
where $\mathcal{U}$ is the set of conceivable distributions $\mu$, $\sigma$ is of the form $\sigma := (\sigma_1, \tau_1, \sigma_2, \tau_2) \in \{\pm 1\}^4$,  $\delta \left( \mu, \rho \right) $ is defined by
	\begin{align*}
		\delta \left( \mu, \rho \right)   := \textrm{H}\left( \rho\right) -\frac{d}{2} \left( D_\textrm{KL} \left( \mu \vert\vert \rho \otimes \rho\right) + \beta   \sum_{\sigma \in A_1}  \mu (\sigma) + 2 \beta  \sum_{\sigma \in A_2}  \mu(\sigma)\right)
	\end{align*}
	and
	\begin{align*}
		B := \{\pm\}^4 \setminus \left\{ (++--), (-++-), (++-+), (+++-), (+---), (-+--) \right\}.
	\end{align*}
\end{lemma}

\begin{proof}
We start off the formulation of the second moment from \Lem~\ref{second_moment_start}. In the next lines, we will establish four asymptotic equalities\footnote{The basic idea for the proof is the same as the one in \cite{Coja_2020}. The contribution of this paper is a more precise calculation that allows us to reduce the error term to order $\exp O   \left(\frac{1}{n}\right)$.}. Let us start with
\begin{align*}
	\left( dn -1\right) !! &= \frac{\left( d n\right) !}{\left( \frac{d n}{2}\right) ! 2^\frac{d n}{2}}= 2^{-\frac{d n}{2}} \cdot \sqrt{2} \left( \frac{d n}{e}\right) ^{d n} \left( \frac{d n}{2 e} \right)^{-\frac{d n}{2}} \exp \left( O \left( \frac{1}{n} \right) \right)
	= \exp \left( \frac{1}{2} \log \left( 2\right) + \frac{d n}{2} \log \left( d n\right) -  \frac{d n}{2} + O \left( \frac{1}{n} \right) \right).
\end{align*}
Second, we take a closer look at
	\begin{align*}
		\mathcal{X}_\mu &= \binom{n}{\rho_{++}n, \rho_{+-}n, \rho_{-+}n, \rho_{--} n } = \frac{n!}{\left( \rho_{++} n \right)!\left( \rho_{+-} n \right)! \left( \rho_{-+} n \right)! \left( \rho_{--} n \right)!  }\\
		&= \left(2 \pi \right)^{- \frac{3}{2}}   \left( n^3 \rho_{++} \rho_{+-} \rho_{-+} \rho_{--}\right)^{-\frac{1}{2}} \left( \frac{n}{e}\right) ^{n}  \left( \frac{\rho_{++} n}{e}\right) ^{-\rho_{++} n}\left( \frac{\rho_{+-} n}{e}\right) ^{- \rho_{+-} n}\\
		&\cdot \left( \frac{\rho_{-+} n}{e}\right) ^{- \rho_{-+} n}\left( \frac{\rho_{--} n}{e}\right) ^{- \rho_{--} n}\exp \left( O \left( \frac{1}{n} \right) \right)\\
		&=\left(2 \pi \right)^{- \frac{3}{2}}   \left( n^3 \rho_{++} \rho_{+-} \rho_{-+} \rho_{--}\right)^{-\frac{1}{2}} \underbrace{\rho_{++}^{-\rho_{++} n}\rho_{+-}^{-\rho_{+-} n}\rho_{-+}^{-\rho_{-+} n}\rho_{--}^{-\rho_{--} n}}_{= \exp \left( n \cdot \textrm{H} \left( \rho \right) \right) } \exp \left(  O \left( \frac{1}{n}  \right)\right) \\
		&= \frac{1}{\sqrt{8 \pi^3 n^3 \rho_{++} \rho_{+-} \rho_{-+} \rho_{--}}}\exp \left( n \cdot \textrm{H} \left( \rho \right) + O \left( \frac{1}{n} \right)\right).
	\end{align*}
	Moving on to the third term, we obtain
	\begin{align*}
		\mathcal{Y}_\mu &= \prod_{i, j \in \{ \pm\}} \binom{d n \rho_{i j}}{d n \mu_{i j ++}, d n \mu_{i j +-}, d n \mu_{i j -+}, d n \mu_{i j --}}\\ 
		&= \prod_{i, j \in \{ \pm\}} \left(2 \pi d n \right)^{- \frac{3}{2}} \sqrt{ \frac{\rho_{i j}}{\mu_{i j ++} \mu_{i j +-}\mu_{i j -+}\mu_{i j --}}} \\
		& \qquad \qquad \cdot \rho_{i j}^{d n \rho_{i j}}\mu_{i j ++}^{- d n \mu_{i j ++}} \mu_{i j +-}^{- d n \mu_{i j +-}} \mu_{i j -+}^{- d n \mu_{i j -+}}\mu_{i j --}^{- d n \mu_{i j --}} \exp \left( O \left( \frac{1}{n} \right) \right)\\
		&= \frac{1}{\left( 2 \pi d n\right)^6}  \left( \prod_{i, j \in \{ \pm\}}  \sqrt{ \frac{\rho_{i j}}{\mu_{i j ++} \mu_{i j +-}\mu_{i j -+}\mu_{i j --}}}\right)   \exp \left(  d n \left( \textrm{H}\left( \mu\right)-  \textrm{H}\left( \rho\right)\right) + O \left( \frac{1}{n} \right) \right)\\
	\end{align*}
	Last, we consider the fourth term
	\begin{align*}
		\mathcal{Z}_\mu &= \left( d n \mu_{++--}\right)!  \left( d n \mu_{-++-}\right)!  \prod_{k \in \{ \pm\}} \left( \left( d n \mu_{+ k - k}\right)! \left( d n \mu_{k + k -}\right)!\right) \prod_{i, j \in \{ \pm\}}\left(d n \mu_{i j i j} -1 \right)!!\\
		&= 2 \pi d n \sqrt{ \mu_{++--} \mu_{-++-}} \left( \frac{d n \mu_{++--}}{e}\right)^{d n \mu_{++--}}  \left( \frac{d n \mu_{-++-}}{e}\right)^{d n \mu_{-++-}} \cdot  \exp \left( O \left( \frac{1}{n} \right) \right)\\
		&\qquad \qquad \cdot \prod_{k \in \{ \pm\}}\left(2 \pi d n \sqrt{ \mu_{+k-k} \mu_{k+k-}} \left( \frac{d n \mu_{+k-k}}{e}\right)^{d n \mu_{+k-k}}  \left( \frac{d n \mu_{k+k-}}{e}\right)^{d n \mu_{k+k-}} \right) \cdot \prod_{i, j \in \{ \pm\}} \frac{\left( d n \mu_{i j i j}\right) !}{\left( \frac{d n \mu_{i j i j}}{2}\right)! \cdot 2^\frac{d n \mu_{i j i j}}{2}}
	\end{align*}
	In order to proceed with the fourth equation, we keep in mind
	\begin{align*}
		\frac{\left( d n \mu_{i j i j}\right) !}{\left( \frac{d n \mu_{i j i j}}{2}\right)! \cdot 2^\frac{d n \mu_{i j i j}}{2}} &= 2^{-\frac{d n \mu_{i j i j}}{2}} \cdot \sqrt{2} \left( \frac{d n \mu_{i j i j}}{e}\right)^{d n \mu_{i j i j}} \left( \frac{d n \mu_{i j i j}}{2 e}\right)^{-\frac{d n \mu_{i j i j}}{2}}  \cdot  \exp \left( O \left( \frac{1}{n} \right) \right)\\
		&= \sqrt{2} \left( \frac{d n \mu_{i j i j}}{e}\right)^{\frac{d n \mu_{i j i j}}{2}} \cdot  \exp \left( O \left( \frac{1}{n} \right) \right)
	\end{align*}
	which brings us back to
\begin{align*}
	\mathcal{Z}_\mu
	&= \left( 2 \pi d n\right)^3  \sqrt{ \mu_{++--} \mu_{-++-}} \left( \frac{d n \mu_{++--}}{e}\right)^{d n \mu_{++--}}  \left( \frac{d n \mu_{-++-}}{e}\right)^{d n \mu_{-++-}} \cdot  \exp \left( O \left( \frac{1}{n} \right) \right)\\
	&\cdot \prod_{k \in \{ \pm\}}\left( \sqrt{ \mu_{+k-k} \mu_{k+k-}} \left( \frac{d n \mu_{+k-k}}{e}\right)^{d n \mu_{+k-k}}  \left( \frac{d n \mu_{k+k-}}{e}\right)^{d n \mu_{k+k-}} \right)
	\cdot \prod_{i, j \in \{ \pm\}} \sqrt{2} \left( \frac{d n \mu_{i j i j}}{e}\right)^{\frac{d n \mu_{i j i j}}{2}}\\
	&= 4 \left( 2 \pi d n\right)^3  \sqrt{ \mu_{++--} \mu_{-++-}} \left( \frac{d n \mu_{++--}}{e}\right)^{d n \mu_{++--}}  \left( \frac{d n \mu_{-++-}}{e}\right)^{d n \mu_{-++-}} \cdot  \exp \left( O \left( \frac{1}{n} \right) \right)\\
	&\qquad \qquad \cdot  \sqrt{ \mu_{++-+} \mu_{+++-}} \left( \frac{d n \mu_{++-+}}{e}\right)^{d n \mu_{++-+}}  \left( \frac{d n \mu_{+++-}}{e}\right)^{d n \mu_{+++-}} \sqrt{ \mu_{+---} \mu_{-+--}} \\
	&\qquad \qquad \cdot \left( \frac{d n \mu_{+---}}{e}\right)^{d n \mu_{+---}}  \left( \frac{d n \mu_{-+--}}{e}\right)^{d n \mu_{-+--}} \left( \frac{d n \mu_{++++}}{e}\right)^{\frac{d n \mu_{++++}}{2}}\\
	&\qquad \qquad \cdot   \left( \frac{d n \mu_{+-+-}}{e}\right)^{\frac{d n \mu_{+-+-}}{2}} \left( \frac{d n \mu_{-+-+}}{e}\right)^{\frac{d n \mu_{-+-+}}{2}} \left( \frac{d n \mu_{----}}{e}\right)^{\frac{d n \mu_{----}}{2}}.
\end{align*}
To simplify this rather complicated term further, we recall the symmetry of our model (see \eqref{mu_symmetry}). Applying this insight to our calculation yields
\begin{align*}
	\mathcal{Z}_\mu = 2^5 \pi^3 d^3 n^3 \sqrt{\mu_{++--} \mu_{-++-}  \mu_{++-+} \mu_{+++-} \mu_{+---} \mu_{-+--}}
	\cdot \exp\left( - \frac{d n}{2}\cdot \textrm{H}\left( \mu  \right) - \frac{d n}{2} + \frac{d n}{2}\log\left( dn \right) + O \left( \frac{1}{n} \right) \right).
\end{align*}
Next, we combine these four results starting with
\begin{align*}
	\frac{\mathcal{Z}_\mu}{\left( d n - 1\right) !!}
	&= 2^{\frac{9}{2}} \pi^3 d^3 n^3 \sqrt{\mu_{++--} \mu_{-++-}  \mu_{++-+} \mu_{+++-} \mu_{+---} \mu_{-+--}} \cdot \exp\left( - \frac{d n}{2}\cdot \textrm{H}\left( \mu  \right)  + O \left( \frac{1}{n} \right) \right) .
\end{align*}
Finally, we arrive at
\begin{align*}
	\frac{\mathcal{X}_\mu \mathcal{Y}_\mu \mathcal{Z}_\mu}{\left( d n - 1\right) !!} &= \exp \left( \frac{d n}{2} \textrm{H}\left( \mu\right) - n \left( d-1\right) \textrm{H}\left( \rho\right) - \frac{9}{2} \log \left( n\right) -  \frac{9}{2} \log \left( 2 \pi \right) - 3 \log\left( d\right) + O \left( \frac{1}{n} \right)\right) \\
	&\qquad \cdot \exp \left( - \frac{1}{2}\log \left(  \rho_{++} \rho_{+-} \rho_{-+} \rho_{--}\right)  + \frac{1}{2} \sum_{i, j \in \{ \pm\}} \log \left( \frac{\rho_{i j}}{\mu_{i j ++} \mu_{i j +-}\mu_{i j -+}\mu_{i j --}}\right)\right)\\
    &\qquad \cdot \exp \left( \frac{1}{2} \log \left( \mu_{++--} \mu_{-++-}  \mu_{++-+} \mu_{+++-} \mu_{+---} \mu_{-+--}\right) + \frac{3}{2} \log \left( 2\right)\right)\\
	&= \exp \left( \frac{d n}{2}\left(  \textrm{H}\left( \mu\right) - 2 \textrm{H}\left( \rho\right)\right) + n \textrm{H}\left( \rho\right)- \frac{9}{2} \log \left( 2 \pi n\right) - 3 \log\left( d \right)+ O \left( \frac{1}{n} \right)\right) \\
	&\qquad \cdot \exp \left(-  \frac{1}{2}\log \left( \rho_{++} \rho_{+-} \rho_{-+} \rho_{--}\right)   + \frac{1}{2} \log \left( \frac{\rho_{++} \rho_{+-} \rho_{-+} \rho_{--}}{\prod_{\sigma \in \{\pm\}^4}\mu_{\sigma}}\right)\right)\\
	&\qquad \cdot \exp \left( \frac{1}{2} \log \left( \mu_{++--} \mu_{-++-}  \mu_{++-+} \mu_{+++-} \mu_{+---} \mu_{-+--}\right)+ \frac{3}{2} \log \left( 2\right)\right)\\
	&= \exp \left( -\frac{d n}{2}D_\textrm{KL} (\mu \vert\vert \rho \otimes \rho) + n \textrm{H}\left( \rho\right)- \frac{9}{2} \log \left(  \pi n\right)  - 3 \log\left( 2 d \right)  - \frac{1}{2} \log \left( \prod_{ \sigma \in \{\pm\}^4} \mu_{\sigma}\right) \right)\\
	&\qquad \cdot \exp \left( \frac{1}{2} \log \left( \mu_{++--} \mu_{-++-}  \mu_{++-+} \mu_{+++-} \mu_{+---} \mu_{-+--}\right)+ O \left( \frac{1}{n} \right)\right)
\end{align*}
To simplify this expression, we introduce the set
	\begin{align*}
		B = \{\pm\}^4 \setminus \left\{ (++--), (-++-), (++-+), (+++-), (+---), (-+--) \right\}
	\end{align*}
to write
	\begin{align*}
		\frac{\mathcal{X}_\mu \mathcal{Y}_\mu \mathcal{Z}_\mu}{\left( d n - 1\right) !!}
		&= \frac{1}{8 d^3 \pi^\frac{9}{2} n^\frac{9}{2} \sqrt{\prod_{\sigma \in B}\mu_{\sigma}}}\exp \left(n \textrm{H}\left( \rho\right) -\frac{d n}{2}D_\textrm{KL} \left( \mu \vert\vert \rho \otimes \rho\right) + O \left( \frac{1}{n}\right) \right).
	\end{align*}
	With this result , the second moment turns out to be
	\begin{align*}
		\mathbb{E}\left[ Z_{\Gbold, \beta}^2   \right] &= \sum_{\mu \in \mathcal{U}} 	\frac{\mathcal{X}_\mu \mathcal{Y}_\mu \mathcal{Z}_\mu}{\left( d n - 1\right) !!} \cdot \exp \left( - \beta \frac{d n}{2} \left( \sum_{\sigma \in A_1}  \mu (\sigma) + 2 \sum_{\sigma \in A_2}  \mu(\sigma) \right) \right) \\
		&= \sum_{\mu \in \mathcal{U}} 	\frac{1}{8 d^3 \pi^\frac{9}{2} n^\frac{9}{2} \sqrt{\prod_{\sigma \in B}\mu_{\sigma}}}\exp \left(n \delta \left( \mu, \rho \right) + O \left( \frac{1}{n}\right) \right) \label{second_moment_before_Laplace}
	\end{align*}
	where $\mathcal{U}$ is the set of conceivable distributions $\mu$, $\sigma$ is of the form $\sigma := (\sigma_1, \tau_1, \sigma_2, \tau_2) \in \{\pm 1\}^4$,  $\delta \left( \mu, \rho \right) $ is defined by
	\begin{align*}
		\delta \left( \mu, \rho \right)   := \textrm{H}\left( \rho\right) -\frac{d}{2} \left( D_\textrm{KL} \left( \mu \vert\vert \rho \otimes \rho\right) + \beta   \sum_{\sigma \in A_1}  \mu (\sigma) + 2 \beta  \sum_{\sigma \in A_2}  \mu(\sigma)\right)
	\end{align*}
	$A_1$ is defined by
	\[
	A_1 := \left\lbrace x \in \{\pm 1\}^4: \sigma_1 = \sigma_2 \textrm{ and } \tau_1 \neq \tau_2 \right\rbrace \cup \left\lbrace x \in \{\pm 1\}^4: \sigma_1 \neq \sigma_2 \textrm{ and } \tau_1 = \tau_2 \right\rbrace,
	\]
	and $A_2$ is given by
	\[
	A_2 := \left\lbrace x \in \{\pm 1\}^4: \sigma_1 = \sigma_2 \textrm{ and } \tau_1 = \tau_2 \right\rbrace.
	\]
\end{proof}

Our ultimate goal is to apply the Laplace method. In order to keep things manageable, we will substitute certain variables using basic symmetry and composition arguments. First, we note that $\rho$ can be simply obtained by calculating the marginals of $\mu$, that is
\begin{align*}
	\rho_{++} &= \mu_{++++} + \mu_{+++-} + \mu_{++-+} + \mu_{++--} \\
	\rho_{+-} &= \mu_{+-++} + \mu_{+-+-} + \mu_{+--+} + \mu_{+---} \\
	\rho_{-+} &= \mu_{-+++} + \mu_{-++-} + \mu_{-+-+} + \mu_{-+--} \\
	\rho_{--} &= \mu_{--++} + \mu_{--+-} + \mu_{---+} + \mu_{----}.
\end{align*}
By construction we know that
\begin{align*}
	\mu_{+--+} &= \mu_{-++-} , \qquad \mu_{++--} =\mu_{--++}, \qquad \mu_{+++-} = \mu_{+-++}\\
	\mu_{++-+} &= \mu_{-+++}, \qquad	\mu_{+---} = \mu_{--+-}, \qquad \mu_{-+--} = \mu_{---+}\\
	\mu_{----} &= 1 - \sum_{\sigma \in \left\lbrace \pm 1 \right\rbrace^4, \sigma \neq \left( -1, -1, -1, -1 \right)  } \mu_\sigma
\end{align*}
Bringing these results together we are left with $9$ variables, which we rename in the following order for notational convenience
\begin{align*}
	x_1 &:= \mu_{+--+} = \mu_{-++-}, \qquad  x_2 := \mu_{++--}=\mu_{--++}, \qquad	x_3 := \mu_{+++-} = \mu_{+-++},\\
	x_4 &:= \mu_{++-+} = \mu_{-+++}, \qquad x_5 := \mu_{+---} = \mu_{--+-}, \qquad x_6 := \mu_{-+--} = \mu_{---+}\\
	x_7 &:= \mu_{+-+-}, \qquad x_8 := \mu_{-+-+}, \qquad x_9 := \mu_{++++}
\end{align*}
which implies
\begin{align*}
	\mu_{----} &= 1 - \sum_{\sigma \in \left\lbrace \pm 1 \right\rbrace^4, \sigma \neq \left( -1, -1, -1, -1 \right)  } \mu_\sigma   = 1 - 2 x_1 - 2 x_2 - 2 x_3 - 2 x_4 - 2 x_5 -2 x_6 -x_7 - x_8 - x_9\\
	\rho_{++} &= x_9 + x_3 +  x_4 + x_2 = x_2 + x_3 + x_4 + x_9\\
	\rho_{+-} &= x_3 + x_7 + x_1 + x_5= x_1 + x_3 + x_5 + x_7  \\
	\rho_{-+} &=  x_4+ x_1 + x_8 + x_6 = x_1  + x_4 + x_6 + x_8\\
	\rho_{--} &= x_2 + x_5 + x_6 + \mu_{----}= 1 - 2 x_1- x_2 - 2 x_3 - 2  x_4 - x_5 - x_6 -x_7 - x_8 - x_9.
\end{align*}
In order to apply the Laplace method to the second moment, let us consider the function
\begin{align*}
	\delta \left( \mu, \rho \right)   = \textrm{H}\left( \rho\right) -\frac{d}{2} \left( D_\textrm{KL} \left( \mu \vert\vert \rho \otimes \rho\right) + \beta   \sum_{\sigma \in A_1}  \mu (\sigma) + 2 \beta  \sum_{\sigma \in A_2}  \mu(\sigma)\right) .
\end{align*}
We continue by reformulating terms
\begin{align*}
	D_\textrm{KL} \left( \mu \vert\vert \rho \otimes \rho\right)
	&= 2 x_1 \log \left( \frac{x_1}{\rho_{+-} \rho_{-+}} \right) + 2 x_2 \log \left( \frac{x_2}{\rho_{++} \rho_{--}} \right) + 2 x_3 \log \left( \frac{x_3}{\rho_{++} \rho_{+-}}\right) + 2 x_4 \log \left(\frac{x_4}{\rho_{++} \rho_{-+}}\right) \\
	&\qquad + 2 x_5 \log \left(\frac{x_5}{\rho_{+-} \rho_{--}} \right)
	+ 2 x_6 \log \left(\frac{x_6}{\rho_{-+} \rho_{--}} \right) + x_7 \log \left(\frac{x_7}{\rho_{+-}^2} \right) + x_8 \log \left( \frac{x_8}{\rho_{-+}^2} \right)\\
	&\qquad + x_9 \log \left( \frac{x_9 }{\rho_{++}^2}\right)  + \mu_{----} \log\left( \frac{ \mu_{----}}{\rho_{--}^2}\right) \\
	&= 2 x_1 \log \left( x_1\right) + 2 x_2 \log \left( x_2\right) + 2 x_3 \log \left( x_3\right) + 2 x_4 \log \left(x_4 \right) + 2 x_5 \log \left( x_5\right)+ 2 x_6 \log \left( x_6\right)   \\
	&\qquad + x_7 \log \left( x_7\right) + x_8 \log \left( x_8\right)+ x_9 \log \left( x_9\right) + \mu_{----} \log\left( \mu_{----}\right) - 2 \rho_{++}  \log \left( \rho_{++}\right)\\
	&\qquad - 2 \rho_{+-}  \log \left( \rho_{+-} \right) - 2 \rho_{-+} \log \left( \rho_{-+} \right) -2 \rho_{--} \log \left( \rho_{--}\right).
\end{align*}
As a consequence, we obtain
\begin{align*}
	\delta \left( \mu, \rho \right)   &= \textrm{H}\left( \rho\right) -\frac{d}{2} \left( D_\textrm{KL} \left( \mu \vert\vert \rho \otimes \rho\right) + \beta   \sum_{\sigma \in A_1}  \mu (\sigma) + 2 \beta  \sum_{\sigma \in A_2}  \mu(\sigma)\right) \\
	&= \left(  1 - d\right) \textrm{H}\left( \rho\right) + \frac{d}{2} \textrm{H}\left( \mu \right) - d \beta  \left( x_3 + x_4 + x_5 + x_6 +  x_7 + x_8 + x_9 + \mu_{----} \right)
\end{align*}

\subsection{Application of  the Laplace method to the second moment}

To apply the Laplace method we need to determine the maximum of $\delta \left( \mu, \rho \right) $. This is achieved with the following Lemma. Due to its technical and tedious nature, the proof of the lemma is outsourced to a separate section (see section \ref{section_maximization_sec_mom}).
\begin{lemma}\label{sec_mom_optimum}
For $0<\beta<\bks$, we have
\begin{align*}
	\max_{\mu \in \mathcal{O}'} \delta\left( \mu, \rho\right)=\delta\left( \mu^*, \rho^*\right)  =\left( 2 - d\right)  \log \left( 2\right) + d \log \left(1 + e^{- \beta}  \right)
	\end{align*}
where $\mathcal{O}'$ denotes the set of all $\mu$ that are conceivable under the assumption that the event $\mathcal{O}$ occurs. The unique maximum is obtained at
	\begin{align*}
		\mu_{++++}^* &= \mu_{----}^* = \mu_{+-+-}^* = \mu_{-+-+}^* = \frac{e^{-2 \beta}}{4 \left( 1 + e^{- \beta}\right)^2}\\
		\mu_{+--+}^* &= \mu_{--++}^* = \mu_{-++-}^* = \mu_{++--}^* = \frac{1}{4 \left( 1 + e^{- \beta}\right)^2}\\
		\mu_{+++-}^* &= \mu_{++-+}^* = \mu_{+-++}^* = \mu_{-+++}^* = \mu_{---+}^*= \mu_{--+-}^*= \mu_{-+--}^*= \mu_{+---}^* = \frac{e^{- \beta}}{4 \left( 1 + e^{- \beta}\right)^2}
	\end{align*}
which also implies
	\begin{align*}
		\rho_{++}^* = \rho_{+-}^* = \rho_{-+}^* = \rho_{--}^* = \frac{1}{4}.
	\end{align*}
\end{lemma}

Having determined the maximum, we next need to evaluate the Hessian at the optimal point. The derivation of the Hessian matrix and evaluation at the optimal point is not too difficult. Thus, we just state the result here and refer the interested reader to \Sec~\ref{sec_hessian}.

\begin{lemma}[Hessian for the second moment]\label{hessian_second_moment}
	We have
	\begin{align*}
		\det \left(-  \mathrm{D}^2 \delta \left( \mu^*, \rho^* \right) \right) = 2^{17} d^6 e^{- 8 \beta} \left( 1 + e^\beta \right)^{16} \left( d e^\beta - d + 2 \right)^2 \left(2 e^{2 \beta} + 2 d e^\beta  - d e^{2 \beta} - d + 2 \right).
	\end{align*}
\end{lemma}

\begin{proof}[Proof of \Prop~\ref{prop_second_moment_pairing_model}]
With \Lem s~\ref{sec_mom_optimum} and \ref{hessian_second_moment} in place, we still need to determine the lattice matrix and its determinant. Similar to the notation for the first moment, we let $A_{\textrm{second}}$ denote the matrix consisting of the elements in the basis of the lattice for the second moment. Recalling the following definitions
\begin{align*}
	x_1 &= \mu_{+--+} =  \mu_{-++-} = \frac{1}{d n} \sum_{(u,v) \in E} \vecone \left\lbrace \sigma (u) = \tau (v) \neq \tau (u) = \sigma(v)\right\rbrace  \\
	x_3 &= \mu_{+++-} =  \mu_{+-++} = \frac{1}{d n} \sum_{(u,v) \in E} \vecone \left\lbrace \sigma (u) = \sigma (v) = +1  \wedge \tau (u) \neq \tau(v)\right\rbrace \\
	x_4 &= \mu_{++-+} =  \mu_{-+++} = \frac{1}{d n} \sum_{(u,v) \in E} \vecone\left\lbrace \sigma (u) \neq \sigma (v)   \wedge \tau (u) = \tau(v) = +1\right\rbrace  \\
	x_7 &= \mu_{+-+-}  = \frac{2}{d n} \sum_{(u,v) \in E} \vecone \left\lbrace \sigma (u) = \sigma (v)=+1   \wedge \tau (u) = \tau(v) = -1\right\rbrace\\
	x_8 &= \mu_{-+-+} = \frac{2}{d n} \sum_{(u,v) \in E} \vecone\left\lbrace \sigma (u) = \sigma (v)=-1   \wedge \tau (u) = \tau(v) = +1\right\rbrace\\
	x_9 &= \mu_{++++} = \frac{2}{d n} \sum_{(u,v) \in E} \vecone \left\lbrace \sigma (u) = \sigma (v)= \tau (u) = \tau(v) = +1\right\rbrace
\end{align*}
we immediately obtain the diagonal entries for the respective $x$'s, i.e $\frac{1}{d}$ and $\frac{2}{d}$. From here on, things get more complicated. Since $\rho_{++}, \rho_{+-}, \rho_{-+},$ and $\rho_{--}$ each count fractions of the set of nodes (which contains $n$ nodes in total) their entries in the lattice matrix all have to be multiples of $\frac{1}{n}$. Furthermore, we recall the following binding conditions
\begin{align*}
	\rho_{++} &= x_9 + x_3 +  x_4 + x_2 = x_2 + x_3 + x_4 + x_9\\
	\rho_{+-} &= x_3 + x_7 + x_1 + x_5= x_1 + x_3 + x_5 + x_7  \\
	\rho_{-+} &=  x_4+ x_1 + x_8 + x_6 = x_1  + x_4 + x_6 + x_8.
\end{align*}
Combining these two points, $x_2, x_5,$ and $x_6$ each need to be chosen such that the sums consisting of four summands each add up to a number that is a multiple of $\frac{1}{n}$. Let us focus on $x_2$. Similar arguments apply to $x_5$ and $x_6$. For $x_2$, the above equation can be reformulated as
\begin{align*}
	x_2 = \rho_{++} -  x_3 -  x_4 -x_9 = \frac{b_2}{n} -  \frac{b_3}{d n} - \frac{b_4}{d n} - \frac{2 \cdot b_9}{d n}
\end{align*}
where $b_i \in \mathbb{N}, i \in [9]$ are the scalars for the linear combination yielding the desired $\mu$. From the reformulated equation we immediately obtain the matrix entries $A_{\textrm{second}, 2,2} =1$, $A_{\textrm{second}, 2,3} = - 1/d$, $A_{\textrm{second}, 2,4} = -1/d$, and $A_{\textrm{second}, 2,9} = -2/d$. Following through this procedure for $x_5$ and $x_6$, we obtain the remaining entries of $A_{\textrm{second}}$ that are different from zero. This enables us to calculate the determinant of interest:
\begin{align} \label{lattice_second_moment}
	\det \left( A_{\textrm{second}} \right)= \det  \begin{pmatrix} 
		\frac{1}{d} & 0 & 0& 0 & 0 & 0 & 0 & 0 & 0 \\ 
		0 & 1& - \frac{1}{d} &- \frac{1}{d}& 0 & 0 & 0 & 0 & - \frac{2}{d}  \\ 
		0 & 0 & \frac{1}{d} & 0 & 0 & 0& 0& 0 &0  \\
		0 & 0 & 0 & \frac{1}{ d} & 0 &0&0&0 &0  \\
		- \frac{1}{d}& 0 &- \frac{1}{d}& 0& 1& 0&- \frac{2}{d}& 0 &0  \\
		- \frac{1}{d} & 0 & 0 & - \frac{1}{d} & 0 & 1& 0& - \frac{2}{d}&0  \\
		0 & 0 & 0 & 0 &0 & 0& \frac{2}{d}& 0 &0  \\
		0 & 0& 0 &0& 0 & 0& 0& \frac{2}{d} &0  \\
		0 &0 & 0 & 0 & 0 & 0& 0& 0 &\frac{2}{d}  \\
	\end{pmatrix}  = \frac{2^3}{d^6}.
\end{align}

With these results in place, we can apply \Thm~\ref{thm_janson} to the expression for the second moment in \Lem~\ref{second_moment_simplified} which yields
	\begin{align*}
		\mathbb{E}\left[ Z_{\Gbold, \beta}^2 \vecone { \{\mathcal{O} \}}  \right] 
		&= \exp \left( O \left( \frac{1}{n}\right)\right) \cdot  \sum_{\mu \in \mathcal{U}} 	\frac{1}{8 d^3 \pi^\frac{9}{2} n^\frac{9}{2} \sqrt{\prod_{\sigma \in B}\mu_{\sigma}}}\exp \left(n \delta \left( \mu, \rho \right)  \right) \\
		&= \exp \left( O \left( \frac{1}{n}\right)\right) \cdot  	\frac{\left( 2 \pi n \right)^{\frac{9}{2}}  \exp \left(n \delta \left( \mu^*, \rho^* \right)  \right)}{8 d^3 \pi^\frac{9}{2} n^\frac{9}{2} \sqrt{\prod_{\sigma \in B}\mu^*_{\sigma}} \det \left( A_{\textrm{second}} \right)  \sqrt{	\det \left(-  \mathrm{D}^2 \delta \left( \mu^*, \rho^* \right) \right)}}.
	\end{align*}
Using \Lem s~\ref{sec_mom_optimum} and \ref{hessian_second_moment} and the determinant of the lattice matrix from \eqref{lattice_second_moment}, we arrive at
	\begin{align*}
		\mathbb{E}\left[ Z_{\Gbold, \beta}^2  \vecone { \{\mathcal{O} \}} \right] 
		&= \exp \left( O \left( \frac{1}{n}\right)\right) \cdot  	\frac{ \left( 1 + e^{- \beta}\right)^{10}  \exp \left(n \left( \left( 2 - d\right)  \log \left( 2\right) + d \log \left(1 + e^{- \beta}  \right) \right)   \right)}{  e^{- 10 \beta} 	\left( 1 + e^\beta \right)^{8}  \left( d e^\beta - d + 2 \right) \sqrt{\left(2 e^{2 \beta} + 2 d e^\beta  - d e^{2 \beta} - d + 2 \right)}}\\
		&= \exp \left( O \left( \frac{1}{n}\right)\right) \cdot  	\frac{ \left( 1 + e^{ \beta}\right)^{2}  \exp \left(n \left( \left( 2 - d\right)  \log \left( 2\right) + d \log \left(1 + e^{- \beta}  \right) \right)   \right)}{  \left( d e^\beta - d + 2 \right) \sqrt{\left(2 e^{2 \beta} + 2 d e^\beta  - d e^{2 \beta} - d + 2 \right)}}.
	\end{align*}
\end{proof}

\subsection{The simple d-regular case}

Having established the second moment in the pairing model $\Gbold$,  we still have to adapt the result to the $d$-regular model $\G$ of interest. As we will see, a pairing variant (not the same as for the first moment) $\Gptwo$ of the planted model will be a useful tool to do so. The pairing variant $\Gptwo$ is defined as follows. First, draw two  spin assignments $\SIGMA^*, \TAU^* \in \left\lbrace \pm 1 \right\rbrace^n$ independently and uniformly at random. Then, draw a graph $\Gptwo$ according to the probability distribution
\begin{align*}
	\mathbb{P} \brk{ \Gptwo = G \vert \SIGMA^*, \TAU^* } \propto \exp \bc{ - \beta \mathcal{H}_G \bc{ \SIGMA^*} - \beta \mathcal{H}_G \bc{ \TAU^*} }.
\end{align*}
where $G$ might again feature self-loops and double-edges. With some effort, we obtain the next result.

\begin{lemma} \label{prob_sec_pairing_model_simple}
	For $d \geq 0$ and $\beta > 0$ we have 
	\begin{equation*}
		\Pr \brk{\Gptwo \mathrm{\ is \ simple}} \sim \exp \bc{- \bc{d -1} \frac{2}{\bc{1 + e^\beta}^2} - \bc{d -1}^2 \frac{\bc{1 + e^{2 \beta}}^2}{\bc{1 + e^\beta}^4}}.
	\end{equation*}
\end{lemma}
\begin{proof}[Proof of Lemma \ref{prob_sec_pairing_model_simple}]
	This proof is based on an idea in \cite{Coja_2020} (Lemma 4.6). First of all, we are interested in the number of self-loops $X$ in $\Gptwo$ on the one hand, and the number of double edges $Y$ on the other hand. For notational convenience, we let $\mathcal{G} \bc{\sigma, \mu}$ be the event that the generated graph has $\frac{d n}{2} \mu_{++++}$ edges that connect two vertices that each have been assigned two positive spins; the same is assumed to hold for all entries of $\mu$ and the respective types of edges. With these definitions in place, we move on to the expectations of $X$ and $Y$. Instead of calculating the two directly, we decompose the two to simplify the following calculations.
	
	So let us start with the number of self-loops $X$. Basically, there are four different types of self-loops in our model, $X_{++},X_{+-},X_{-+},$ and $X_{--}$. The index in each of the four cases just refers to the spin pair assigned to the vertex of the self-loop. Then, the expectation of $X_{++}$ can be formulated as
	\begin{align*}
		\Erw \brk{X_{++} \vert \mathcal{G} \bc{\sigma, \mu}} &= \frac{\rho_{++} n \binom{d}{2} \binom{d n \rho_{++} -2 }{d n \mu_{++++} -2} \bc{d n \mu_{++++} -3}!!}{\binom{d n \rho_{++}}{d n \mu_{++++}} \bc{d n \mu_{++++} -1}!!}
		= \frac{\rho_{++} n \frac{d !}{2 (d -2)!} \frac{\bc{d n \rho_{++} -2 }!}{\bc{d n \mu_{++++} -2}! \bc{d n \rho_{++} - d n \mu_{++++}}!}}{ \frac{\bc{d n \rho_{++} }!}{\bc{d n \mu_{++++}}! \bc{d n \rho_{++} - d n \mu_{++++}}!} \bc{d n \mu_{++++} -1} }\\
		&= \frac{n \mu_{++++} d (d-1)}{2 \bc{d n \rho_{++} - 1}} \sim \frac{\mu_{++++} (d -1)}{2 \rho_{++}}
	\end{align*}
	where, in the first step, we already cancelled out the factors that appeared both in the numerator and denominator. By almost identical calculations, we obtain
	\begin{align*}
		\Erw \brk{X_{+-} \vert \mathcal{G} \bc{\sigma, \mu}} &\sim \frac{\mu_{+-+-} (d -1)}{2 \rho_{+-}}, \qquad
		\Erw \brk{X_{-+} \vert \mathcal{G} \bc{\sigma, \mu}} \sim \frac{\mu_{-+-+} (d -1)}{2 \rho_{-+}}, \\
		&\text{and} \qquad \Erw \brk{X_{--} \vert \mathcal{G} \bc{\sigma, \mu}} \sim \frac{\mu_{----} (d -1)}{2 \rho_{--}}.
	\end{align*}
	Bringing these four results together and plugging in the optimal point $\bc{\mu^*, \rho^*}$, we arrive at
	\begin{align}\label{expectation_self_loops}
		\Erw \brk{X \vert \mathcal{G} \bc{\sigma, \mu}} \sim \bc{d -1} \frac{2}{\bc{1 + e^\beta}^2}.
	\end{align}
	
	With a similar argument, we determine the expectation of the number of double edges $Y$. More precisely, we decompose $Y$ into the random variables $Y_{\sigma_1, \tau_1, \sigma_2, \tau_2}$ with $\bc{\sigma_1, \tau_1, \sigma_2, \tau_2} \in \cbc{\pm 1}^4$. Each $Y_{\sigma_1, \tau_1, \sigma_2, \tau_2}$ is just the number of double edges between two vertices where the first vertex is assigned to the spin-pair $\bc{\sigma_1, \tau_1}$ and the second to the pair $\bc{\sigma_2, \tau_2}$. Let us start with the four spin configurations with $\bc{\sigma_1, \tau_1} = \bc{\sigma_2, \tau_2}$. In order to keep the calculations simple, we focus on $Y_{++++}$ and then extend the results to $Y_{+-+-}$, $Y_{-+-+}$, and $Y_{----}$.
	\begin{align*}
		\Erw \brk{Y_{++++} \vert \mathcal{G} \bc{\sigma, \mu}} &= \frac{2 \binom{\rho_{++} n}{2} \binom{d}{2}^2 \binom{d n \rho_{++} -4}{d n \mu_{++++} -4} \bc{d n \mu_{++++} -5}!!}{\binom{d n \rho_{++}}{d n \mu_{++++}} \bc{d n \mu_{++++} -1}!!}\\
		&\sim \frac{\bc{ \rho_{++} n}^2 \bc{\frac{d (d-1)}{2}}^2 \bc{d n \mu_{++++}}^4}{\bc{ \rho_{++} d n}^4 \bc{d n \mu_{++++}}^2} = \frac{(d-1)^2}{4} \frac{\mu_{++++}^2}{\rho_{++}^2}
	\end{align*}
	where, in the first step, we already cancelled out the factors that occured both in the numerator and denominator. Following this line of thought, we can also state
	\begin{align*}
		\Erw \brk{Y_{+-+-} \vert \mathcal{G} \bc{\sigma, \mu}} &\sim \frac{(d-1)^2}{4} \frac{\mu_{+-+-}^2}{\rho_{+-}^2}, \qquad
		\Erw \brk{Y_{-+-+} \vert \mathcal{G} \bc{\sigma, \mu}} \sim \frac{(d-1)^2}{4} \frac{\mu_{-+-+}^2}{\rho_{-+}^2}, \\
		&\text{and} \qquad \Erw \brk{Y_{----} \vert \mathcal{G} \bc{\sigma, \mu}} \sim \frac{(d-1)^2}{4} \frac{\mu_{----}^2}{\rho_{--}^2}.
	\end{align*}
	
	For the next calculation, we consider the sum of $Y_{+---}$ and $Y_{--+-}$. Since the edges in our model are undirected, it is not suitable to make a distinction between the two.
	\begin{align*}
		\Erw \brk{Y_{+---} + Y_{--+-} \vert \mathcal{G} \bc{\sigma, \mu}} &= \frac{2 \rho_{+-} \rho_{--} n^2 \binom{d}{2}^2 \binom{d n \rho_{+-} -2}{d n \mu_{+---} -2} \binom{d n \rho_{--} -2}{d n \mu_{+---} -2} \bc{d n \mu_{+---} -2}!}{\binom{d n \rho_{+-}}{d n \mu_{+---}} \binom{d n \rho_{--}}{d n \mu_{+---}}\bc{d n \mu_{+---}}!}\\
		&\sim \frac{2 \rho_{+-} \rho_{--} n^2 \bc{\frac{d (d-1)}{2}}^2 \bc{d n \mu_{+---}}^4}{\bc{ \rho_{+-} d n}^2 \bc{ \rho_{--} d n}^2 \bc{d n \mu_{+---}}^2}
		= \frac{(d-1)^2}{4} \frac{\mu_{+---}^2}{\rho_{+-} \rho_{--}}
	\end{align*}
	Here, we once again tacitly cancelled out the factors in the first expression that are included in both the numerator and denominator. The same approach can be iteratively applied to the remaining types of double edges, which eventually yields
	\begin{align*}
		\Erw \brk{Y_{-+--} + Y_{---+} \vert \mathcal{G} \bc{\sigma, \mu}} &\sim \frac{(d-1)^2}{4} \frac{\mu_{-+--}^2}{\rho_{-+} \rho_{--}}, \qquad
		\Erw \brk{Y_{+++-} + Y_{+-++} \vert \mathcal{G} \bc{\sigma, \mu}} \sim \frac{(d-1)^2}{4} \frac{\mu_{+++-}^2}{\rho_{++} \rho_{+-}}, \\
		\Erw \brk{Y_{-+++} + Y_{++-+} \vert \mathcal{G} \bc{\sigma, \mu}} &\sim \frac{(d-1)^2}{4} \frac{\mu_{++-+}^2}{\rho_{++} \rho_{-+}}, \qquad
		\Erw \brk{Y_{++--} + Y_{--++} \vert \mathcal{G} \bc{\sigma, \mu}} \sim \frac{(d-1)^2}{4} \frac{\mu_{++--}^2}{\rho_{++} \rho_{--}}, \\
		&\text{and} \qquad \Erw \brk{Y_{+--+} + Y_{-++-} \vert \mathcal{G} \bc{\sigma, \mu}} \sim \frac{(d-1)^2}{4} \frac{\mu_{+--+}^2}{\rho_{+-} \rho_{-+}}.
	\end{align*}
	Taking the sum of all these findings and plugging in the optimal point $\bc{\mu^*, \rho^*}$, we finally obtain
	\begin{align}\label{expectation_double_edges}
		\Erw \brk{Y \vert \mathcal{G} \bc{\sigma, \mu}} \sim \bc{d -1}^2 \frac{\bc{1 + e^{2 \beta}}^2}{\bc{1 + e^\beta}^4}.
	\end{align}
	With the statements \eqref{expectation_self_loops} and \eqref{expectation_double_edges} in mind, we claim that for all $k, \ell \geq  1$
	\begin{align}\label{expectation_loop_self_edge_combined}
		\Erw \brk{ \prod_{i = 1}^k \bc{X -i + 1}  \prod_{j = 1}^\ell \bc{Y - j + 1}} 
		\sim  \bc{\bc{d -1} \frac{2}{\bc{1 + e^\beta}^2}}^k \bc{\bc{d -1}^2 \frac{\bc{1 + e^{2 \beta}}^2}{\bc{1 + e^\beta}^4}}^\ell
	\end{align}
	holds. This can be seen as follows. In \eqref{expectation_self_loops} and \eqref{expectation_double_edges}, we placed just one loop or double edge, respectively. To obtain \eqref{expectation_loop_self_edge_combined}, we now have to place some fixed numbers $k$ and $\ell$ of self-loops and double edges. Since $n$ approaches infinity, the probability that any choices of self-loops and double-edges overlap is bounded by $O (1/n)$. Thus, the desired result can be leveraged from \eqref{expectation_self_loops} and \eqref{expectation_double_edges}.
	
	With  \eqref{expectation_loop_self_edge_combined} in place, we immediately see
	\begin{equation*}
		\Pr \brk{\Gptwo \in \mathcal{S}} = \Pr \brk{X =Y =0} \sim \exp \bc{- \bc{d -1} \frac{2}{\bc{1 + e^\beta}^2} - \bc{d -1}^2 \frac{\bc{1 + e^{2 \beta}}^2}{\bc{1 + e^\beta}^4}}.
	\end{equation*}
	which concludes the proof.
\end{proof}
Now, we are equipped to prove Proposition \ref{prop_second_moment}.

\begin{proof}[Proof of Proposition \ref{prop_second_moment}]
We again use the asymptotic equality
	\begin{equation}\label{second_moment_pairing_regular_connection}
		\Erw \left[ Z_{\G, \beta}^2  \vecone { \{\mathcal{O} \}}   \right] \sim  \frac{\Pr \brk{\Gptwo \mathrm{\  is \ simple}}}{\Pr \brk{\Gbold \textrm{ is simple}}}\Erw \left[ Z_{\Gbold, \beta}^2  \vecone { \{\mathcal{O} \}}   \right].
	\end{equation}
Thus, the desired result is obtained by combining equation \eqref{second_moment_pairing_regular_connection}, \Prop~\ref{prop_second_moment_pairing_model}, Lemma \ref{prob_sec_pairing_model_simple}, and Fact \ref{prob_pairing_model_simple}.
\end{proof}

\section{Second Moment Optimization / Proof of \Lem~\ref{sec_mom_optimum}}\label{section_maximization_sec_mom}

In this section we solve the maximization problem 
\begin{align*}
	\max_{\mu \in \mathcal{O}'} \delta\left( \mu, \rho\right)  
\end{align*}
where
\begin{align*}
	\delta \left( \mu, \rho \right)   &:= \textrm{H}\left( \rho\right) -\frac{d}{2} \left( D_\textrm{KL} \left( \mu \vert\vert \rho \otimes \rho\right) + \beta   \sum_{\sigma \in A_1}  \mu (\sigma) + 2 \beta  \sum_{\sigma \in A_2}  \mu(\sigma)\right) \qquad \text{and} \\
    A_1 &:= \left\lbrace x \in \{\pm 1\}^4: \sigma_1 = \sigma_2 \textrm{ and } \tau_1 \neq \tau_2 \right\rbrace \cup \left\lbrace x \in \{\pm 1\}^4: \sigma_1 \neq \sigma_2 \textrm{ and } \tau_1 = \tau_2 \right\rbrace \qquad \text{and} \\
    A_2 &:= \left\lbrace x \in \{\pm 1\}^4: \sigma_1 = \sigma_2 \textrm{ and } \tau_1 = \tau_2 \right\rbrace.
\end{align*}
with vectors of the form $\sigma = (\sigma_1, \tau_1, \sigma_2, \tau_2) \in \{\pm 1\}^4$. Furthermore, $\mathcal{O}'$ denotes the set of all $\mu$ that are conceivable given that the event $\mathcal{O}$ from \eqref{eq_def_O} holds. At this point, we exploit the spatial mixing argument. Keeping Lemma \ref{lem_o} in mind, we limit our attention to the event $\mathcal{O}$. This is a crucial step for the following calculations because it allows us to perform the reparametrization 
\begin{align} \label{eq_rho_alpha}
    \rho_\alpha (\sigma_1, \tau_1) := \frac{1 + \alpha \cdot  \vecone \cbc{\sigma_1 = \tau_1} - \alpha \cdot  \vecone \cbc{\sigma_1 \neq \tau_1}}{4} 
\end{align}
where $\alpha \in (-1, +1)$ and $(\sigma_1, \tau_1) \in \{\pm 1\}^2$.
Now, the proof strategy is as follows. First, we minimize $\delta$ with respect to $\mu$. This will provide us with a solution of $\mu$ formulated in terms of $\rho$ or $\alpha$, respectively. In a second step, all that remains to do is to maximize the function $\delta$ with respect to $\alpha$.

\subsection{Minimization with respect to $\mu$}
Instead of solving the optimization in one step, we start by considering
\begin{align*}
g(\mu, \rho_\alpha):= D_\textrm{KL} (\mu \vert\vert \rho_\alpha \otimes \rho_\alpha) + \beta\cdot \sum_{\sigma \in A_1}  \mu (\sigma) + 2 \cdot \beta \cdot \sum_{\sigma \in A_2}  \mu(\sigma)    
\end{align*}
where $A_1$ is defined by
\begin{align*}
A_1 := \left\lbrace x \in \{\pm 1\}^4: \sigma_1 = \sigma_2 \textrm{ and } \tau_1 \neq \tau_2 \right\rbrace \cup \left\lbrace x \in \{\pm 1\}^4: \sigma_1 \neq \sigma_2 \textrm{ and } \tau_1 = \tau_2 \right\rbrace    
\end{align*}
and $A_2$ is given by
\begin{align*}
A_2 := \left\lbrace x \in \{\pm 1\}^4: \sigma_1 = \sigma_2 \textrm{ and } \tau_1 = \tau_2 \right\rbrace.    
\end{align*}
Note that the entropy term $H(\rho)$ is independent of $\mu$ and thus not relevant for optimizing with respect to $\mu$.
The above formulation immediately brings us to the constrained  minimization problem
\begin{align*}
	&\min_{\mu \in \mathcal{O}'} g(\mu,  \rho_\alpha)\\
	\textrm{ s.t. } \forall \left( \sigma_1, \tau_1 \right) \in  &\{\pm 1\}^2: \sum_{\left( \sigma_2, \tau_2 \right) \in  \{\pm 1\}^2} \mu (\sigma_1, \tau_1, \sigma_2, \tau_2) = \rho_\alpha \left( \sigma_1, \tau_1 \right)  \\
	\forall \left( \sigma_2, \tau_2 \right) \in  &\{\pm 1\}^2: \sum_{\left( \sigma_1, \tau_1 \right) \in  \{\pm 1\}^2} \mu (\sigma_1, \tau_1, \sigma_2, \tau_2) = \rho_\alpha \left( \sigma_2, \tau_2 \right)
\end{align*}
where $\mathcal{P}\left( \{\pm 1\}^4  \right)$ denotes the set of all probability distributions on $\{\pm 1\}^4$.  For ease of notation, we will drop the index $\alpha$ and just write $\rho$. As a first step, we point out that due to symmetry the optimal $\mu^*$ will have the following properties:
\begin{align*}
	\mu^*_{++++} &= \mu^*_{----}, \qquad \mu^*_{++--} = \mu^*_{--++}, \\
	\mu^*_{+-+-} &= \mu^*_{-+-+}, \qquad \mu^*_{+--+} = \mu^*_{-++-},\\
	\mu^*_{+++-} = \mu^*_{++-+}&= \mu^*_{+-++} = \mu^*_{-+++} =	\mu^*_{---+} = \mu^*_{--+-}= \mu^*_{-+--} = \mu^*_{+---}.
\end{align*}
From the above reparametrization, we additionally emphasize that both
\begin{align*}
\rho_{++} = \rho_{--}  \hspace{3 em}  \textrm{ and } \hspace{3 em} \rho_{+-} = \rho_{-+}    
\end{align*}
hold irrespective of the chosen $\alpha$. This fact directly entails that setting up the Lagrangian function for our minimization problem will only require two distinct Lagrangian multipliers, namely $\lambda_{++}$ and $\lambda_{+-}$. Put differently, we are going to consider the following Lagrangian function $\mathcal{L}$:
\begin{align*}
	\mathcal{L} \left( \mu, \lambda_{++}, \lambda_{+-}  \right) := 
	&g(\mu,  \rho_\alpha) - \lambda_{++} \cdot \left( \sum_{\left(  \sigma_1, \tau_1 \right) \in \left\lbrace (-1, -1), (+1, +1) \right\rbrace } \left[ \sum_{\left( \sigma_2, \tau_2 \right) \in  \{\pm 1\}^2} \mu \left(\sigma_1, \tau_1, \sigma_2, \tau_2 \right)\right]  - \rho_\alpha \left( \sigma_1, \tau_1 \right)\right) \\
	&- \lambda_{++} \cdot \left( \sum_{\left(  \sigma_2, \tau_2 \right) \in \left\lbrace (-1, -1), (+1, +1) \right\rbrace } \left[ \sum_{\left( \sigma_1, \tau_1 \right) \in  \{\pm 1\}^2} \mu \left(\sigma_1, \tau_1, \sigma_2, \tau_2 \right)\right]  - \rho_\alpha \left(\sigma_2, \tau_2 \right)\right) \\
	&- \lambda_{+-} \cdot \left( \sum_{\left(  \sigma_1, \tau_1 \right) \in \left\lbrace (-1, +1), (+1, -1) \right\rbrace } \left[ \sum_{\left( \sigma_2, \tau_2 \right) \in  \{\pm 1\}^2} \mu \left( \sigma_1, \tau_1, \sigma_2, \tau_2 \right)\right]  - \rho_\alpha \left(  \sigma_1, \tau_1 \right)\right) \\
	&- \lambda_{+-} \cdot \left( \sum_{\left(  \sigma_2, \tau_2 \right) \in \left\lbrace (-1, +1), (+1, -1) \right\rbrace } \left[ \sum_{\left( \sigma_1, \tau_1 \right) \in  \{\pm 1\}^2} \mu \left( \sigma_1, \tau_1, \sigma_2, \tau_2 \right)\right]  - \rho_\alpha \left( \sigma_2, \tau_2 \right)\right)
\end{align*}
Keeping the symmetry in mind, it suffices to consider the following derivatives of the Lagrangian function
\begin{align*}
	\frac{\partial \mathcal{L} \left( \mu, \lambda_{++}, \lambda_{+-}  \right)}{\partial \mu_{++++}} &= 1+ \log \left( \frac{\mu_{++++}}{\rho_{++}^2}\right) + 2 \beta - 2 \lambda_{++}  \\
	\frac{\partial \mathcal{L} \left( \mu, \lambda_{++}, \lambda_{+-}  \right)}{\partial \mu_{++--}} &=  1+ \log \left( \frac{\mu_{++--}}{\rho_{++}^2}\right) - 2 \lambda_{++}\\
	\frac{\partial \mathcal{L} \left( \mu, \lambda_{++}, \lambda_{+-}  \right)}{\partial \mu_{+-+-}} &= 1+ \log \left( \frac{\mu_{+-+-}}{\rho_{+-}^2}\right) + 2 \beta - 2 \lambda_{+-}\\
	\frac{\partial \mathcal{L} \left( \mu, \lambda_{++}, \lambda_{+-}  \right)}{\partial \mu_{+--+}} &= 1+ \log \left( \frac{\mu_{+--+}}{\rho_{+-}^2}\right) - 2 \lambda_{+-}\\
	\frac{\partial \mathcal{L} \left( \mu, \lambda_{++}, \lambda_{+-}  \right)}{\partial \mu_{+++-}}&= 1+ \log \left( \frac{\mu_{+++-}}{\rho_{++}\cdot \rho_{+-}}\right) + \beta - \lambda_{++} - \lambda_{+-}
\end{align*}
Setting these derivatives equal to zero, we instantly obtain
\begin{align}
	\mu_{++++}^* &= \rho_{++}^2 \exp \left(  2 \lambda_{++} - 2 \beta -1 \right) = \rho_{++}^2 x_1^2 e^{-2\beta} \label{mu_first_condition_sec_mom}\\
	\mu_{++--}^* &= \rho_{++}^2 \exp \left(  2 \lambda_{++} -1 \right) = \rho_{++}^2 x_1^2 \\
	\mu_{+-+-}^* &= \rho_{+-}^2 \exp \left(  2 \lambda_{+-} - 2 \beta -1 \right)= \rho_{+-}^2 x_2^2 e^{-2\beta} \\
	\mu_{+--+}^*& = \rho_{+-}^2 \exp \left(  2 \lambda_{+-} -1 \right)= \rho_{+-}^2 x_2^2  \\
	\mu_{+++-}^*& = \rho_{++} \rho_{+-}  \exp \left(  \lambda_{++} + \lambda_{+-} - \beta -1 \right)= \rho_{++}\rho_{+-} x_1 x_2  e^{-\beta}\label{mu_fifth_condition_sec_mom}
\end{align}
where the second equalities in each line represent a notational simplification by introducing $x_1 := e^{\lambda_{++}-\frac{1}{2}}$ and $x_2 := e^{\lambda_{+-}-\frac{1}{2}}$. 
\begin{lemma}\label{lemma_sol_min_sec_moment}
    The above system of equations (\ref{mu_first_condition_sec_mom})-(\ref{mu_fifth_condition_sec_mom}) and therefore the minimization problem of $g \left( \mu, \rho_\alpha \right)$ has the unique solution
    \begin{align*}
    x_1 =  2 \sqrt{ \frac{  \left( 1+ e^{-2 \beta} \right)^2 + \alpha \left( 1- e^{-2 \beta} \right)^2  - 2  e^{- \beta} z }{ \left( 1+ \alpha \right)^2 \left( 1+ e^{-2 \beta}\right) \left( 1- e^{-2 \beta}\right)^2 }}.    
    \end{align*}   
    and
    \begin{align*}
    x_2 =  2 \sqrt{ \frac{  \left( 1+ e^{-2 \beta} \right)^2 - \alpha \left( 1- e^{-2 \beta} \right)^2  \pm 2  e^{- \beta} z }{ \left( 1- \alpha \right)^2 \left( 1+ e^{-2 \beta}\right) \left( 1- e^{-2 \beta}\right)^2 }}.
    \end{align*}
\end{lemma}

\begin{proof}
In order to solve this system of equations, we recall two of the initial constraints of our minimization problem, i.e.
\begin{align*}
	\mu_{++++} + \mu_{+++-} + \mu_{++-+}+ \mu_{++--} &= \rho_{++} \\
	\mu_{+-++} + \mu_{+---} + \mu_{+--+} + \mu_{+-+-} &= \rho_{+-}.
\end{align*}
Plugging in the  $\mu^*$ we derived above and once again keeping in mind the symmetry of the problem, the two constraints can be reformulated into
\begin{align*}
	\rho_{++} x_1^2 e^{-2\beta} + 2\rho_{+-} x_1 x_2  e^{-\beta} + \rho_{++} x_1^2 &= 1 \\
	\rho_{+-} x_2^2 e^{-2\beta} + 2\rho_{++} x_1 x_2  e^{-\beta} + \rho_{+-} x_2^2 &= 1.
\end{align*}
which in turn yields 
\[
x_1 = \frac{1- x_2^2 \rho_{+-} \left( 1+ e^{-2 \beta}\right) }{2 \rho_{++} e^{- \beta} x_2}.
\]
Substituting $x_1$ into the first constraint, we arrive at
\begin{align*}
	\frac{\left(1 + e^{- 2 \beta}\right) \left( 1- x_2^2 \rho_{+-} \left(1 + e^{- 2 \beta}\right) \right)^2 }{4 \rho_{++} e^{- 2 \beta} x_2^2} + \frac{\rho_{+-}}{\rho_{++}} \cdot \left(  1- x_2^2 \rho_{+-} \left(1 + e^{- 2 \beta}\right) \right) =1 
\end{align*}
For notational convenience, we substitute $x:= x_2^2$ and $\kappa := 1 + e^{- 2 \beta}$. As a consequence, the previous equation can be expressed as
\begin{align*}
	x^2 \left( \rho_{+-}^2  \kappa^3 - 4 \rho_{+-}^2 \kappa e^{-2 \beta}\right)  + x \left( 4 \rho_{+-} e^{-2 \beta} - 2 \rho_{+-} \kappa^2 - 4 \rho_{++} e^{-2 \beta} \right) + \kappa = 0. 
\end{align*}
Now, we are able to apply the quadratic formula which yields
\begin{align*}
	x &=  \frac{- 4 \rho_{+-} e^{-2 \beta} + 2 \rho_{+-} \kappa^2 + 4 \rho_{++} e^{-2 \beta}  \pm \sqrt{\left( 4 \rho_{+-} e^{-2 \beta} - 2 \rho_{+-} \kappa^2 - 4 \rho_{++} e^{-2 \beta} \right)^2 - 4 \rho_{+-}^2 \kappa^2 \left( \kappa^2 - 4 e^{-2 \beta} \right) }}{2 \left( \rho_{+-}^2  \kappa^3 - 4 \rho_{+-}^2 \kappa e^{-2 \beta}\right) }\\
	&= \frac{\frac{1}{2} \left( 1- \alpha \right)  \left( 1+ e^{-2 \beta} \right)^2 + \left( 1+ \alpha \right) e^{-2 \beta} - \left( 1- \alpha\right) e^{-2 \beta} \pm \sqrt{\left( 2 \alpha e^{-2 \beta} + \frac{1}{2} \left( 1 - \alpha \right) \kappa^2 \right)^2 -  \frac{1}{4} \left( 1 - \alpha \right)^2 \kappa^4 + \left( 1 - \alpha \right)^2 \kappa^2 e^{-2 \beta}} }{\frac{1}{8} \left( 1- \alpha \right)^2 \left( 1+ e^{-2 \beta}\right) \left( \left( 1+ e^{-2 \beta}\right)^2 - 4 e^{-2 \beta}\right) }\\
	&= 4 \frac{  \left( 1+ e^{-2 \beta} \right)^2 - \alpha \left( 1- e^{-2 \beta} \right)^2  \pm 2 \sqrt{ 4 \alpha^2 e^{-4 \beta} + 2 \alpha e^{-2 \beta} \left(  1- \alpha \right)  \left( 1+ e^{-2 \beta}\right)^2  + \left( 1 - \alpha \right)^2 \left( 1+ e^{-2 \beta}\right)^2 e^{-2 \beta}} }{ \left( 1- \alpha \right)^2 \left( 1+ e^{-2 \beta}\right) \left( 1- e^{-2 \beta}\right)^2  }\\
\end{align*}
Focusing on the square root term, we note
\begin{align*}
	&4 \alpha^2 e^{-4 \beta} + 2 \alpha e^{-2 \beta} \left(  1- \alpha \right)  \left( 1+ e^{-2 \beta}\right)^2  + \left( 1 - \alpha \right)^2 \left( 1+ e^{-2 \beta}\right)^2 e^{-2 \beta} \\
	&= 4 \alpha^2 e^{-4 \beta}  -  \alpha^2 e^{-2 \beta}  \left( 1+ e^{-2 \beta}\right)^2 + \left( 1+ e^{-2 \beta}\right)^2 e^{-2 \beta}  = e^{-2 \beta} \left( \left( 1+ e^{-2 \beta}\right)^2 - \alpha^2  \left( 1- e^{-2 \beta}\right)^2  \right) 
\end{align*}
which leads us to
\begin{align*}
x = 4 \frac{  \left( 1+ e^{-2 \beta} \right)^2 - \alpha \left( 1- e^{-2 \beta} \right)^2  \pm 2  e^{- \beta} \sqrt{\left( 1+ e^{-2 \beta}\right)^2 - \alpha^2  \left( 1- e^{-2 \beta}\right)^2} }{ \left( 1- \alpha \right)^2 \left( 1+ e^{-2 \beta}\right) \left( 1- e^{-2 \beta}\right)^2 }
= 4 \frac{  \left( 1+ e^{-2 \beta} \right)^2 - \alpha \left( 1- e^{-2 \beta} \right)^2  \pm 2  e^{- \beta} z }{ \left( 1- \alpha \right)^2 \left( 1+ e^{-2 \beta}\right) \left( 1- e^{-2 \beta}\right)^2 }
\end{align*}
where we have introduced $z := \sqrt{\left( 1+ e^{-2 \beta}\right)^2 - \alpha^2  \left( 1- e^{-2 \beta}\right)^2}$ for notational convenience. At this point, we recall that $x =x_2^2$ to arrive at
\begin{align*}
x_2 =  2 \sqrt{ \frac{  \left( 1+ e^{-2 \beta} \right)^2 - \alpha \left( 1- e^{-2 \beta} \right)^2  \pm 2  e^{- \beta} z }{ \left( 1- \alpha \right)^2 \left( 1+ e^{-2 \beta}\right) \left( 1- e^{-2 \beta}\right)^2 }} 
\end{align*}
where we discard the negative square root since $x_2$ by definition is of the form $x_2 = e^{\lambda_{+-}-\frac{1}{2}}$ and thereby always non-negative. This leaves us with two potential solutions for $x_2$ which only differ in the $\pm$ sign in the above equation. Leaving out the detailed calculation, it is easy to show that choosing $+$ at the $\pm$ sign would result in a negative $x_1$. However, similar to $x_2$, $x_1  = e^{\lambda_{++}-\frac{1}{2}}$ also cannot become negative by construction. As a consequence, the only remaining and suitable candidate for $x_2$ and thereby the solution is
\begin{align*}
x_2 =  2 \sqrt{ \frac{  \left( 1+ e^{-2 \beta} \right)^2 - \alpha \left( 1- e^{-2 \beta} \right)^2  - 2  e^{- \beta} z }{ \left( 1- \alpha \right)^2 \left( 1+ e^{-2 \beta}\right) \left( 1- e^{-2 \beta}\right)^2 }}.    
\end{align*}
With this solution for $x_2$ we are now able to calculate  the optimal $x_1$. More specifically, we recall the formula we have derived a few steps back
\begin{align*}
x_1 = \frac{1- x_2^2 \rho_{+-} \left( 1+ e^{-2 \beta}\right) }{2 \rho_{++} e^{- \beta} x_2}.    
\end{align*}
Plugging in the optimal $x_2$, we arrive at the expression
\begin{align} \label{eq_lagrange_x1}
	x_1 = \frac{1- \frac{  \left( 1+ e^{-2 \beta} \right)^2 - \alpha \left( 1- e^{-2 \beta} \right)^2  - 2  e^{- \beta} z }{ \left( 1- \alpha \right)  \left( 1- e^{-2 \beta}\right)^2 } }{\left(  1 + \alpha \right)  e^{- \beta} \sqrt{ \frac{  \left( 1+ e^{-2 \beta} \right)^2 - \alpha \left( 1- e^{-2 \beta} \right)^2  - 2  e^{- \beta} z }{ \left( 1- \alpha \right)^2 \left( 1+ e^{-2 \beta}\right) \left( 1- e^{-2 \beta}\right)^2 }}}
	= \frac{\left( -4 e^{- \beta} + 2  z\right) \sqrt{1 + e^{- 2 \beta}} }{\left(  1 + \alpha \right)  \left( 1- e^{-2 \beta}\right)  \sqrt{   \left( 1+ e^{-2 \beta} \right)^2 - \alpha \left( 1- e^{-2 \beta} \right)^2  - 2  e^{- \beta} z }}.
\end{align}
Next, we claim that
\begin{align*}
x_1 =  2 \sqrt{ \frac{  \left( 1+ e^{-2 \beta} \right)^2 + \alpha \left( 1- e^{-2 \beta} \right)^2  - 2  e^{- \beta} z }{ \left( 1+ \alpha \right)^2 \left( 1+ e^{-2 \beta}\right) \left( 1- e^{-2 \beta}\right)^2 }}.    
\end{align*}
Indeed, we find starting at \eqref{eq_lagrange_x1} that
\begin{align*}
	x_1 = \frac{\left( -4 e^{- \beta} + 2  z\right) \sqrt{1 + e^{- 2 \beta}} }{\left(  1 + \alpha \right)  \left( 1- e^{-2 \beta}\right)  \sqrt{   \left( 1+ e^{-2 \beta} \right)^2 - \alpha \left( 1- e^{-2 \beta} \right)^2  - 2  e^{- \beta} z }}.
\end{align*}
Thus, our claim is equivalent to
\begin{align}\label{simpopt}
	\left( z - 2 e^{- \beta}\right)^2  \left( 1 + e^{-2 \beta} \right)^2 
	= \left( \left( 1+ e^{-2 \beta} \right)^2 + \alpha \left( 1- e^{-2 \beta} \right)^2  - 2  e^{- \beta} z\right) \left( \left( 1+ e^{-2 \beta} \right)^2 - \alpha \left( 1- e^{-2 \beta} \right)^2  - 2  e^{- \beta} z\right).
\end{align}
To see that \eqref{simpopt} is indeed true, we execute the following auxiliary calculation:
\begin{align*}
	&\left( \left( 1+ e^{-2 \beta} \right)^2 + \alpha \left( 1- e^{-2 \beta} \right)^2  - 2  e^{- \beta} z\right) \left( \left( 1+ e^{-2 \beta} \right)^2 - \alpha \left( 1- e^{-2 \beta} \right)^2  - 2  e^{- \beta} z\right) \\
	&=  \left( 1+ e^{-2 \beta} \right)^4 - 4 e^{- \beta} z  \left( 1+ e^{-2 \beta} \right)^2 + 4 e^{-2 \beta} \left(  \left( 1+ e^{-2 \beta} \right)^2  - \alpha^2  \left( 1- e^{-2 \beta} \right)^2 \right) - \alpha^2  \left( 1- e^{-2 \beta} \right)^4 \\
	&= \left( 1+ e^{-2 \beta} \right)^2  \left(z^2  - 4 e^{- \beta} z + 4 e^{-2 \beta}\right) =  \left( 1+ e^{-2 \beta} \right)^2  \left( z - 2e^{- \beta}\right)^2.
\end{align*}
Hence, we established our claim and thus know
\begin{align*}
x_1 =  2 \sqrt{ \frac{  \left( 1+ e^{-2 \beta} \right)^2 + \alpha \left( 1- e^{-2 \beta} \right)^2  - 2  e^{- \beta} z }{ \left( 1+ \alpha \right)^2 \left( 1+ e^{-2 \beta}\right) \left( 1- e^{-2 \beta}\right)^2 }}.    
\end{align*}
\end{proof}
Let us bring together our findings of this subsection. Due to the well-known fact that the Kullback-Leibler divergence is convex in its input parameters, we immediately see that the function  $ g(\mu,  \rho_\alpha)$ is convex as well. As a consequence, the $\mu^*$ we have just calculated is indeed the minimum. Put differently, we are now able to state
\begin{align*}
	\min_{\mu \in \mathcal{O}' } g(\mu,  \rho_\alpha) = g(\mu^*,  \rho_\alpha).
\end{align*}
Although this statement is satisfactory, we would favor a more explicit expression. This is achieved by the following Lemma.
\begin{lemma}\label{lemma_g_simplified_version}
    We have 
    \begin{align*}
	g(\mu^*,  \rho_\alpha) &= 2 \log\left( 2\right)  - \log \left(  \left( 1 +e^{- 2 \beta}\right)  \left( 1- e^{-2 \beta}\right) ^2\right) - \left(1 + \alpha \right) \log \left( 1 + \alpha \right)  \\
	&- \left(1 - \alpha \right) \log \left( 1 - \alpha \right) + \frac{1 + \alpha}{2} \log\left(  \left( 1+ e^{-2 \beta} \right)^2 + \alpha \left( 1- e^{-2 \beta} \right)^2  - 2  e^{- \beta} z \right) \\
	&+\frac{1 - \alpha}{2} \log\left(  \left( 1+ e^{-2 \beta} \right)^2 - \alpha \left( 1- e^{-2 \beta} \right)^2  - 2  e^{- \beta} z \right).
\end{align*}
\end{lemma}

\begin{proof}
In order to get to the desired expression, we take a closer look at the Kullback-Leibler divergence for the optimal $\mu^*$
\begin{align*}
	D_\textrm{KL} (\mu^* \vert\vert \rho_\alpha \otimes \rho_\alpha) 
	&= 2 \mu_{++++}^* \log \left( x_1^2 e^{-2 \beta}\right) +2 \mu_{++--}^* \log \left( x_1^2 \right) +2 \mu_{+-+-}^* \log \left( x_2^2 e^{-2 \beta}\right) \\
	&\qquad +2 \mu_{+--+}^* \log \left( x_1^2 \right) +8 \mu_{+++-}^* \log \left( x_1 x_2  e^{- \beta}\right) \\
	&= - \beta \left(4 \mu_{++++}^* + 4 \mu_{+-+-}^* + 8 \mu_{+++-}^*  \right) + \log \left( x_1\right) \left( 4 \mu_{++++}^* + 4 \mu_{++--}^* + 8 \mu_{+++-}^*\right) \\
	&\qquad + \log \left( x_2\right) \left( 4 \mu_{+-+-}^* + 4 \mu_{+--+}^* + 8 \mu_{+++-}^*\right) 
\end{align*}
Using the reformulation of $D_\textrm{KL} (\mu^* \vert\vert \rho_\alpha \otimes \rho_\alpha)$, $g(\mu^*,  \rho_\alpha)$ can be formulated as
\begin{align*}
	g(\mu^*,  \rho_\alpha)
	=  \log \left( x_1\right) \left( 4 \mu_{++++}^* + 4 \mu_{++--}^* + 8 \mu_{+++-}^*\right) + \log \left( x_2\right) \left( 4 \mu_{+-+-}^* + 4 \mu_{+--+}^* + 8 \mu_{+++-}^*\right).
\end{align*}
This expression in turn is suitable for inserting $x_1$ and $x_2$ leading to
\begin{align*}
	g(\mu^*,  \rho_\alpha) &= 2 \log\left( 2\right)  - \log \left(  \left( 1 +e^{- 2 \beta}\right)  \left( 1- e^{-2 \beta}\right) ^2\right)\\
	&\qquad + \left[ \log\left(  \left( 1+ e^{-2 \beta} \right)^2 + \alpha \left( 1- e^{-2 \beta} \right)^2  - 2  e^{- \beta} z \right) - 2 \log \left(1+ \alpha \right)  \right] \cdot \left( 2 \mu_{++++}^* + 2 \mu_{++--}^* + 4 \mu_{+++-}^*\right) \\
	&\qquad + \left[ \log\left(  \left( 1+ e^{-2 \beta} \right)^2 - \alpha \left( 1- e^{-2 \beta} \right)^2  - 2  e^{- \beta} z \right) - 2 \log \left(1- \alpha \right)  \right] \cdot \left( 2 \mu_{+-+-}^* + 2 \mu_{+--+}^* + 4 \mu_{+++-}^*\right).
\end{align*}
Since $\mu^*$ is a probability measure by definition, we can exploit the identity
\begin{align*}
	2 \mu_{+-+-}^* + 2 \mu_{+--+}^* + 4 \mu_{+++-}^* = 1 - 2 \mu_{++++}^* - 2 \mu_{++--}^* - 4 \mu_{+++-}^*
\end{align*}
to rearrange $g(\mu^*,  \rho_\alpha)$ as
\begin{align*}
	g(\mu^*,  \rho_\alpha) &= 2 \log\left( 2\right)  - \log \left(  \left( 1 +e^{- 2 \beta}\right)  \left( 1- e^{-2 \beta}\right) ^2\right)
	+  \log\left(  \left( 1+ e^{-2 \beta} \right)^2 - \alpha \left( 1- e^{-2 \beta} \right)^2  - 2  e^{- \beta} z \right) - 2 \log \left(1- \alpha \right) \\
	&\qquad + \left[ \log\left( \frac{\left( 1+ e^{-2 \beta} \right)^2 + \alpha \left( 1- e^{-2 \beta} \right)^2  - 2  e^{- \beta} z}{\left( 1+ e^{-2 \beta} \right)^2 - \alpha \left( 1- e^{-2 \beta} \right)^2  - 2  e^{- \beta} z} \right) - 2 \log \left(\frac{1 + \alpha}{1- \alpha} \right)  \right] \cdot \left( 2 \mu_{++++}^* + 2 \mu_{++--}^* + 4 \mu_{+++-}^*\right).
\end{align*}
To keep the terms relatively brief, we define
\begin{align*}
	T_1 &:=   2 \log\left( 2\right)  - \log \left(  \left( 1 +e^{- 2 \beta}\right)  \left( 1- e^{-2 \beta}\right) ^2\right)
	+  \log\left(  \left( 1+ e^{-2 \beta} \right)^2 - \alpha \left( 1- e^{-2 \beta} \right)^2 - 2  e^{- \beta} z \right)
	- 2 \log \left(1- \alpha \right) \\
	T_2 &:= 2 \mu_{++++}^* + 2 \mu_{++--}^* + 4 \mu_{+++-}^* \\
	T_3 &:= \log\left( \frac{\left( 1+ e^{-2 \beta} \right)^2 + \alpha \left( 1- e^{-2 \beta} \right)^2  - 2  e^{- \beta} z}{\left( 1+ e^{-2 \beta} \right)^2 - \alpha \left( 1- e^{-2 \beta} \right)^2  - 2  e^{- \beta} z} \right) - 2 \log \left(\frac{1 + \alpha}{1- \alpha} \right) 
\end{align*}
which implies $g(\mu^*,  \rho_\alpha) = T_1 + T_2 \cdot T_3$. In the next step, we will plug in $\mu^*$ in order to simplify $T_2$
\begin{align*}
	T_2 &= 2 \mu_{++++}^* + 2 \mu_{++--}^* + 4 \mu_{+++-}^*  = 2 x_1^2 \rho_{++}^2 \left(1 + e^{-2 \beta} + 2 e^{- \beta} \frac{\rho_{+-} x_2}{\rho_{++} x_1} \right) \\
	&= \frac{  \left( 1+ e^{-2 \beta} \right)^2 + \alpha \left( 1- e^{-2 \beta} \right)^2  - 2  e^{- \beta} z }{ 2 \left( 1- e^{-2 \beta}\right)^2 } \\
	& \qquad + e^{- \beta} \frac{\sqrt{\left( \left( 1+ e^{-2 \beta} \right)^2 + \alpha \left( 1- e^{-2 \beta} \right)^2  - 2  e^{- \beta} z  \right) \left( \left( 1+ e^{-2 \beta} \right)^2 - \alpha \left( 1- e^{-2 \beta} \right)^2  - 2  e^{- \beta} z\right) }}{\left( 1+ e^{-2 \beta}\right) \left( 1- e^{-2 \beta}\right)^2}
\end{align*}
Applying \eqref{simpopt} to the term in the square root yields
\begin{align*}
	T_2 &=  \frac{  \left( 1+ e^{-2 \beta} \right)^2 + \alpha \left( 1- e^{-2 \beta} \right)^2  - 2  e^{- \beta} z }{ 2 \left( 1- e^{-2 \beta}\right)^2 } + e^{- \beta} \frac{\left( z - 2 e^{- \beta}\right) \left( 1+ e^{-2 \beta} \right) }{\left( 1+ e^{-2 \beta}\right) \left( 1- e^{-2 \beta}\right)^2}\\
	&= \frac{  \left( 1+ e^{-2 \beta} \right)^2 + \alpha \left( 1- e^{-2 \beta} \right)^2  - 4 e^{-2 \beta} }{ 2 \left( 1- e^{-2 \beta}\right)^2 } = \frac{1 + \alpha}{2}.
\end{align*}
Coming back to $g(\mu^*,  \rho_\alpha)$, we obtain the expression that Lemma \ref{lemma_g_simplified_version} promised
\begin{align*}
	g(\mu^*,  \rho_\alpha) &= 2 \log\left( 2\right)  - \log \left(  \left( 1 +e^{- 2 \beta}\right)  \left( 1- e^{-2 \beta}\right) ^2\right) - \left(1 + \alpha \right) \log \left( 1 + \alpha \right) - \left(1 - \alpha \right) \log \left( 1 - \alpha \right) \\
	&\qquad + \frac{1 + \alpha}{2} \log\left(  \left( 1+ e^{-2 \beta} \right)^2 + \alpha \left( 1- e^{-2 \beta} \right)^2  - 2  e^{- \beta} z \right)
	+\frac{1 - \alpha}{2} \log\left(  \left( 1+ e^{-2 \beta} \right)^2 - \alpha \left( 1- e^{-2 \beta} \right)^2  - 2  e^{- \beta} z \right).
\end{align*}
\end{proof}
\subsection{Maximization with respect to $\alpha$}

In this subsection we focus on the function
\begin{align*}
	f_d \left( \alpha, \beta \right) :=  \log (2) + \textrm{H}\left( \frac{1 + \alpha}{2}\right) - \frac{d}{2} g(\mu^*,  \rho_\alpha).
\end{align*}
which results from plugging in the definition of $\rho$ in terms of $\alpha$ from \eqref{eq_rho_alpha}. More specifically, we are interested in solving the optimization
\begin{align*}
	\max_{-1 < \alpha <1}	f_d \left( \alpha, \beta \right).
\end{align*}
which will immediately yield the answer to our initial optimization problem over $\delta \left( \mu, \rho \right)$. Note that we tacitly exploit the results of both Lemma \ref{lemma_sol_min_sec_moment} and Lemma \ref{lemma_g_simplified_version} to be able to state a function $f_d \left( \alpha, \beta \right)$ that only depends on $d, \alpha,$ and $\beta$. As a consequence, we have to prove the following statement.

\begin{lemma}\label{lemma_optimal_alpha}
Assume that $0 < \beta < \bks$. Then we have
\begin{align*}
	\arg \max_{-1 < \alpha <1}	f_d \left( \alpha, \beta \right) = 0.
\end{align*}
\end{lemma}

\begin{proof}
To solve the maximization with respect to $\alpha$, we  calculate the derivatives. Let us start with the simpler ones, namely the first and second derivative of the entropy with respect to $\alpha$:
\begin{align*}
	\frac{\partial \textrm{H}\left( \frac{1 + \alpha}{2}\right) }{\partial \alpha} = \frac{1}{2} \log \left( 1- \alpha \right) -  \frac{1}{2} \log \left( 1+ \alpha \right)
\end{align*}
and
\begin{align*}
	\frac{\partial^2 \textrm{H}\left( \frac{1 + \alpha}{2}\right) }{\partial \alpha^2} = \dfrac{1}{2} \left( \frac{-1}{1- \alpha} - \frac{1}{1 + \alpha} \right) = - \frac{1}{1 - \alpha^2}.
\end{align*}
Before we continue with our main task, let us state a useful observation which will be helpful in the following calculations. Let
\begin{align*}
    z = \sqrt{\left( 1+ e^{-2 \beta}\right)^2 - \alpha^2  \left( 1- e^{-2 \beta}\right)^2}
\end{align*}
Then, we have
\begin{align*}
	\frac{\partial z }{\partial \alpha} &= \frac{- \alpha  \left( 1- e^{-2 \beta} \right)^2}{\sqrt{\left( 1+ e^{-2 \beta}\right)^2 - \alpha^2  \left( 1- e^{-2 \beta}\right)^2}} = - \alpha  \left( 1- e^{-2 \beta} \right)^2 z^{-1}.
\end{align*}
Next, we determine the first two derivatives for $g(\mu^*,  \rho_\alpha)$. Starting with the first derivative, we find
\begin{align*}
	&\frac{\partial g(\mu^*,  \rho_\alpha) }{\partial \alpha} = -\frac{1 + \alpha}{1 + \alpha} - \log \left( \frac{1+ \alpha}{2}\right) + \frac{1 - \alpha}{1 - \alpha} + \log \left( \frac{1- \alpha}{2}\right)\\
	&+ \frac{1}{2}\left[   \log\left(  \left( 1+ e^{-2 \beta} \right)^2 + \alpha \left( 1- e^{-2 \beta} \right)^2  - 2  e^{- \beta} z \right) -  \log\left(  \left( 1+ e^{-2 \beta} \right)^2 - \alpha \left( 1- e^{-2 \beta} \right)^2  - 2  e^{- \beta} z \right)\right] \\
	&+ \frac{1 + \alpha}{2} \cdot \frac{ \left( 1- e^{-2 \beta} \right)^2 + 2 e^{- \beta}\alpha  \left( 1- e^{-2 \beta} \right)^2 z^{-1}}{\left( 1+ e^{-2 \beta} \right)^2 + \alpha \left( 1- e^{-2 \beta} \right)^2  - 2  e^{- \beta} z}
	+ \frac{1 - \alpha}{2} \cdot \frac{ -\left( 1- e^{-2 \beta} \right)^2 + 2 e^{- \beta}\alpha  \left( 1- e^{-2 \beta} \right)^2 z^{-1}}{\left( 1+ e^{-2 \beta} \right)^2 - \alpha \left( 1- e^{-2 \beta} \right)^2  - 2  e^{- \beta} z}
\end{align*}
For the next simplification, we focus on the last two summands of the previously stated derivative, i.e.
\begin{align*}
	& \frac{1 + \alpha}{2} \cdot \frac{ \left( 1- e^{-2 \beta} \right)^2 + 2 e^{- \beta}\alpha  \left( 1- e^{-2 \beta} \right)^2 z^{-1}}{\left( 1+ e^{-2 \beta} \right)^2 + \alpha \left( 1- e^{-2 \beta} \right)^2  - 2  e^{- \beta} z}
	+ \frac{1 - \alpha}{2} \cdot \frac{ -\left( 1- e^{-2 \beta} \right)^2 + 2 e^{- \beta}\alpha  \left( 1- e^{-2 \beta} \right)^2 z^{-1}}{\left( 1+ e^{-2 \beta} \right)^2 - \alpha \left( 1- e^{-2 \beta} \right)^2  - 2  e^{- \beta} z}\\
	&= \frac{\left( 1- e^{-2 \beta} \right)^2}{2 z} \cdot \frac{\left(  1 + \alpha  \right) \left( z + 2 e^{- \beta}\alpha \right) \left(   \left( 1+ e^{-2 \beta} \right)^2 - \alpha \left( 1- e^{-2 \beta} \right)^2  - 2  e^{- \beta} z\right) }{\left(  \left( 1+ e^{-2 \beta} \right)^2 + \alpha \left( 1- e^{-2 \beta} \right)^2  - 2  e^{- \beta} z \right)  \left(   \left( 1+ e^{-2 \beta} \right)^2 - \alpha \left( 1- e^{-2 \beta} \right)^2  - 2  e^{- \beta} z\right) }\\
	&+ \frac{\left( 1- e^{-2 \beta} \right)^2}{2 z} \cdot \frac{\left(  1 - \alpha  \right) \left( -z + 2 e^{- \beta}\alpha \right) \left(   \left( 1+ e^{-2 \beta} \right)^2 + \alpha \left( 1- e^{-2 \beta} \right)^2  - 2  e^{- \beta} z\right) }{\left(  \left( 1+ e^{-2 \beta} \right)^2 + \alpha \left( 1- e^{-2 \beta} \right)^2  - 2  e^{- \beta} z \right)  \left(   \left( 1+ e^{-2 \beta} \right)^2 - \alpha \left( 1- e^{-2 \beta} \right)^2  - 2  e^{- \beta} z\right) }
\end{align*}
Again, we restrict our attention to one term, namely
\begin{align*}
	&\left(  1 + \alpha  \right) \left( z + 2 e^{- \beta}\alpha \right) \left(   \left( 1+ e^{-2 \beta} \right)^2 - \alpha \left( 1- e^{-2 \beta} \right)^2  - 2  e^{- \beta} z\right) \\
	&+ \left(  1 - \alpha  \right) \left( -z + 2 e^{- \beta}\alpha \right) \left(   \left( 1+ e^{-2 \beta} \right)^2 + \alpha \left( 1- e^{-2 \beta} \right)^2  - 2  e^{- \beta} z\right)\\
	&=\left( \left( 1 + e^{- 2 \beta}\right)^2 - 2 e^{- \beta} z  \right) \left( 4 e^{- \beta} \alpha + 2 \alpha z \right) + \alpha \left( 1 - e^{- 2 \beta}\right)^2 \left( -2 z - 4 e^{- \beta} \alpha^2 \right) \\
	&= 2 \alpha z 4 e^{-2 \beta} - 8  e^{-2 \beta} \alpha z= 0.
\end{align*}
As a result, the first derivative can be reduced to
\begin{align*}
	\frac{\partial g(\mu^*,  \rho_\alpha) }{\partial \alpha} &= - \log \left( 1+ \alpha\right) + \log \left( 1- \alpha\right) +
	\frac{1}{2} \log\left( \frac{ \left( 1+ e^{-2 \beta} \right)^2 + \alpha \left( 1- e^{-2 \beta} \right)^2  - 2  e^{- \beta} z} {\left( 1+ e^{-2 \beta} \right)^2 - \alpha \left( 1- e^{-2 \beta} \right)^2  - 2  e^{- \beta} z }\right).
\end{align*}
Based on this result, we can instantly compute the second derivative
\begin{align*}
	\frac{\partial^2 g(\mu^*,  \rho_\alpha) }{\partial \alpha^2}
	 = -\frac{2}{1- \alpha^2} + \frac{1}{2} \cdot \left[ \frac{ \left( 1- e^{-2 \beta} \right)^2 + 2 e^{- \beta}\alpha  \left( 1- e^{-2 \beta} \right)^2 z^{-1}}{\left( 1+ e^{-2 \beta} \right)^2 + \alpha \left( 1- e^{-2 \beta} \right)^2  - 2  e^{- \beta} z} -  \frac{ -\left( 1- e^{-2 \beta} \right)^2 + 2 e^{- \beta}\alpha  \left( 1- e^{-2 \beta} \right)^2 z^{-1}}{\left( 1+ e^{-2 \beta} \right)^2 - \alpha \left( 1- e^{-2 \beta} \right)^2  - 2  e^{- \beta} z} \right].
\end{align*}
Once again, we apply \eqref{simpopt} to get to
\begin{align*}
	\frac{\partial^2 g(\mu^*,  \rho_\alpha) }{\partial \alpha^2}
	&= -\frac{2}{1- \alpha^2} + \left( 1 - e^{-2 \beta}\right)^2 \frac{2 z \left( \left( 1+ e^{-2 \beta} \right)^2  - 2  e^{- \beta} z\right) - 4 e^{- \beta} \alpha^2 \left( 1- e^{-2 \beta} \right)^2} {2z \left( z - 2 e^{- \beta}\right)^2  \left( 1 + e^{-2 \beta} \right)^2}\\
	&= -\frac{2}{1- \alpha^2} + \left( 1 - e^{-2 \beta}\right)^2 \frac{ \left( 1 + e^{-2 \beta} \right)^2 \left(z - 2 e^{- \beta} \right) } {z \left( z - 2 e^{- \beta}\right)^2  \left( 1 + e^{-2 \beta} \right)^2} =  \frac{ \left( 1 - e^{-2 \beta} \right)^2 } {z \left( z - 2 e^{- \beta}\right) } -\frac{2}{1- \alpha^2}.
\end{align*}
Finally, combining the derivatives of the entropy and $g(\mu^*, \rho_a)$ we arrive at
\begin{align*}
	&\frac{\partial f_d \left( \alpha, \beta \right) }{\partial \alpha} = \frac{\partial \textrm{H}\left( \frac{1 + \alpha}{2}\right) }{\partial \alpha} - \frac{d}{2} \cdot \frac{\partial g(\mu^*,  \rho_\alpha) }{\partial \alpha}\\
	&= \frac{1}{2} \log \left( 1- \alpha \right) -  \frac{1}{2} \log \left( 1+ \alpha \right) \\
	&\qquad + \frac{d}{2} \left(  \log \left( 1+ \alpha\right) - \log \left( 1- \alpha\right) - \frac{1}{2} \log\left( \frac{ \left( 1+ e^{-2 \beta} \right)^2 + \alpha \left( 1- e^{-2 \beta} \right)^2  - 2  e^{- \beta} z} {\left( 1+ e^{-2 \beta} \right)^2 - \alpha \left( 1- e^{-2 \beta} \right)^2  - 2  e^{- \beta} z }\right)\right) \\
	&=    \frac{d - 1}{2} \log \left( 1+ \alpha \right) - \frac{d -1}{2} \log \left( 1- \alpha \right) - \frac{d}{4} \log\left( \frac{ \left( 1+ e^{-2 \beta} \right)^2 + \alpha \left( 1- e^{-2 \beta} \right)^2  - 2  e^{- \beta} z} {\left( 1+ e^{-2 \beta} \right)^2 - \alpha \left( 1- e^{-2 \beta} \right)^2  - 2  e^{- \beta} z }\right)
\end{align*}
and
\begin{align*}
	\frac{\partial^2 f_d \left( \alpha, \beta \right) }{\partial \alpha^2} &= \frac{\partial^2 \textrm{H}\left( \frac{1 + \alpha}{2}\right) }{\partial \alpha^2} - \frac{d}{2} \cdot \frac{\partial^2 g(\mu^*,  \rho_\alpha) }{\partial^2 \alpha}
	= - \frac{1}{1 - \alpha^2} - \frac{d}{2} \cdot \left( \frac{ \left( 1 - e^{-2 \beta} \right)^2 } {z \left( z - 2 e^{- \beta}\right) } -\frac{2}{1- \alpha^2} \right)\\
	&= d \cdot \frac{z^2 -4  z e^{-\beta} + \left( 1 + e^{-2 \beta} \right)^2 - \alpha^2 \left( 1 - e^{-2 \beta} \right)^2 - \left( 1 - e^{-2 \beta} \right)^2 + \alpha^2 \left( 1 - e^{-2 \beta} \right)^2}{\left( 1 - \alpha^2 \right) 2  z \left( z - 2 e^{- \beta}\right)} - \frac{1}{1- \alpha^2}\\
	&= \frac{d-2}{2 \left( 1- \alpha^2\right) } - \frac{d e^{- \beta}}{\left( 1- \alpha^2\right) z} =\frac{d-2}{2 \left( 1- \alpha^2\right) } - \frac{d }{\left( 1- \alpha^2\right) \sqrt{\left( e^\beta + e^{- \beta}\right)^2 - \alpha^2  \left( e^\beta - e^{- \beta}\right)^2}}\\
	&= \frac{d-2}{2 \left( 1- \alpha^2\right) } - \frac{d }{\left( 1- \alpha^2\right) \sqrt{\left( 1- \alpha^2\right) \cdot \left( e^{2 \beta} + e^{-2 \beta}\right)+ 2 + 2 \alpha^2 }}.
\end{align*}
Furthermore, we note that for every $\beta $ we have for $\alpha = 0$
\begin{align*}
	\frac{\partial f_d  }{\partial \alpha}\left( 0, \beta \right) = 0.
\end{align*}
Now, to complete the maximization with respect to $\alpha$, we claim is that the global maximum of $f_d \left( \alpha, \beta \right) $ is at $\alpha = 0$ as long as $\beta<\bks$. We prove this claim in two steps. First, we show that $\frac{\partial^2 f_d  }{\partial \alpha^2} \left( \alpha, \beta \right)$ is increasing in $\beta$. Subsequently, we establish that $\frac{\partial^2 f_d  }{\partial \alpha^2} \left( \alpha, \beta^* \right)$ is smaller than zero for all $\alpha \in (-1, 1)$.
As a consequence,  $\frac{\partial^2 f_d  }{\partial \alpha^2} \left( \alpha, \beta \right) < 0$ holds for all $\beta \in \left( 0, \beta^* \right) $ and $\alpha \in (-1, 1)$ and thereby implies that the maximum  of $f_d \left( \alpha, \beta \right) $ is attained at $\alpha = 0$ for $\beta < \beta^*$.
The previously performed technical rearrangements are helpful for calculating the next derivative in a straightforward manner.
\begin{align*}
	\frac{\partial}{\partial \beta} \left( \frac{\partial^2 f_d  }{\partial \alpha^2}\right) \left( \alpha, \beta \right)
	&=  \frac{d}{2 \left( 1- \alpha^2\right)} \frac{ \left( 1- \alpha^2\right) \cdot \left( 2 \beta e^{2 \beta -1} - 2 \beta e^{-2 \beta -1}\right)}{ \left[ \left( 1- \alpha^2\right) \cdot \left( e^{2 \beta} + e^{-2 \beta}\right)+ 2 + 2 \alpha^2 \right]^{\frac{3}{2}}}\\
	&= \underbrace{\frac{d}{2} 2 \beta  e^{2 \beta -1}  \left( 1 -  e^{-4 \beta }\right)}_{>0} \underbrace{\left[ \underbrace{\left( 1- \alpha^2\right)}_{>0} \cdot \left( e^{2 \beta} + e^{-2 \beta}\right)+ 2 + 2 \alpha^2 \right]^{-\frac{3}{2}} }_{>0} > 0
\end{align*}
where we restrict our attention to $-1< \alpha <1$.  All that remains to do is to plug in the Kesten-Stigum bound into the second derivative with respect to alpha which yields
\begin{align*}
	\frac{\partial^2 f_d  }{\partial \alpha^2} \left( \alpha, \beta^* \right) &= \frac{d-2}{2 \left( 1- \alpha^2\right) } - \frac{d}{ 1- \alpha^2} \cdot \left[ \left( 1- \alpha^2\right) \cdot \left( \frac{\left( \sqrt{d-1}+1\right) ^2}{\left( \sqrt{d-1}-1\right) ^2} +  \frac{\left( \sqrt{d-1}-1\right) ^2}{\left( \sqrt{d-1}+1\right) ^2}\right) + 2 + 2 \alpha^2 \right]^{- \frac{1}{2}} \\
	&= \frac{d-2}{2 \left( 1- \alpha^2\right) } - \frac{d}{ 1- \alpha^2} \cdot \left[ \left( 1- \alpha^2\right) \cdot \left( \frac{\left( \sqrt{d-1}+1\right) ^4 + \left( \sqrt{d-1}-1\right) ^4}{\left(d-1 -1\right) ^2} \right) + 2 + 2 \alpha^2 \right]^{- \frac{1}{2}} \\
	&= \frac{d-2}{2 \left( 1- \alpha^2\right) } - \frac{d}{ 1- \alpha^2} \cdot \left[ \frac{ 4 d^2 + \alpha^2 \left( 16 - 16d\right) }{\left(d-2\right) ^2}   \right]^{- \frac{1}{2}}\\
	&= \frac{d-2}{2 \left( 1- \alpha^2\right) } \cdot \underbrace{\left( 1 - \frac{d}{\sqrt{d^2 - 4 \alpha^2 \left( d-1\right)}}\right) }_{<0}< 0
\end{align*}
where we assume both $d > 2$ and $-1< \alpha <1$. This concludes the maximization problem.
\end{proof}

What remains is to bring all the findings of this section together.

\begin{proof}[Proof of Lemma \ref{sec_mom_optimum}]
Substituting $\alpha = 0$ from Lemma \ref{lemma_optimal_alpha} into the previous reformulations, we can state that $\delta \left( \mu, \rho \right)$ obtains its optimum at $\mu^*$ where
\begin{align*}
	\mu_{++++}^* &= \mu_{----}^* = \mu_{+-+-}^* = \mu_{-+-+}^* = \frac{e^{-2 \beta}}{4 \left( 1 + e^{- \beta}\right)^2}\\
	\mu_{+--+}^* &= \mu_{--++}^* = \mu_{-++-}^* = \mu_{++--}^* = \frac{1}{4 \left( 1 + e^{- \beta}\right)^2}\\
	\mu_{+++-}^* &= \mu_{++-+}^* = \mu_{+-++}^* = \mu_{-+++}^* = \mu_{---+}^*= \mu_{--+-}^*= \mu_{-+--}^*= \mu_{+---}^* = \frac{e^{- \beta}}{4 \left( 1 + e^{- \beta}\right)^2}
\end{align*}
which also implies
\begin{align*}
	\rho_{++}^* = \rho_{+-}^* = \rho_{-+}^* = \rho_{--}^* = \frac{1}{4}
\end{align*}
and
\begin{align*}
	&\delta \left( \mu^*, \rho^* \right) =  \textrm{H}\left( \rho^* \right) -\frac{d}{2} \left( D_\textrm{KL} \left( \mu^* \vert\vert \rho^* \otimes \rho^* \right) + \beta   \sum_{\sigma \in A_1}  \mu^* (\sigma) + 2 \beta  \sum_{\sigma \in A_2}  \mu^* (\sigma)\right)\\
	&=  \left( 2 - 2 d\right)  \log \left( 2\right) + d \log \left( 2 \left(1 + e^{- \beta} \right)  \right) + \frac{d}{2} \left( -  \frac{e^{-2 \beta}}{ \left( 1 + e^{- \beta}\right)^2} \log \left(e^{-2 \beta} \right)  - \frac{2 e^{- \beta}}{ \left( 1 + e^{- \beta}\right)^2} \log \left(e^{- \beta} \right) \right)  - d \beta  \frac{e^{-2 \beta} + e^{- \beta}}{ \left( 1 + e^{- \beta}\right)^2}\\
	&=  \left( 2 - d\right)  \log \left( 2\right) + d \log \left(1 + e^{- \beta}  \right) + d \beta \left(  \frac{e^{-2 \beta}}{ \left( 1 + e^{- \beta}\right)^2}   + \frac{ e^{- \beta}}{ \left( 1 + e^{- \beta}\right)^2}  \right) - d \beta  \frac{e^{-2 \beta} + e^{- \beta}}{ \left( 1 + e^{- \beta}\right)^2}\\
	&=  \left( 2 - d\right)  \log \left( 2\right) + d \log \left(1 + e^{- \beta}  \right).
\end{align*}
Lemma \ref{sec_mom_optimum} readily follows.
\end{proof}

\section{The Hessian for the second moment / Proof of \Lem~\ref{hessian_second_moment}} \label{sec_hessian}

The proof of \Lem~\ref{hessian_second_moment} boils down to tedious calculations of the first and second partial derivatives.
As a starting point we reformulate $\delta \left( \mu, \rho \right)$ with the restricted number of variables.
\begin{align*}
	&\delta \left( \mu, \rho \right)   = \textrm{H}\left( \rho\right) -\frac{d}{2} \left( D_\textrm{KL} \left( \mu \vert\vert \rho \otimes \rho\right) + \beta   \sum_{\sigma \in A_1}  \mu (\sigma) + 2 \beta  \sum_{\sigma \in A_2}  \mu(\sigma)\right) \\
	&= \left(  1 - d\right) \textrm{H}\left( \rho\right) + \frac{d}{2} \textrm{H}\left( \mu \right) - d \beta  \left( x_3 + x_4 + x_5 + x_6 +  x_7 + x_8 + x_9 + \mu_{----} \right)
\end{align*}
Now, let us turn to the first  derivatives of  $\textrm{H}\left( \rho\right) $
\begin{align*}
	\frac{\partial  \textrm{H}\left( \rho\right)}{\partial x_1} &=    - \log \left( \rho_{+-} \right)  - 1  - \log \left( \rho_{-+} \right)  - 1 + 2 \log \left( \rho_{--} \right) + 2 \\
	\frac{\partial  \textrm{H}\left( \rho\right)}{\partial x_2} &= - \log \left( \rho_{++} \right)  - 1   +   \log \left( \rho_{--} \right) + 1 \\
	\frac{\partial  \textrm{H}\left( \rho\right)}{\partial x_3} &= - \log \left( \rho_{++} \right)  - 1   - \log \left( \rho_{+-} \right)  - 1   + 2 \log \left( \rho_{--} \right) + 2 \\
	\frac{\partial  \textrm{H}\left( \rho\right)}{\partial x_4} &= - \log \left( \rho_{++} \right)  - 1     - \log \left( \rho_{-+} \right)  - 1 + 2 \log \left( \rho_{--} \right) + 2 \\
	\frac{\partial  \textrm{H}\left( \rho\right)}{\partial x_5} &=  - \log \left( \rho_{+-} \right)  - 1  +  \log \left( \rho_{--} \right) + 1 \\
	\frac{\partial  \textrm{H}\left( \rho\right)}{\partial x_6} &=   - \log \left( \rho_{-+} \right)  - 1 +  \log \left( \rho_{--} \right) + 1 \\
	\frac{\partial  \textrm{H}\left( \rho\right)}{\partial x_7} &=    - \log \left( \rho_{+-} \right)  - 1   + \log \left( \rho_{--} \right) + 1 \\
	\frac{\partial  \textrm{H}\left( \rho\right)}{\partial x_8} &=   - \log \left( \rho_{-+} \right)  - 1 + \log \left( \rho_{--} \right) + 1 \\
	\frac{\partial  \textrm{H}\left( \rho\right)}{\partial x_9} &= - \log \left( \rho_{++} \right)  - 1    +  \log \left( \rho_{--} \right) + 1 \\
\end{align*}
and the first derivatives of $\textrm{H}\left( \mu \right) $
\begin{align*}
	\frac{\partial  \textrm{H}\left( \mu\right)}{\partial x_1} &=  - 2 \log \left( x_1 \right) - 2 + 2 \log \left( \mu_{----} \right)  + 2  \\
	\frac{\partial  \textrm{H}\left( \mu\right)}{\partial x_2} &=  - 2 \log \left( x_2 \right) - 2 + 2 \log \left( \mu_{----} \right)  + 2  \\
	\frac{\partial  \textrm{H}\left( \mu\right)}{\partial x_3} &= -  2 \log \left( x_3 \right) - 2  + 2 \log \left( \mu_{----} \right)  + 2 \\
	\frac{\partial  \textrm{H}\left( \mu\right)}{\partial x_4} &= -  2 \log \left(  x_4\right) - 2 + 2 \log \left( \mu_{----} \right)  + 2\\
\end{align*}	
\begin{align*}
	\frac{\partial  \textrm{H}\left( \mu\right)}{\partial x_5} &=   -  2 \log \left( x_5 \right) - 2 + 2 \log \left( \mu_{----} \right)  + 2\\
	\frac{\partial  \textrm{H}\left( \mu\right)}{\partial x_6} &=    -  2 \log \left( x_6 \right) - 2 + 2 \log \left( \mu_{----} \right)  + 2\\
	\frac{\partial  \textrm{H}\left( \mu\right)}{\partial x_7} &=    -   \log \left( x_7 \right) - 1 + \log \left( \mu_{----} \right)  + 1\\
	\frac{\partial  \textrm{H}\left( \mu\right)}{\partial x_8} &=    -   \log \left( x_8 \right) - 1 + \log \left( \mu_{----} \right)  + 1 \\
	\frac{\partial  \textrm{H}\left( \mu\right)}{\partial x_9} &=  -   \log \left( x_9 \right) - 1 + \log \left( \mu_{----} \right)  + 1.
\end{align*}
For the second derivatives we obtain
\begin{align*}
	\frac{\partial^2  \delta \left( \mu, \rho \right)}{\partial x_1^2} &=  \left( 1 - d \right) \left(   - \frac{1}{\rho_{+-}} - \frac{1}{\rho_{-+}} -  \frac{4}{\rho_{--}}\right) + \frac{d}{2} \left( - \frac{2}{x_1} - \frac{4}{\mu_{----}} \right)\\
	\frac{\partial^2  \delta \left( \mu, \rho \right)}{\partial x_2 \partial x_1} &=  \left( 1 - d \right)  \left( - \frac{2}{\rho_{--}}\right) + \frac{d}{2} \left( - \frac{4}{\mu_{----}}\right) \\
	\frac{\partial^2  \delta \left( \mu, \rho \right)}{\partial x_3 \partial x_1} &=  \left( 1 - d \right)  \left( - \frac{1}{\rho_{+-}} - \frac{4}{\rho_{--}}\right) + \frac{d}{2} \left( - \frac{4}{\mu_{----}}\right) \\
	\frac{\partial^2  \delta \left( \mu, \rho \right)}{\partial x_4 \partial x_1} &=  \left( 1 - d \right)  \left(  - \frac{1}{\rho_{-+}} - \frac{4}{\rho_{--}}\right) + \frac{d}{2} \left( - \frac{4}{\mu_{----}}\right) \\
	\frac{\partial^2  \delta \left( \mu, \rho \right)}{\partial x_5 \partial x_1} &=  \left( 1 - d \right)  \left( - \frac{1}{\rho_{+-}} - \frac{2}{\rho_{--}}\right) + \frac{d}{2} \left( - \frac{4}{\mu_{----}}\right) \\
	\frac{\partial^2  \delta \left( \mu, \rho \right)}{\partial x_6 \partial x_1} &=  \left( 1 - d \right)  \left( - \frac{1}{\rho_{-+}} - \frac{2}{\rho_{--}}\right) + \frac{d}{2} \left( - \frac{4}{\mu_{----}}\right) \\
	\frac{\partial^2  \delta \left( \mu, \rho \right)}{\partial x_7 \partial x_1} &=  \left( 1 - d \right)  \left( - \frac{1}{\rho_{+-}} - \frac{2}{\rho_{--}}\right) + \frac{d}{2} \left( - \frac{2}{\mu_{----}}\right) \\
	\frac{\partial^2  \delta \left( \mu, \rho \right)}{\partial x_8 \partial x_1} &=  \left( 1 - d \right)  \left( - \frac{1}{\rho_{-+}} - \frac{2}{\rho_{--}}\right) + \frac{d}{2} \left( - \frac{2}{\mu_{----}}\right) \\
	\frac{\partial^2  \delta \left( \mu, \rho \right)}{\partial x_9 \partial x_1} &=  \left( 1 - d \right)  \left(  - \frac{2}{\rho_{--}}\right) + \frac{d}{2} \left( - \frac{2}{\mu_{----}}\right) \\
\end{align*}
and
\begin{align*}
	\frac{\partial^2  \delta \left( \mu, \rho \right)}{\partial x_2^2} &=  \left( 1 - d \right) \left( - \frac{1}{\rho_{++}}- \frac{1}{\rho_{--}}\right) + \frac{d}{2} \left(  - \frac{2}{x_2} - \frac{4}{\mu_{----}}\right)\\
	\frac{\partial^2  \delta \left( \mu, \rho \right)}{\partial x_3 \partial x_2} &= \left( 1 - d \right) \left( - \frac{1}{\rho_{++}} - \frac{2}{\rho_{--}}\right) + \frac{d}{2} \left( - \frac{4}{\mu_{----}} \right)\\
	\frac{\partial^2  \delta \left( \mu, \rho \right)}{\partial x_4 \partial x_2} &=  \left( 1 - d \right) \left(  - \frac{1}{\rho_{++}}- \frac{2}{\rho_{--}}\right) + \frac{d}{2} \left( - \frac{4}{\mu_{----}} \right)\\
	\frac{\partial^2  \delta \left( \mu, \rho \right)}{\partial x_5 \partial x_2} &=  \left( 1 - d \right) \left( - \frac{1}{\rho_{--}}\right) + \frac{d}{2} \left(  - \frac{4}{\mu_{----}}\right)\\
\end{align*}
\begin{align*}
	\frac{\partial^2  \delta \left( \mu, \rho \right)}{\partial x_6 \partial x_2} &=  \left( 1 - d \right) \left( - \frac{1}{\rho_{--}}\right) + \frac{d}{2} \left( - \frac{4}{\mu_{----}} \right)\\
	\frac{\partial^2  \delta \left( \mu, \rho \right)}{\partial x_7 \partial x_2} &=  \left( 1 - d \right) \left( - \frac{1}{\rho_{--}}\right) + \frac{d}{2} \left( - \frac{2}{\mu_{----}} \right)\\
	\frac{\partial^2  \delta \left( \mu, \rho \right)}{\partial x_8 \partial x_2} &= \left( 1 - d \right) \left( - \frac{1}{\rho_{--}}\right) + \frac{d}{2} \left( - \frac{2}{\mu_{----}} \right)\\
	\frac{\partial^2  \delta \left( \mu, \rho \right)}{\partial x_9 \partial x_2} &=  \left( 1 - d \right) \left(  - \frac{1}{\rho_{++}}- \frac{1}{\rho_{--}}\right) + \frac{d}{2} \left( - \frac{2}{\mu_{----}} \right).\\
\end{align*}
We continue with
\begin{align*}
	\frac{\partial^2  \delta \left( \mu, \rho \right)}{\partial x_3^2} &=  \left( 1 - d \right) \left(-\frac{1}{\rho_{++}} - \frac{1}{\rho_{+-}} - \frac{4}{\rho_{--}} \right) + \frac{d}{2} \left( - \frac{2}{x_3} - \frac{4}{\mu_{----}}\right) \\
	\frac{\partial^2  \delta \left( \mu, \rho \right)}{\partial x_4 \partial x_3} &= \left( 1 - d \right) \left( -\frac{1}{\rho_{++}} - \frac{4}{\rho_{--}}\right) + \frac{d}{2} \left( -   \frac{4}{\mu_{----}} \right) \\
	\frac{\partial^2  \delta \left( \mu, \rho \right)}{\partial x_5 \partial x_3} &= \left( 1 - d \right) \left( -\frac{1}{\rho_{+-}} - \frac{2}{\rho_{--}}\right) + \frac{d}{2} \left( -   \frac{4}{\mu_{----}} \right) \\
	\frac{\partial^2  \delta \left( \mu, \rho \right)}{\partial x_6 \partial x_3} &= \left( 1 - d \right) \left(  - \frac{2}{\rho_{--}}\right) + \frac{d}{2} \left( -   \frac{4}{\mu_{----}} \right) \\
	\frac{\partial^2  \delta \left( \mu, \rho \right)}{\partial x_7 \partial x_3} &= \left( 1 - d \right) \left( -\frac{1}{\rho_{+-}}  - \frac{2}{\rho_{--}}\right) + \frac{d}{2} \left( -   \frac{2}{\mu_{----}} \right) \\
	\frac{\partial^2  \delta \left( \mu, \rho \right)}{\partial x_8 \partial x_3} &= \left( 1 - d \right) \left(  - \frac{2}{\rho_{--}}\right) + \frac{d}{2} \left( -   \frac{2}{\mu_{----}} \right) \\
	\frac{\partial^2  \delta \left( \mu, \rho \right)}{\partial x_9 \partial x_3} &= \left( 1 - d \right) \left( - \frac{1}{\rho_{++}} - \frac{2}{\rho_{--}}\right) + \frac{d}{2} \left( -   \frac{2}{\mu_{----}} \right)
\end{align*}
and
\begin{align*}
	\frac{\partial^2  \delta \left( \mu, \rho \right)}{\partial x_4^2} &= \left( 1 - d\right) \left( - \frac{1}{\rho_{++}} - \frac{1}{\rho_{-+}} - \frac{4}{\rho_{--}}\right) + \frac{d}{2} \left(- \frac{2}{x_4} - \frac{4}{\mu_{----}}\right) \\
	\frac{\partial^2  \delta \left( \mu, \rho \right)}{\partial x_5 \partial x_4} &=  \left( 1 - d\right) \left( - \frac{2}{\rho_{--}}\right)  + \frac{d}{2} \left( - \frac{4}{\mu_{----}}\right) \\
	\frac{\partial^2  \delta \left( \mu, \rho \right)}{\partial x_6 \partial x_4} &= \left( 1 - d\right) \left(- \frac{1}{\rho_{-+}} - \frac{2}{\rho_{--}} \right)  + \frac{d}{2} \left( - \frac{4}{\mu_{----}}\right) \\
	\frac{\partial^2  \delta \left( \mu, \rho \right)}{\partial x_7 \partial x_4} &= \left( 1 - d\right) \left(- \frac{2}{\rho_{--}}\right)  + \frac{d}{2} \left( - \frac{2}{\mu_{----}}\right)\\
	\frac{\partial^2  \delta \left( \mu, \rho \right)}{\partial x_8 \partial x_4} &= \left( 1 - d\right) \left(- \frac{1}{\rho_{-+}} - \frac{2}{\rho_{--}}\right)  + \frac{d}{2} \left( - \frac{2}{\mu_{----}}\right)\\
	\frac{\partial^2  \delta \left( \mu, \rho \right)}{\partial x_9 \partial x_4} &= \left( 1 - d\right) \left(- \frac{1}{\rho_{++}} - \frac{2}{\rho_{--}}\right)  + \frac{d}{2} \left( - \frac{2}{\mu_{----}}\right)
\end{align*}
and
\begin{align*}
	\frac{\partial^2  \delta \left( \mu, \rho \right)}{\partial x_5^2} &=  \left( 1 - d\right)  \left(  - \frac{1}{\rho_{+-}} - \frac{1}{\rho_{--}}\right)  + \frac{d}{2} \left( - \frac{2}{x_5} - \frac{4}{\mu_{----}} \right) \\
	\frac{\partial^2  \delta \left( \mu, \rho \right)}{\partial x_6 \partial x_5} &= \left( 1 - d\right)  \left( - \frac{1}{\rho_{--}} \right)  + \frac{d}{2} \left(  - \frac{4}{\mu_{----}}\right)\\
	\frac{\partial^2  \delta \left( \mu, \rho \right)}{\partial x_7 \partial x_5} &= \left( 1 - d\right)  \left( - \frac{1}{\rho_{+-}} - \frac{1}{\rho_{--}} \right)  + \frac{d}{2} \left( - \frac{2}{\mu_{----}}  \right)\\
	\frac{\partial^2  \delta \left( \mu, \rho \right)}{\partial x_8 \partial x_5} &=\left( 1 - d\right)  \left(- \frac{1}{\rho_{--}}  \right)  + \frac{d}{2} \left( - \frac{2}{\mu_{----}} \right) \\
	\frac{\partial^2  \delta \left( \mu, \rho \right)}{\partial x_9 \partial x_5} &= \left( 1 - d\right)  \left(- \frac{1}{\rho_{--}}  \right)  + \frac{d}{2} \left( - \frac{2}{\mu_{----}} \right)
\end{align*}
and
\begin{align*}
	\frac{\partial^2  \delta \left( \mu, \rho \right)}{\partial x_6^2} &=  \left( 1 - d\right)  \left( - \frac{1}{\rho_{-+}} - \frac{1}{\rho_{--}} \right)  + \frac{d}{2} \left(  - \frac{2}{x_6} - \frac{4}{\mu_{----}}\right) \\
	\frac{\partial^2  \delta \left( \mu, \rho \right)}{\partial x_7 \partial x_6} &= \left( 1 - d\right)  \left(  - \frac{1}{\rho_{--}}\right)  + \frac{d}{2} \left(  - \frac{2}{\mu_{----}}\right)\\
	\frac{\partial^2  \delta \left( \mu, \rho \right)}{\partial x_8 \partial x_6} &=\left( 1 - d\right)  \left( - \frac{1}{\rho_{-+}} - \frac{1}{\rho_{--}}\right)  + \frac{d}{2} \left( - \frac{2}{\mu_{----}} \right) \\
	\frac{\partial^2  \delta \left( \mu, \rho \right)}{\partial x_9 \partial x_6} &= \left( 1 - d\right)  \left( - \frac{1}{\rho_{--}} \right)  + \frac{d}{2} \left(  - \frac{2}{\mu_{----}}\right)
\end{align*}
and
\begin{align*}
	\frac{\partial^2  \delta \left( \mu, \rho \right)}{\partial x_7^2} &=  \left( 1 - d\right)  \left(- \frac{1}{\rho_{+-}} - \frac{1}{\rho_{--}} \right)  + \frac{d}{2} \left( -\frac{1}{x_7} - \frac{1}{\mu_{----}}\right) \\
	\frac{\partial^2  \delta \left( \mu, \rho \right)}{\partial x_8 \partial x_7} &=\left( 1 - d\right)  \left(- \frac{1}{\rho_{--}}\right)  + \frac{d}{2} \left( - \frac{1}{\mu_{----}}\right) \\
	\frac{\partial^2  \delta \left( \mu, \rho \right)}{\partial x_9 \partial x_7} &= \left( 1 - d\right)  \left(- \frac{1}{\rho_{--}}\right)  + \frac{d}{2} \left(- \frac{1}{\mu_{----}} \right)
\end{align*}
and
\begin{align*}
	\frac{\partial^2  \delta \left( \mu, \rho \right)}{\partial x_8^2} &=  \left( 1 - d\right)  \left(- \frac{1}{\rho_{-+}} - \frac{1}{\rho_{--}} \right)  + \frac{d}{2} \left( -\frac{1}{x_8} - \frac{1}{\mu_{----}}\right) \\
	\frac{\partial^2  \delta \left( \mu, \rho \right)}{\partial x_9 \partial x_8} &=\left( 1 - d\right)  \left(- \frac{1}{\rho_{--}}\right)  + \frac{d}{2} \left( - \frac{1}{\mu_{----}}\right) \\
	\frac{\partial^2  \delta \left( \mu, \rho \right)}{\partial x_9^2} &=  \left( 1 - d\right)  \left(- \frac{1}{\rho_{++}} - \frac{1}{\rho_{--}} \right)  + \frac{d}{2} \left( -\frac{1}{x_9} - \frac{1}{\mu_{----}}\right).
\end{align*}
Recall the definition of $\mu^*$
\begin{align*}
	\mu_{++++}^* &= \mu_{----}^* = \mu_{+-+-}^* = \mu_{-+-+}^* = \frac{e^{-2 \beta}}{4 \left( 1 + e^{- \beta}\right)^2}\\
	\mu_{+--+}^* &= \mu_{--++}^* = \mu_{-++-}^* = \mu_{++--}^* = \frac{1}{4 \left( 1 + e^{- \beta}\right)^2}\\
	\mu_{+++-}^* &= \mu_{++-+}^* = \mu_{+-++}^* = \mu_{-+++}^* = \mu_{---+}^*= \mu_{--+-}^*= \mu_{-+--}^*= \mu_{+---}^* = \frac{e^{- \beta}}{4 \left( 1 + e^{- \beta}\right)^2}
\end{align*}
which implies
\begin{align*}
	\rho_{++}^* = \rho_{+-}^* = \rho_{-+}^* = \rho_{--}^* = \frac{1}{4}
\end{align*}
and
\begin{align*}
	\delta \left( \mu^*, \rho^* \right)
	&= \left(  1 - d\right) \textrm{H}\left( \rho^*\right) + \frac{d}{2} \textrm{H}\left( \mu^* \right) - d \beta  \left( x_3^* + x_4^* + x_5^* + x_6^* +  x_7^* + x_8^* + x_9^* + \mu_{----}^* \right)\\
	&=  \left( 2 - d\right)  \log \left( 2\right) + d \log \left(1 + e^{- \beta}  \right).
\end{align*}
Evaluating the above derivatives at $\mu^*, \rho^*$ we obtain the Hessian at $\mu^*, \rho^*$.
\begin{align*}
	\mathrm{D}^2 \delta \left( \mu^*, \rho^* \right) &=  4 \left( d - 1 \right) \begin{pmatrix} 
		6 & 2 & 5 & 5 & 3 & 3 & 3 & 3 & 2  \\
		2 & 2 & 3 & 3 & 1 & 1 & 1 & 1 & 2  \\ 
		5 & 3 & 6 & 5 & 3 & 2& 3& 2 &3  \\
		5 & 3 & 5 & 6 & 2 &3&2&3 &3  \\
		3 & 1 & 3 & 2& 2& 1&2& 1 &1  \\
		3 & 1 & 2 & 3 & 1 & 2& 1& 2 &1  \\
		3 & 1 & 3 & 2 & 2 & 1& 2& 1 &1  \\
		3 & 1 & 2 & 3 & 1 & 2& 1& 2 &1  \\
		2 & 2 & 3 & 3 & 1 & 1& 1& 1 &2  \\
	\end{pmatrix} \\
	& \qquad -  2 d \frac{\left( 1 + e^{- \beta}\right)^2}{e^{-2 \beta}} \begin{pmatrix} 
		4 & 4 & 4 & 4 & 4 & 4 & 2 & 2 & 2  \\ 
		4 & 4 & 4 & 4 & 4 & 4 & 2 & 2 & 2  \\ 
		4 & 4 & 4 & 4 & 4 & 4& 2& 2 &2  \\
		4 & 4 & 4 & 4 & 4 &4&2&2 &2  \\
		4 & 4 & 4 & 4& 4& 4&2& 2 &2  \\
		4 & 4 & 4 & 4 & 4 & 4& 2& 2 &2  \\
		2 & 2 & 2 & 2 & 2 & 2& 1& 1 &1  \\
		2 & 2 & 2 & 2 & 2 & 2& 1& 1 &1  \\
		2 & 2 & 2 & 2 & 2 & 2& 1& 1 &1  \\
	\end{pmatrix} \\
	& \qquad -  2 d \left( 1 + e^{- \beta}\right)^2 \begin{pmatrix} 
		2 & 0 & 0& 0 & 0 & 0 & 0 & 0 & 0 \\ 
		0 & 2 & 0 & 0& 0 & 0 & 0 & 0 & 0  \\ 
		0 & 0 & 2 e^{\beta} & 0 & 0 & 0& 0& 0 &0  \\
		0 & 0 & 0 & 2 e^{\beta} & 0 &0&0&0 &0  \\
		0 & 0 & 0 & 0& 2 e^{\beta}& 0&0& 0 &0  \\
		0 & 0 & 0 & 0 & 0 & 2 e^{\beta}& 0& 0 &0  \\
		0 & 0 & 0 & 0 & 0 & 0& e^{2 \beta}& 0 &0  \\
		0 & 0& 0 & 0 & 0 & 0& 0& e^{2 \beta} &0  \\
		0 & 0 & 0 & 0 & 0 & 0& 0& 0 &e^{2 \beta}  \\
	\end{pmatrix} \\
\end{align*}
The lemma now follows from calculating the determinant of the preceding expression.

\end{document}